\documentclass[a4paper,reqno,12pt]{amsart}
\usepackage[margin=1in]{geometry}
\usepackage{kpfonts}
\usepackage{amscd}
\usepackage{amsthm}
\usepackage{amssymb}
\usepackage{amsmath}
\usepackage{mathtools}
\usepackage{enumitem}
\usepackage{epsfig}
\usepackage{tikz}
\usepackage{tikz-cd}
\usetikzlibrary{decorations.markings,shapes.geometric,matrix,arrows,positioning}
\usepackage{todonotes}
\usepackage[final]{microtype}
\usepackage[hidelinks]{hyperref}
\def\MR#1{} 

\newcommand{\tcoev}{\stackrel{\longleftarrow}{\operatorname{coev}}}
\newcommand{\tev}{\stackrel{\longleftarrow}{\operatorname{ev}}}
\newcommand{\ev}{\stackrel{\longrightarrow}{\operatorname{ev}}}
\newcommand{\coev}{\stackrel{\longrightarrow}{\operatorname{coev}}}
\newcommand{\p}[1]{\ensuremath{\bar {#1}}}
\newcommand{\ms}[1]{\mbox{\tiny$#1$}}

\tikzset{anchorbase/.style={baseline={([yshift=-0.5ex]current bounding box.center)}},
  int/.style={thick},
  cross line/.style={preaction={draw=white,line width=6pt,-}},
  wall/.style={thin,double,blue},
  middlearrow/.style={postaction=decorate,decoration={markings,mark=at
    position .55 with {\arrow{stealth};}}},
  middlearrowrev/.style={postaction=decorate,decoration={markings,mark=at
    position .55 with {\arrowreversed{stealth};}}},
  ev/.style={shape=rectangle, draw}
}

\newcommand{\C}{\mathbb{C}}
\newcommand{\R}{\mathbb{R}}

\newcommand{\Z}{\mathbb{Z}}
\newcommand{\Q}{\mathbb{Q}}
\newcommand{\cat}{\mathcal{C}}
\newcommand{\catint}{\mathcal{C}_R}
\newcommand{\mcat}{\widetilde{\cat}}

\newcommand{\FR}{\mathsf{Z}}
\newcommand{\Gr}{\mathsf{G}}
\newcommand{\GrH}{\mathsf{H}}
\newcommand{\SSS}{\mathsf{X}}
\newcommand{\SSSY}{\mathsf{Y}}
\newcommand{\Hom}{\textup{\text{Hom}}}
\newcommand{\End}{\textup{\text{End}}}

\renewcommand{\span}{\text{\textnormal{span}}}
\newcommand{\id}{\textup{\text{id}}}
\newcommand{\ima}{\text{\textnormal{im}}}
\newcommand{\coker}{\operatorname{coker}}
\newcommand{\qd}{\mathsf{d}}
\newcommand{\qdim}{\textup{\text{qdim}}\,}

\newcommand{\mt}{\mathsf{tr}}
\newcommand{\str}{\textup{\text{str}}}
\newcommand{\ptr}{\operatorname{ptr}}
\newcommand{\can}{\operatorname{can}}
\newcommand{\Ztwo}{\mathbb{Z} \slash 2 \mathbb{Z}}

\newcommand{\bos}{\mathfrak{t}}
\newcommand{\Bos}{T}
\newcommand{\bosSm}{\mathfrak{s}}
\newcommand{\GW}{\bos_R}
\newcommand{\GWH}{\bos_H}
\newcommand{\Uq}{U_q^{\bos}(\GW)}
\newcommand{\Uqm}{\widetilde{U}_q^{\bos}(\GW)}
\newcommand{\Uqmin}{\widetilde{U}_{\sqrt{-1}}^{\bos}(\GW)}

\newcommand{\UqminNn}{\widetilde{U}_{\sqrt{-1}}^{\bos,\leq 0}(\GW)}
\newcommand{\UqTor}{\widetilde{U}_{\sqrt{-1}}^{\bos}(\bos)}

\newcommand{\UqN}{U_q^{\bos,\geq 0}(\GW)}

\newcommand{\gloo}{\mathfrak{gl}(1 \vert 1)}
\newcommand{\psloo}{\mathfrak{psl}(1 \vert 1)}
\newcommand{\Uoo}{U(1 \vert 1)}
\newcommand{\uoo}{\mathfrak{u}(1 \vert 1)}
\newcommand{\met}{\kappa}
\newcommand{\eff}{\textup{\text{eff}}}
\newcommand{\emet}{\met_{\eff}}

\newcommand{\vect}{\mathsf{vect}}
\newcommand{\ZVect}{\mathsf{V}\mathsf{ect}^{\FR}}
\newcommand{\Cob}{\mathsf{C}\mathsf{ob}}
\newcommand{\CobAd}{\mathsf{C}\mathsf{ob}^{\textnormal{ad}}}

\newcommand{\CobAdExt}{\check{\mathsf{C}}\mathsf{ob}^{\mathsf{ad}}}
\newcommand{\ZCat}{\check{\mathsf{C}}\mathsf{at}^{\FR}}

\newcommand{\TQFT}{\mathcal{Z}}
\newcommand{\CS}{{\mathcal{S}}}

\newcommand{\coh}{\omega}
\renewcommand{\D}{{\mathscr{D}}}

\newcommand{\CGP}{{\rm CGP}}

\newcommand{\sgn}{\textup{\text{sgn}}}
\newcommand{\Int}{\operatorname{int}}

\def\pser#1{[\![#1]\!]} 

\newtheorem{Lem}{Lemma}[section]
\newtheorem{Prop}[Lem]{Proposition}

\newtheorem{Def}[Lem]{Definition}

\theoremstyle{plain}
\newtheorem{Thm}[Lem]{Theorem}
\newtheorem{ThmIntro}{Theorem}

\newtheorem{Assump}{Assumption}

\theoremstyle{definition}
\newtheorem{examplex}[Lem]{Example}
\newenvironment{Ex}
  {\pushQED{\qed}\examplex}
  {\popQED\endexamplex}
  
 \theoremstyle{definition}
\newtheorem{remarkx}[Lem]{Remark}
\newenvironment{Rem}
  {\pushQED{\qed}\remarkx}
  {\popQED\endremarkx}
  
\theoremstyle{definition}

\newcommand{\epsh}[2]
         {\begin{array}{c} \hspace{-1.3mm}
        \raisebox{-4pt}{\epsfig{figure=#1,height=#2}}
        \hspace{-1.9mm}\end{array}}
        
\newcommand{\comm}[1]{}

\author[N. Garner]{Niklas Garner}
\address{Department of Physics \\ University of Washington\\
Seattle, Washington 98195 \\ USA}
\email{nkgarner@uw.edu}

\author[N. Geer]{Nathan Geer}
\address{Department of Mathematics and Statistics \\ Utah State University\\
Logan, Utah 84322 \\ USA}
\email{nathan.geer@usu.edu}

\author[M.\,B. Young]{Matthew B. Young}
\address{Department of Mathematics and Statistics \\ Utah State University\\
Logan, Utah 84322 \\ USA}
\email{matthew.young@usu.edu}

\title[Gaiotto--Witten theory and TQFT]{B-twisted Gaiotto--Witten theory and topological quantum field theory}

\date{\today}

\begin{document}

\begin{abstract}
We develop representation theoretic techniques to construct three dimensional non-semisimple topological quantum field theories which model homologically truncated topological B-twists of abelian Gaiotto--Witten theory with linear matter. Our constructions are based on relative modular structures on the category of weight modules over an unrolled quantization of a Lie superalgebra. The Lie superalgebra, originally defined by Gaiotto and Witten, is associated to a complex symplectic representation of a metric abelian Lie algebra. The physical theories we model admit alternative realizations as Chern--Simons-Rozansky--Witten theories and supergroup Chern--Simons theories and include as particular examples global forms of $\gloo$-Chern--Simons theory and toral Chern--Simons theory. Fundamental to our approach is the systematic incorporation of non-genuine line operators which source flat connections for the topological flavour symmetry of the theory.
\end{abstract}

\maketitle

\setcounter{tocdepth}{1}

\tableofcontents

\section*{Introduction}
\addtocontents{toc}{\protect\setcounter{tocdepth}{1}}

We construct three dimensional non-semisimple topological quantum field theories (TQFTs) which model topological B-twists of Gaiotto--Witten theories with abelian gauge group and linear hypermultiplet matter. Our constructions are based on the representation theory of an unrolled quantization of a Lie superalgebra associated to a symplectic representation of a metric abelian Lie algebra.

\subsection*{Background and motivation}
The notion of a topological twist is a tool used by physicists to extract TQFTs from supersymmetric quantum field theories. Much like the more familiar two dimensional quantum field theories with $\mathcal{N}=(2,2)$ supersymmetry and the resulting topological A- and B-models \cite{witten1988a,witten1992}, three dimensional quantum field theories with $\mathcal{N}=4$ supersymmetry admit two topological twists. The topological A-twist is a dimensional reduction of the four dimensional twist giving rise to Donaldson--Witten theory \cite{witten1988b} and for gauge theories with Yang--Mills kinetic term leads to the Coulomb branch constructions of Braverman, Finkelberg and Nakajima \cite{nakajima2016,braverman2018}. The topological B-twist, as we describe in more detail below, gives rise to Rozansky--Witten theory \cite{rozansky1997}.

By a mechanism originally discovered by Gaiotto and Witten in the study of four dimensional $\mathcal{N}=4$ supersymmetric Yang--Mills theory with a spatially varying $\theta$-angle \cite{gaiotto2010b}, it is possible to construct Chern--Simons-matter theories with $\mathcal{N}=4$ supersymmetry, and hence admitting topological twists. The input data for the simplest Gaiotto--Witten theories includes a gauge group $G$ and a holomorphic Hamiltonian $G$-manifold $X$ satisfying some additional conditions. The A-twists of these theories were formulated in \cite{koh2009} but are still not well understood. Their B-twists were first described by Kapustin and Saulina \cite{kapustin2009b} and naturally interpolate between two particularly important three dimensional quantum field theories: they become Chern--Simons theories when $X$ is a point and Rozansky--Witten theories when $G$ is a point. Correspondingly, B-twists of these Gaiotto--Witten theories are known as \emph{Chern--Simons-Rozansky--Witten theories}.

Chern--Simons theory is a quantum gauge theory defined by its compact gauge Lie group $G$ and level $\met \in H^4(BG; \Z)$ \cite{witten1989}. At a formal level, partition functions in Chern--Simons theory are invariants of closed oriented $3$-manifolds. More generally, the inclusion of Wilson loops leads to invariants of $3$-manifolds containing a link coloured by representations of $G$. As such, Chern--Simons theory is a natural physical setting for many constructions in quantum topology, including the Jones, Kauffman and HOMFLYPT polynomials. A key feature of Chern--Simons theory is the finite semisimplicity of its category of (Wilson) line operators. This allows for a physical understanding of Chern--Simons theory via rational conformal field theory \cite{witten1989,elitzur1989}, namely, chiral Wess--Zumino--Witten theory, and underlies its mathematical construction, in the form of an Atiyah--Segal-style TQFT, as the Reshetikhin--Turaev theory of a modular tensor category associated to the pair $(G,\met)$ \cite{turaev1994}. See \cite{dijkgraaf1990,freed1993,freed1994}, \cite{reshetikhin1991,andersen1995} and \cite{joyal1993,belov2005,stirling2008} when $G$ finite, semisimple and toral, respectively, and \cite{freed2010,henriques2017} for general proposals. When $G$ is simple simply connected, the relevant modular tensor category is the semisimplified category of modules over the small quantum group of $\mathfrak{g}_{\C}$ at a $\met$-dependent root of unity.

Much less understood is Rozansky--Witten theory, a topological B-twist of a $\mathcal{N}=4$ supersymmetric $\sigma$-model with target a holomorphic symplectic manifold $(X, \omega)$ \cite{rozansky1997}. In view of its central role in three dimensional mirror symmetry \cite{intriligator1996} and related mathematics \cite{teleman2014,braverman2019}, it is an important open problem to construct Rozansky--Witten theory as a TQFT. A significant obstacle to doing so is that Rozansky--Witten theory is neither finite nor semisimple: its category of line operators is expected to be the $2$-periodic derived category of coherent sheaves on $X$ \cite{kapustin2009}. On the other hand, the lack of finiteness and semisimplicity leads to rich mathematical structures in the theory. Derived Lie theoretic aspects of Rozansky--Witten theory were studied in \cite{kapranov1999,kontsevich1999}, formalizing the intuition that Rozansky--Witten theory is a fermionic analogue of Chern--Simons theory, with $X$ and $\omega$ playing the roles of (the classifying stack of) $G$ and $\met$, respectively. Perturbative studies of Rozansky--Witten theory are found in \cite{qiu2009,qiu2010,roberts2010,chan2017}. The structure of topological boundary conditions and defects in Rozansky--Witten theory was explored in \cite{kapustin2009,kapustin2010,teleman2020,oblomkov2023} and led, via the Cobordism Hypothesis \cite{lurie2009}, to the construction of partially defined TQFTs when $X$ a cotangent bundle using derived algebraic geometry \cite{benzvi2010} and $X=T^{\vee} \C^n$ using categories of matrix factorizations \cite{brunner2023}.

The Chern--Simons-Rozansky--Witten theories of \cite{kapustin2009b} are an amalgam of these two theories. We assume that $(X,\omega)$ admits an action of a compact Lie group $G$ by holomorphic Hamiltonian automorphisms. As described in \cite{gaiotto2010b}, it is only possible couple Chern--Simons gauge fields to this flavour symmetry $G$ when it satisfies the \emph{Fundamental Identity}, discussed below. When the theory has a linear target space, that is, $X$ is a complex symplectic representation of $G$, the Fundamental Identity allows for the construction of complex Lie superalgebra $\mathfrak{g}_X$ from the complexified Lie algebra $\mathfrak{g}_{\C}$ of $G$ and $X$. The resulting Chern--Simons-Rozansky--Witten theory is equivalent to a supergroup Chern--Simons theory based on a global form of $\mathfrak{g}_X$. Perturbative aspects of Chern--Simons-Rozansky--Witten theory were studied in \cite{kallen2013}, where it was reformulated as an AKSZ theory \cite{alexandrov1997}. In view of the utility of the boundary Wess--Zumino--Witten models in ordinary Chern--Simons theories, the first author recently proposed boundary logarithmic vertex operator algebras whose modules model the categories of line operators in A- and B-twisted Gaiotto--Witten theories \cite{garner2023}, building on earlier work in the setting of three dimensional $\mathcal{N}=4$ gauge theories with Yang--Mills kinetic term \cite{costello2019,costello2019b}.

\subsection*{Main results}

We construct three dimensional non-semisimple TQFTs which model B-twisted Gaiotto--Witten theory with the following input data:
\begin{itemize}
\item Gauge group $G=T$, a non-trivial connected abelian Lie group with metric $\met$.
\item Hypermultiplet matter representation $X=T^{\vee} R \simeq R \oplus R^{\vee}$, a linear holomorphic symplectic manifold, where $R$ is a complex representation of $T$.
\end{itemize}
Since the gauge group is abelian and the target is linear, we call these theories linear abelian Gaiotto--Witten theories. The input data is required to satisfy the Fundamental Identity along with further technical conditions (Assumptions \ref{assump:quantAssumptDualInt}, \ref{assump:noZeroRoots} and \ref{assump:unimodular} in the body of the paper). As indicated above, the theory we construct can alternatively be seen as the Chern--Simons-Rozansky--Witten theory of the Hamiltonian $T$-manifold $T^{\vee} R$ at level $\met$ or the Chern--Simons theory whose gauge Lie supergroup has complexified Lie superalgebra $\GW$, defined below, and bosonic subgroup $T$.

Our approach to non-semisimple TQFT is via the theory of relative modular tensor categories \cite{costantino2014,derenzi2022}. As reviewed in Section \ref{sec:relModTQFT}, this theory yields a non-finite, non-semisimple generalization of the Reshetikhin--Turaev construction of a TQFT from a modular tensor category \cite{turaev1994}. Relative modular categories have been used successfully to model topological A-twists of certain three dimensional $\mathcal{N}=4$ Chern--Simons-matter theories \cite{blanchet2016,creutzig2021} and $\Uoo$-Chern--Simons theory \cite{geerYoung2022}; the Rozansky--Witten-theoretic analysis of \cite{gukov2021} is expected to be related to the construction of \cite{creutzig2021} by three dimensional mirror symmetry. The fundamental new idea in physical realizations of relative modular categories is the introduction of background fields coupling to flavour symmetries of the underlying physical theory \cite{gaiotto2019}. In the present context, this involves considering non-genuine line operators which source non-trivial background flat connections for topological flavour symmetry.

To construct relative modular categories relevant to B-twisted Gaiotto--Witten theories, following Kapustin and Saulina \cite{kapustin2009b} and Gaiotto and Witten \cite{gaiotto2010b}, in Section \ref{sec:abGW} we recall that associated to the input data is a metric Lie superalgebra
\[
\GW = \bos \oplus \Pi T^{\vee} R.
\]
Here $\bos$ is the complexified Lie algebra of $T$ and $\Pi T^{\vee} R$ is the representation $T^{\vee} R$ placed in odd degree. The Lie bracket of two even elements is zero, of one even with one odd element is determined by the representation $T^{\vee}R$ and of two odd elements is determined by the metric $\met$ and the holomorphic moment map. The Fundamental Identity is precisely the statement that the bracket satisfies the super Jacobi identity for three odd elements. Under Assumption \ref{assump:quantAssumptDualInt}, which is a mild integrality condition on the pair $(\met,R)$, we introduce in Definition \ref{def:unrolledGW} the primary algebraic object of this paper: an unrolled quantization $\Uq$ of $\GW$ at parameter $q \in \C^{\times} \setminus \{\pm 1\}$. The infinite dimensional Hopf superalgebra $\Uq$ is a semidirect product of a standard quantization of $\GW$ with the enveloping algebra of $\bos$. Let $\cat_R$ be the category of finite dimensional $\Uq$-modules on which $\bos$ acts semisimply and the actions of $Z \in \bos$ and its corresponding grouplike element $K(Z) \in \Uq$ satisfy $K(Z)=q^Z$. The category $\cat_R$ has infinitely many isomorphism classes of simple objects and, unless $R$ is the zero representation, is non-semisimple.
 
In Section \ref{sec:weightMod} we develop the representation theory of $\Uq$. After establishing various structural properties of Verma modules and classifying the simple objects of $\cat_R$ via highest $\bos$-weights (Theorem \ref{thm:simplesSplitAb}), we prove our first main result.

\begin{ThmIntro}[{Theorem \ref{thm:ribbonCat}}]
\label{thm:ribbonCatInt}
The category $\catint$ is $\C$-linear, locally finite, abelian and ribbon and has enough projectives and injectives.
\end{ThmIntro}

The proof of Theorem \ref{thm:ribbonCatInt} begins with a direct construction of a braiding on $\cat_R$; see Proposition \ref{prop:braiding}. The exponential factor of the braiding is determined by the metric $\met$. A key technical result, Proposition \ref{prop:genericSSZWt}, asserts that the category $\cat_R$ is naturally graded by characters of certain distinguished central elements of $\Uq$ in such a way that generic homogenous subcategories of $\cat_R$ are semisimple (although non-finite). This relies on Assumption \ref{assump:noZeroRoots}, which states that the trivial representation does not appear in $R$. (Assumption \ref{assump:noZeroRoots} can effectively be removed; see Section \ref{sec:pslooCS}.) With these preliminaries, Theorem \ref{thm:ribbonCatInt} is proved by verifying that the candidate ribbon structure on $\cat_R$, given by the partial trace of the braiding with respect to a carefully chosen pivotal structure, is balanced on generic homogeneous subcategories of $\cat_R$. By a general argument of \cite{geer2018}, this suffices to ensure balancing globally.

The ribbon category $\cat_R$ or better, its derived category, is a model for the category of line operators in the perturbative Gaiotto--Witten theory associated to $(T,\met,R)$. As an indication of its perturbative nature, note that $\cat_R$ depends on $T$ only through its Lie algebra. To model the non-perturbative theory, including the global form of the gauge group, we equip $\cat_R$ with the additional data of a relative modular structure. This structure includes a grading of $\cat_R$ by an abelian group $\Gr$ for which generic homogenous subcategories are semisimple, a modified trace on the ideal of projective objects of $\cat_R$ and a monoidal functor from an abelian group $\FR$ into the degree $0 \in \Gr$ subcategory $\cat_{R, 0}$ such that generic homogeneous subcategories of $\cat_R$ have only finitely many $\FR$-orbits of simple objects. By work of De Renzi \cite{derenzi2022}, a relative modular structure on $\cat_R$ defines a decorated $\FR$-graded TQFT, that is, a symmetric monoidal functor
\[
\TQFT_{\cat_R}: \CobAd_{\cat_R} \rightarrow \ZVect_{\C}.
\]
The codomain is the category of $\FR$-graded vector spaces with symmetric braiding determined by the relative modular data. The domain is the category of decorated closed surfaces and their admissible bordisms. The decorations include $\cat_R$-coloured ribbon graphs and flat $\Gr$-connections. Admissibility can be seen as a genericity condition. Unlike the Atiyah--Segal framework, the category $\CobAd_{\cat_R}$ is usually not rigid.

For concreteness, in the introduction we focus on a particular class of relative modular structures on $\cat_R$. Other classes are constructed Section \ref{sec:absRelMod}. Recall that in compact Chern--Simons theory, the level $\met$ receives quantum corrections, resulting in the shift of $\met$ by the dual Coxeter number \cite{witten1989}. There is a similar mechanism in Gaiotto--Witten theory which leads to an effective, or quantum-corrected, metric $\emet$ which differs from $\met$ by the quadratic Casimir of $R$.
 
\begin{ThmIntro}[{Theorem \ref{thm:relModCpt}}]
\label{thm:redModInt}
The choice of an even integral lattice $\Gamma \subset (\bos, \emet)$ of full rank which contains the weights of the representation $R$ determines a relative modular structure on $\cat_R$. The resulting TQFT $\TQFT_{\cat_R}$ has the property that each of its state spaces is finite dimensional.
\end{ThmIntro}

The relative modular structure of Theorem \ref{thm:redModInt} is constructed as follows. Tracking $\bos$-weights modulo the dual lattice $\Gamma^{\vee}$ grades $\cat_R$ by the torus $\Gr = \bos^{\vee} \slash \Gamma^{\vee}$ which is Langlands dual to $T_{\Gamma} = (\Gamma \otimes_{\Z} \C) \slash \Gamma$. In particular, the subcategory $\cat_{R,\p 0}$ labels $\Uq$-modules which deform representations of the complex Lie supergroup with Lie superalgebra $\GW$ and bosonic subgroup $T_{\Gamma}$. Assumption \ref{assump:unimodular}, which is a mild genericity condition on $R$, ensures that $\cat_R$ is unimoduar. Since $\cat_R$ is generically semisimple and ribbon, this guarantees the existence of a modified trace on $\cat_R$; see Lemma \ref{lem:catUnimod} and Proposition \ref{prop:mTrace}. We take for $\FR$ the group of one dimensional $\Uq$-modules whose weights lie in the image of the adjoint map $\met^{\flat}: \Gamma \rightarrow \Gamma^{\vee}$.

In physical terms, the TQFT $\TQFT_{\cat_R}$ models the homological truncation of the B-twisted Gaiotto--Witten theory with compact gauge group $T_{c,\Gamma} = (\Gamma \otimes_{\Z} \R) \slash \Gamma$ at level $\met$ and hypermultiplet matter in the representation $T^{\vee} R$. In Section \ref{sec:nonzeroHyp}, we outline various consistency checks of this model. From the physical perspective, the ribbon category $\cat_R$ can be thought of as the category of line operators in the perturbative theory or, equivalently, the theory with simply connected gauge group of type $\GW$. The grading group $\Gr$ agrees with the topological flavour symmetry group of the physical theory. Objects in the homogeneous category $\cat_{R,\p \lambda}$, $\p \lambda \in \Gr$, correspond to non-genuine line operators which couple to flat $\Gr$-connection with holonomy $\p \lambda$ around the support of the operator. Objects of the group $\FR$ label gauge vortex lines and are a key non-perturbative ingredient. Much like in $T_{c,\Gamma}$-Chern--Simons theory, one role of gauge vortices is to screen line operators, resulting in a truncation of the spectrum of line operators. Said differently, line operators whose labels differ by the action of $\FR$ are physically equivalent. From the perspective of vertex operator algebras of boundary local operators, these gauge vortex lines end on boundary monopole operators and are entirely analogous to the modules that extend the perturbative Heisenberg vertex operator algebra to the full, non-perturbative lattice vertex operator algebra which models $T_{c,\Gamma}$-Chern--Simons theory \cite{ballin2022,ballin2023,garner2023b}.

In Theorem \ref{thm:verlindeCpt} we prove a Verlinde formula for the TQFT $\TQFT_{\cat_R}$, relating the partition function of a trivial circle bundle over a genus $g$ surface $\Sigma_g$ to the graded dimensions of the spate space $\TQFT_{\cat_R}(\Sigma_g)$. As an application, we compute the Euler characteristic of the underlying super vector space of $\TQFT_{\cat_R}(\Sigma_g)$:
\[
\chi(\TQFT_{\cat_R}(\Sigma_g))
=
\vert D \vert^{g-1} \sum_{k \in D} \prod_{i=1}^{\dim_{\C} R}\Big( 2 \sin \big(\met^{\vee}(Q_i,k) \pi \big) \Big)^{2g-2}.
\]
Here $D$ is the discriminant group of the metric $\met$ and $Q_i \in \bos^{\vee} $ are the weights of $R$. As we explain in Remark \ref{rem:Bethe}, the above expression for $\chi(\TQFT_{\cat_R}(\Sigma_g))$ is in agreement with physical computations by way of supersymmetric localization.

We summarize the other relative modular structures and corresponding physical theories considered in the body of the paper.
\begin{itemize}
\item In Section \ref{sec:uooCS} we construct models of Chern--Simons theory whose gauge supergroup has compact connected bosonic subgroup and Lie superalgebra $\uoo$. This gives a different approach to the mathematical realizations of particular global forms of $\Uoo$ Chern--Simons theory from \cite{geerYoung2022}, following earlier physical work of Rozansky and Saluer \cite{rozansky1992,rozansky1993,rozansky1994} and Mikhaylov \cite{mikhaylov2015}. The results of this section also provide TQFT realizations of the vertex operator algebraic constructions of \cite{garner2023b}.

\item By considering the degenerate case in which $R$ is the zero representation, we give in Section \ref{sec:toralCS} a new construction of Chern--Simons theory with compact torus gauge group. A novel feature of our construction is that it incorporates Wilson line operators that are attached to topological surface defects or, equivalently, sourcing a non-trivial background flat connection for the topological flavour symmetry $\Gr$ of the theory. Standard approaches to toral Chern--Simons theory incorporate only ordinary Wilson line operators \cite{belov2005,stirling2008}.

\item In Section \ref{sec:abelianHypers} we compare our setting to that of abelian gauged $\mathcal{N}=4$ hypermultiplets, as studied in \cite{ballin2022,ballin2023}. While we argue that the ribbon category $\cat_R$ models a portion of the category of line operators, we do not find a relative modular structure that models the non-perturbative theory. More precisely, when $\FR$ is taken to be the group suggested by \cite{ballin2023} there are infinitely many $\FR$-orbits of simple objects in each homogeneous summand of $\cat_R$. This suggests that an extension of the relative modular framework is required to treat such theories.

\item In Section \ref{sec:pslooCS} we construct models of Chern--Simons theory with gauge Lie superalgebra $\psloo$ coupled to flat $\C^{\times}$-connections, equivalently, Rozansky--Witten theory of $T^{\vee} \C$ with Higgs branch flavour symmetry $\C^{\times}$, recovering earlier results of \cite{geerYoung2022}. From the perspective of constructing Gaiotto--Witten theories, the results of Section \ref{sec:pslooCS} allow us to remove Assumption \ref{assump:noZeroRoots}.
\end{itemize}

As already mentioned, we expect the TQFTs of this paper to model homological truncations of particular physical theories. Indeed, the theories we consider arise as topological twists of supersymmetric quantum field theories. On general grounds, the category of line operators in such a theory is expected to form a differential graded category. It is an interesting mathematical problem to construct from such a differential category a TQFT.

Traditional representation theoretic approaches to compact Chern--Simons theory encode the level $\met$ in a quantum group through the quantum parameter $q$. This approach applies most naturally when there is a basic level, such as the Killing form on a simple Lie algebra, to which all other levels are related by scaling. Since there is no such basic level in our setting, we instead encode the level in the relative modular structure on a $q$-independent category $\cat_R$. In other words, we aim to find a suitable specialization of $q$ which works for all levels $\met$. However, in some cases, the natural specialization is $q=-1$, at which $\Uq$ is not defined. This leads us to consider $\Uqm$, an unrolled quantization of a non-standard presentation of $\GW$, whose specialization $\Uqmin$ serves as a replacement for the desired $q=-1$ specialization of $\Uq$. By working with the category $\mcat_R$ of weight $\Uqmin$-modules we are able to treat all levels uniformly. The TQFT constructions of the paper are therefore in terms of $\mcat_R$. Nevertheless, to ease the comparison with the mathematics literature, we develop the general theory for $\Uq$ and mention the modifications required for $\Uqm$.

We end this introduction by mentioning some natural extensions of the present work which seem to be accessible with existing techniques.
\begin{itemize}
\item It is natural to extend our results to B-twisted non-abelian Gaiotto--Witten theories, again with linear matter. The Fundamental Identity highly constrains the relevant Lie superalgebras, being, up to subquotients and summands with abelian bosonic subalgebra, direct sums of $\mathfrak{gl}(m \vert n)$ and $\mathfrak{osp}(m \vert 2n)$. In some cases, weight modules over unrolled quantizations of such Lie superalgebras are known to admit relative modular structures \cite{blanchet2016,derenzi2020,anghel2021,ha2022,geerYoung2022}, but their connection to physics deserves further study.

\item The supergroup Chern--Simons theories arising from B-twisted $\mathcal{N}=4$ gauge theories can have non-compact bosonic subgroups. For example, the B-twists of SQED \cite{ballin2022} and abelian gauged $\mathcal{N}=4$ hypermultiplets \cite{ballin2023} are of this type. To treat non-compact abelian gauge groups, our results must be extended to the category of $\Uq$-modules in which only the subalgebra of $\bos$ dual to the compact subgroup is required to act semisimply.
\end{itemize}
 
\subsection*{Acknowledgements}
The authors thank T.\, Creutzig, P.\, Crooks, T.\, Dimofte and W.\, Niu for discussions. N.\,Garner is supported by funds from the Department of Physics and the College of Arts \& Sciences at the University of Washington, Seattle. N.\,Geer is partially supported by National Science Foundation grant DMS-2104497. M.\,B.\,Young is partially supported by National Science Foundation grant DMS-2302363 and a Simons Foundation Travel Support for Mathematicians grant (Award ID 853541).

\section{Relative modular categories and topological quantum field theory}
\label{sec:relModTQFT}

We recall background material on relative modular categories and the topological quantum field theories (TQFTs) they define, following \cite{costantino2014,derenzi2022}. For background material on monoidal categories we refer the reader to \cite{etingof2015}.

\subsection{Ribbon categories}
\label{sec:ribbCat}

Let $(\cat,\otimes, \mathbb{I})$ be a $\C$-linear monoidal category. We assume that the functor $\otimes$ is $\C$-bilinear and the $\C$-algebra map $\C \to \End_\cat(\mathbb{I})$ is an isomorphism. If $\cat$ is rigid with braiding $\{c_{V,W}: V \otimes W \rightarrow W \otimes V \mid V,W \in \cat\}$ and compatible twist $\{\theta_V: V \rightarrow V \mid V \in \cat\}$, then $\cat$ is called a \emph{$\C$-linear ribbon category}. The dual of an object $V$ is denoted by $V^{\vee}$. Our diagrammatic conventions for ribbon categories are that diagrams are read left to right, bottom to top and
\[
\id_V
=
\begin{tikzpicture}[anchorbase]
\draw[->,thick] (0,0) -- node[left] {\small$V$} (0,1);
\end{tikzpicture}
\qquad , \qquad
\id_{V^{\vee}}
=
\begin{tikzpicture}[anchorbase]
\draw[<-,thick] (0,0) -- node[left] {\small$V$} (0,1);
\end{tikzpicture}
\]
\[
\tev_V
=
\begin{tikzpicture}[anchorbase]
\draw[->,thick] (0,0)  arc (0:180:0.5 and 0.75);
\node at (-1.4,0)  {$V$};
\end{tikzpicture}
\qquad, \qquad
\tcoev_V
=
\begin{tikzpicture}[anchorbase]
\draw[<-,thick] (0,0)  arc (180:360:0.5 and 0.75);
\node at (-0.5,-0.1)  {$V$};
\end{tikzpicture}
\]
\[
\ev_V
=
\begin{tikzpicture}[anchorbase]
\draw[<-,thick] (0,0)  arc (0:180:0.5 and 0.75);
\node at (0.4,0)  {$V$};
\end{tikzpicture}
\qquad, \qquad
\coev_V
=
\begin{tikzpicture}[anchorbase]
\draw[->,thick] (0,0)  arc (180:360:0.5 and 0.75);
\node at (1.5,-0.1)  {$V$};
\end{tikzpicture}
\]
\[
c_{V,W}
=
\begin{tikzpicture}[anchorbase]
\draw[->,thick] (0.5,0) -- node[right,near start] {\small $W$} (0,1);
\draw[->,thick,cross line] (0,0) -- node[left,near start] {\small$V$} (0.5,1);
\end{tikzpicture}
\qquad, \qquad
\theta_V
=
\begin{tikzpicture}[anchorbase]
\draw[->,thick,rounded corners=8pt] (0.25,0.25) -- (0,0.5) -- (0,1);
\draw[thick,rounded corners=8pt,cross line] (0,0) -- (0,0.5) -- (0.25,0.75);
\draw[thick] (0.25,0.75) to [out=30,in=330] (0.25,0.25);
\node at (-0.2,0.2)  {$V$};
\end{tikzpicture}
\]
for objects $V, W \in \cat$. An object $V\in\cat$ is \emph{regular} if $\tev_V:V^{\vee} \otimes V \rightarrow \mathbb{I}$ is an epimorphism. 

\subsection{Modified traces}
\label{sec:mTrace}

Let $V$ and $W$ be objects of a $\C$-linear ribbon category $\cat$. The right partial trace is the map
\begin{eqnarray*}
\ptr_W : \End_{\cat}(V \otimes W) & \rightarrow & \End_{\cat}(V) \\
f & \mapsto & (\id_V \otimes \ev_W) \circ (f \otimes \id_{W^{\vee}}) \circ (\id_V \otimes \tcoev_W).
\end{eqnarray*}

\begin{Def}
A full subcategory $\mathcal{I} \subset \cat$ is an {\em ideal} if $U\otimes V \in \mathcal{I}$ whenever $U\in \mathcal{I}$ and $V \in \cat$ and if $U \in \mathcal{I}$ and $V \in \cat$ are such that there exist morphisms $f:V\to U$ and $g:U\to V$ satisfying $g \circ f=\id_{V}$, then $V\in \mathcal{I}$.
\end{Def}

\begin{Def}[{\cite[\S 3.2]{geer2011}}]\label{def:mtrace}
\begin{enumerate}
\item A \emph{modified trace} on an ideal $\mathcal{I} \subset \cat$ is a collection of $\C$-linear functions 
$\mt=\{\mt_V:\End_\cat(V) \rightarrow \C \mid V \in \mathcal{I} \}$ which satisfies:
\begin{enumerate}
 \item \emph{Cyclicity}: $\mt_V(f \circ g)=\mt_W(g \circ f)$ for all $V,W \in \mathcal{I}$ and $f \in \Hom_{\cat}(W,V)$ and $g \in \Hom_{\cat}(V,W)$.
\item \emph{Partial trace property}: 
$\mt_{V\otimes W}(f)=\mt_V(\ptr_W(f))$ for all $V \in \mathcal{I}$, $W \in \cat$ and $f\in\End_{\cat}(V\otimes W)$.
\end{enumerate}

\item The \emph{modified dimension} of $V\in \mathcal{I}$ is $\qd(V)=\mt_V(\id_V)$.
\end{enumerate}
\end{Def}

\subsection{Relative modular categories}\label{sec:relModDef}

Let $\cat$ be a $\C$-linear ribbon category.

\begin{Def}
\begin{enumerate}
\item A set $\mathcal{E}=\{ V_i \mid i \in J \}$ of objects of $\cat$ is \emph{dominating} if for any $V \in \cat$ there exist $\{i_1,\dots,i_m \} \subseteq J $ and morphisms $\iota_k \in \Hom_{\cat}(V_{i_k},V)$ and $s_k \in \Hom_{\cat}(V,V_{i_k})$ such that $\id_{V}=\sum_{k=1}^m \iota_k \circ s_k$.

\item A dominating set $\mathcal{E}$ is \emph{completely reduced} if $\dim_\C \Hom_{\cat}(V_i,V_j)=\delta_{ij}$ for all $i,j \in J$.
\end{enumerate}
\end{Def}

Let $\FR$ be an abelian group, written additively. We often identify $\FR$ with the corresponding discrete monoidal category with object set $\FR$.

\begin{Def}
\label{def:free}
A \emph{free realization of $\FR$ in $\cat$} is a monoidal functor $\sigma: \FR \rightarrow \cat$, $k \mapsto \sigma_k$, such that
\begin{enumerate}
\item $\sigma_0=\mathbb{I}$,
\item $\theta_{\sigma_k}=\id_{\sigma_k} \text{ for all } k \in \FR$, and
\item if $V \otimes \sigma_k \simeq V$ for a simple object $V \in \cat$, then $k=0$.  
\end{enumerate}
\end{Def}

We often identify $\sigma : \FR \rightarrow \cat$ with the collection of objects $\sigma_{\FR} := \{\sigma_k \mid k \in \FR\}$, omitting from the notation the monoidal coherence data $\{\sigma_k \otimes \sigma_l \xrightarrow[]{\sim} \sigma_{k+l} \mid k,l \in \FR\}$.

\begin{Def}
\label{def:Gstr}
Let $\Gr$ be an abelian group.  A \emph{$\Gr$-grading on $\cat$} is an equivalence of $\C$-linear categories $\cat \simeq \bigoplus_{g \in \Gr} \cat_g$, where $\{\cat_g \mid g \in \Gr\}$ are full subcategories of $\cat$ such that $\mathbb{I} \in \cat_0$, if $V\in\cat_g$,  then  $V^{\vee}\in\cat_{-g}$ and if $V\in\cat_g$ and $V^{\prime}\in\cat_{g'}$, then $V\otimes V^{\prime}\in\cat_{g+g'}$.
\end{Def}

\begin{Def}\label{def:smallsymm}
A subset $\SSS$ of an abelian group $\Gr$ is \emph{symmetric} if $\SSS=-\SSS$ and \emph{small} if $\bigcup_{i=1}^n (g_i+\SSS) \neq \Gr$ for all $g_1,\ldots ,g_n\in \Gr$ and $n \geq 1$.
\end{Def}

We will use the following basic result in Section \ref{sec:kerGrad}.

\begin{Lem}
\label{lem:sssPullback}
Let $f: \Gr \rightarrow \GrH$ be a surjective homomorphism of abelian groups and $\SSSY \subset \GrH$ a small symmetric subset. Then $f^{-1}(\SSSY) \subset \Gr$ is a small symmetric subset.
\end{Lem}

\begin{proof}
Let $\SSS = f^{-1}(\SSSY)$. If $x \in \SSS$, then $f(-x)=-f(x) \in \SSSY$, since $\SSSY$ is symmetric. This proves that $\SSS$ is symmetric. To prove that $\SSS$ is small, we proceed by contradiction. Suppose that $\Gr= \bigcup_{i=1}^n (\SSS + g_i)$ for some $g_1, \dots, g_n \in \Gr$. Applying $f$ gives
\[
\GrH
=
f(\Gr)
=
f(\bigcup_{i=1}^n (\SSS + g_i))
=
\bigcup_{i=1}^n (f(\SSS) + f(g_i))
\subseteq
\bigcup_{i=1}^n (\SSSY + f(g_i)) \subseteq \GrH,
\]
contradicting smallness of $\SSSY$.
\end{proof}

\begin{Def}
\label{def:preMod}
Let $\Gr$ and $\FR$ be abelian groups and $\SSS \subset \Gr$ a small symmetric subset. A \emph{pre-modular $\Gr$-category relative to $(\FR,\SSS)$} is a $\Gr$-graded $\C$-linear ribbon category $\cat$ with a non-zero modified trace $\mt$ on its ideal of projective objects and a free realization $\sigma: \FR \rightarrow \cat_0$ with the following properties:
\begin{enumerate}
\item \emph{$\FR$-finite generic semisimplicity}: \label{def:genSS} For every $g \in \Gr \setminus \SSS$, there exists a finite set of regular simple\footnote{Since the category $\cat$ is not assumed to be abelian, we take \emph{simple} to mean that endomorphisms consist only of scalars. In the body of the paper, the category $\cat$ will be abelian and the objects of $\Theta(g)$ will be simple in the sense that they have no non-trivial subobjects. Schur's Lemma ensures the compatibility of these two approaches.} objects $\Theta(g):=\{ V_i \mid i \in I_g  \}$ such that
\[
\Theta(g) \otimes \sigma_{\FR}:=\{ V_i \otimes \sigma_k \mid i \in I_g, \; k \in \FR \}
\]
is a completely reduced dominating set for $\cat_g$.

\item \label{def:compat}
There exists a bicharacter $\psi: \Gr \times \FR \rightarrow \C^{\times}$ such that
\begin{equation}
\label{eq:psi}
c_{\sigma_k,V}\circ c_{V,\sigma_k}= \psi(g,k) \cdot  \id_{V \otimes \sigma_k}
\end{equation}
for all $g\in \Gr$, $V \in \cat_g$ and $k \in \FR$.
\end{enumerate}
 \end{Def}
 
\begin{Rem}
Property \eqref{def:genSS} of Definition \ref{def:preMod} implies that the $\Gr$-graded category $\cat$ is generically semisimple with singular locus $\SSS$; see \cite[Definition 3.3]{geer2018}.
\end{Rem}

\begin{Def}
\label{def:ndeg}
Let $\cat$ be a pre-modular $\Gr$-category relative to $(\FR,\SSS)$.
\begin{enumerate}
\item For each $g \in \Gr \setminus \SSS$, the \emph{Kirby colour of index $g$} is $\Omega_g:= \sum_{i \in I_g}\qd(V_i) \cdot V_i$.
\item The \emph{stabilization coefficients} $\Delta_\pm \in \C$ are defined by
\[
\epsh{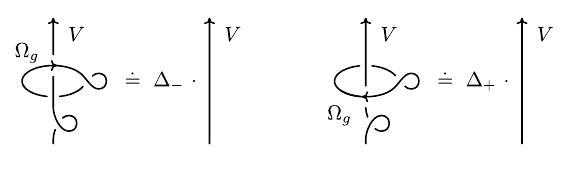}{15ex}
\]
where $g \in \Gr \setminus \SSS$ and $V\in \cat_g$. Here $\dot{=}$ denotes equality after application of the Reshetikhin--Turaev functor $F_{\cat}$.
\item The pre-modular $\Gr$-category $\cat$ is \emph{non-degenerate} if $\Delta_{+}\Delta_{-}\neq 0$. 
\end{enumerate}
\end{Def}

As suggested by the notation, $\Delta_+$ and $\Delta_-$ are independent of the choice of $g \in \Gr \setminus \SSS$ and $V \in \cat_g$ used in their definition \cite[Lemma 5.10]{costantino2014}.

\begin{Def}
\label{def:modG}
A \emph{modular $\Gr$-category relative to $(\FR,\SSS)$} is a pre-modular $\Gr$-category $\mathcal{C}$ relative to $(\FR,\SSS)$ for which there exists a scalar $\zeta \in \C^{\times}$, called the \emph{relative modularity parameter}, such that
\begin{equation}\label{eq:mod}
    \epsh{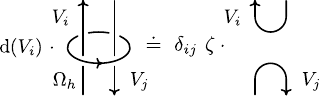}{10ex}
\end{equation}
for any $g,h \in \Gr \setminus \SSS$ and $i,j \in I_g$.
\end{Def}

The relative modularity parameter satisfies $\zeta = \Delta_+ \Delta_-$ \cite[Proposition 1.2]{derenzi2022}, whence relative modular categories are non-degenerate relative pre-modular.

\subsection{Topological quantum field theories from relative modular categories}
\label{sec:CGPTQFT}

In this paper all manifolds are assumed to be oriented.

Let $\cat$ be a modular $\Gr$-category relative to $(\FR,\SSS)$.

Let $\Cob_{\cat}$ be the category of decorated surfaces and their diffeomorphism classes of decorated bordisms, as defined in \cite[\S 2]{derenzi2022}. Briefly, objects of $\Cob_{\cat}$ are tuples $\CS=(\Sigma, \{p_i\}, \coh, \mathcal{L})$ consisting of
\begin{itemize}
\item  a closed surface $\Sigma$ with a set $*$ of distinguished base points, exactly one for each connected component,
\item a finite set $\{p_i\} \subset \Sigma \setminus *$ of oriented framed $\cat$-coloured points,
\item a relative cohomology class $\coh \in H^1(\Sigma\setminus\{p_i\}, * ;\Gr)$ such that $\omega(m_i) = g_i$ is the colour of $p_i$, where $m_i$ is the oriented boundary of a regular neighbourhood of $p_i$, and
\item a Lagrangian subspace ${\mathcal L}\subset H_1(\Sigma; \R)$.
\end{itemize}
Morphisms in $\Cob_{\cat}$ are tuples $\mathcal M = (M,T,\coh,m) : \CS_1 \rightarrow \CS_2$ consisting of
\begin{itemize}
\item a bordism $M: \Sigma_1 \rightarrow \Sigma_2$,
\item a $\cat$-coloured ribbon graph $T \subset M$ whose colouring is compatible with those of the marked points of $\CS_j$, $j=1,2$,
\item a class $\coh \in H^1(M\setminus T, *_1 \cup *_2; \Gr)$ which restricts to $\omega_j$ on $\Sigma_j$, $j=1,2$, and such the colour of each connected component $T_c$ of $T$ is in degree $\omega(m_c) \in \Gr$, where $m_c$ is an oriented meridian of $T_c$, and
\item an integer $m \in \Z$ called the \emph{signature defect}.
\end{itemize}
The morphism $\mathcal{M} :\CS_1 \rightarrow \CS_2$ is called \emph{admissible} if for each connected component $M_c$ of $M$ which is disjoint from the incoming boundary $\Sigma_1$, at least one edge of $T \cap M_c$ is coloured by a projective object of $\cat$ or there exists an embedded closed oriented curve $\gamma \subset M_c$ such that $\coh(\gamma) \in \Gr \setminus \SSS$. Let $\CobAd_{\cat}$ be the subcategory of admissible morphisms. Disjoint union gives $\Cob_{\cat}$ and $\CobAd_{\cat}$ the structure of a symmetric monoidal categories.

Let $\ZVect_{\C}$ be the monoidal category of $\FR$-graded vector spaces and their degree preserving linear maps. We consider $\ZVect_{\C}$ with the symmetric braiding determined by the unique pairing $\gamma: \FR \times \FR \rightarrow \{\pm 1\}$ which makes the diagram
\[
\begin{tikzpicture}[baseline= (a).base]
\node[scale=1.0] (a) at (0,0){
\begin{tikzcd}[column sep=7.0em,row sep=2.0em]
\sigma_{k_1} \otimes \sigma_{k_2} \arrow{r}[above]{c_{\sigma_{k_1},\sigma_{k_2}}} \arrow{d}[left]{\wr} & \sigma_{k_2} \otimes \sigma_{k_1} \arrow{d}[right]{\wr} \\
\sigma_{k_1+k_2} \arrow{r}[below]{\gamma(k_1,k_2) \cdot \id_{\sigma_{k_1+k_2}}} & \sigma_{k_2+k_1}
\end{tikzcd}
};
\end{tikzpicture}
\]
commute for all $k_1, k_2 \in \FR$.

\begin{Thm}[{\cite[Theorem 6.2]{derenzi2022}}]
\label{thm:relModTQFT}
A modular $\Gr$-category $\cat$ relative to $(\FR,\SSS)$ with a choice $\D \in \C^{\times}$ of square root of the relative modularity parameter $\zeta$ defines a symmetric monoidal functor $\TQFT_{\cat}: \CobAd_{\cat} \rightarrow \ZVect_{\C}$.
\end{Thm}

We refer to $\TQFT_{\cat}$ as the (decorated) TQFT assocaited to $\cat$. The values of $\TQFT_{\cat}$ on closed bordisms $\varnothing \rightarrow \varnothing$ coincide with a (renormalization of the) $3$-manifold invariants $\CGP_{\cat}$ of \cite{costantino2014}. Concretely, let $R$ be a $\cat$-coloured ribbon graph in $S^3$. Suppose that an edge of $R$ is coloured by a generic simple object $V \in \cat$, that is, $V$ is simple and $V \in \cat_g$ for some $g \in \Gr \setminus \SSS$. Let $R_V$ be the $(1,1)$-ribbon graph obtained from $R$ by cutting an edge labeled by $V$. Then $F^{\prime}_{\cat}(R) :=\mt_V(R_V) \in \C$ is an isotopy invariant of $R$ \cite{geer2009,geer2011}. With this notation, the partition function of a closed bordism $\mathcal{M}=(M,T,\coh,m)$ with computable surgery presentation $L \subset S^3$ is
\begin{equation}
\label{eq:cgpInvt}
\TQFT_{\cat}(\mathcal{M})
=
\D^{-1-l} \big( \frac{\D}{\Delta_{-}} \big)^{m-\sigma(L)} F'_{\cat}(L\cup T) \in \C.
\end{equation}
Here $l$ is the number of connected components of $L$, $\sigma(L)$ is the signature of the linking matrix of $L$ and each component $L_c$ of $L$ is coloured by the Kirby colour $\Omega_{\coh(m_c)}$.

The TQFT $\TQFT_{\cat}$ is the truncation of a once-extended TQFT $\check{\TQFT}_{\cat}: \CobAdExt_{\cat} \rightarrow \ZCat_{\C}$  \cite[Theorem 6.1]{derenzi2022}. Here $\CobAdExt_{\cat}$ is the bicategory of decorated $1$-manifolds, their decorated admissible $2$ dimensional bordisms and their equivalence classes of decorated admissible $3$ dimensional bordisms with corners and $\ZCat_{\C}$ is the bicategory of $\FR$-graded complete $\C$-linear categories with symmetric monoidal structure determined by $\gamma$. The theory $\check{\TQFT}_{\cat}$ assigns to the circle with cohomology class of holonomy $g \in \Gr$ a category Morita equivalent to the ideal of projective objects of $\cat_g$ \cite[Proposition 7.1]{derenzi2022}.

\begin{Def}[{\cite[\S 1.6]{geerYoung2022}}]
\label{def:relModFinite}
A modular $\Gr$-category $\cat$ relative to $(\FR,\SSS)$ is \emph{TQFT finite} if it has the following properties:
\begin{enumerate}[label=(F\arabic*)]
\item \label{ite:finCat1} For each $g\in \Gr$, there exists a finite set $\{P_j \mid j\in J_g\}$ of projective indecomposables of $\cat_g$ such that any projective indecomposable of $\cat_g$ is isomorphic to $P_j \otimes \sigma_k$ for some $j\in J_g$ and $k \in \FR$.
\item \label{ite:finCat2} For each projective $P \in \cat$, the vector space $\bigoplus_{k \in \FR} \Hom_{\cat}(\sigma_k,P)$ is finite dimensional.
\item \label{ite:finCat3} The full subcategory of projectives of $\cat$ is dominated by the set of projective indecomposables.
\end{enumerate}
\end{Def}

It is proved in \cite[Theorem 1.16]{geerYoung2022} that TQFT finiteness of $\cat$ ensures the finite dimensionality of all state spaces $\TQFT_{\cat}(\CS)$, $\CS \in \CobAd_{\cat}$.

\section{Quantum superalgebras for B-twisted abelian Gaiotto--Witten theory}
\label{sec:abGW}

\subsection{Superalgebra conventions}
\label{sec:conventions}

For a detailed discussion of superalgebra we refer the reader to \cite[I-Supersymmetry]{deligne1999}.

Let $\Ztwo = \{\p 0, \p 1\}$ be the additive group of order two. A \emph{super vector space} is a $\Ztwo$-graded complex vector space $V = V_{\p 0} \oplus V_{\p 1}$. The degree of a homogeneous element $v \in V$ is denoted $\p v \in \Ztwo$. Write $\Pi V$ for the super vector space with reversed parity, so that $(\Pi V)_{\p p} = V_{\p p + \p 1}$.  A \emph{morphism of super vector spaces of degree $\p d \in \Ztwo$} is a $\C$-linear map $f: V \rightarrow W$ which satisfies $\overline{f(v)} = \p v + \p d$ for each homogeneous $v \in V$. A (left) \emph{module} over a unital superalgebra $A$ is a super vector space $M$ together with a unital superalgebra homomorphism $A \rightarrow \End_{\C}(M)$ of degree $\p 0$. In particular, even (resp. odd) elements of $A$ act on $M$ by even (resp. odd) linear transformations. Write $[-,-]$ for the graded commutator in $A$, so that $[a,b] = ab - (-1)^{\p a \p b}ba$ for homogeneous elements $a, b \in A$.

\subsection{Gaiotto--Witten Lie superalgebras}
\label{sec:GWLieSuper}

Following Kapustin and Saulina \cite[\S 3.1]{kapustin2009b} and Gaiotto and Witten \cite[\S 3.2]{gaiotto2010b}, we introduce Lie superalgebras relevant to B-twisted Gaiotto--Witten theory.

Let $\bos$ be a complex Lie algebra with adjoint-invariant metric $\met: \bos \times \bos \rightarrow \C$. Let $\bos^{\vee} = \Hom_{\C}(\bos,\C)$ be the linear dual. There is an induced isomorphism
\[
\met^{\flat}: \bos \rightarrow \bos^{\vee},
\qquad
Z \mapsto \met(Z,-)
\]
with inverse $\met^{\sharp}: \bos^{\vee} \rightarrow \bos$. Let $\met^{\vee}$ be the metric on the linear dual $\bos^{\vee}$ induced by $\met^{\flat}$.

Let $(H,\omega)$ be a complex symplectic representation of $\bos$ and $\mu: H \rightarrow \bos^{\vee}$ a $\bos$-equivariant holomorphic moment map. We assume that the above data satisfies the \emph{Fundamental Identity} of Gaiotto and Witten \cite[\S 3.2]{gaiotto2010b}:
\begin{equation}
\label{eq:FundIdenGen}
\met^{\vee}(\mu,\mu)=0.
\end{equation}
For a general representation theoretic discussion of the Fundamental Identity and its variations, see \cite{medeiros2009}.

Let $\GWH$ be the super vector space $\bos \oplus \Pi H$. Define a Lie superalgebra structure on $\bos_H$ so that $\bos$ is a Lie subalgebra, the Lie bracket of $\bos$ and $\Pi H$ is given by the representation $\bos \rightarrow \mathfrak{sp}(H)$ and
\[
[h_1,h_2] = \met^{\sharp}(h_1(h_2(\mu))),
\qquad
h_1, h_2 \in H.
\]
On the right hand side we view $h_1$ and $h_2$ as constant holomorphic vector fields on $H$. Since $\mu$ is quadratic, $h_1(h_2(\mu))$ is an element $\bos^{\vee}$, whence $\met^{\sharp}(h_1(h_2(\mu)))$ is an element of $\bos$, as required. The only non-obvious task in verifying that the above definition is indeed a Lie superalgebra is the super Jacobi identity for three odd elements, which is seen to follow from equation \eqref{eq:FundIdenGen}. The metric $\met$ and symplectic form $\omega$ combine to define a non-degenerate supersymmetric adjoint-invariant bilinear form $(-,-)$ on $\GWH$.

\subsection{Lie superalgebras for abelian Gaiotto--Witten theory}
\label{sec:abGWClassical}

For the remainder of the paper, we work with the specialization of Section \ref{sec:GWLieSuper} to the case in which the Lie algebra $\bos$ is abelian. Let $r$ be the dimension of $\bos$ and $2n$ the dimension of $H$.

Since $\bos$ is abelian, the symplectic representation $H$ is polarizable, that is, there exists a complex representation $R$ of $\bos$ of dimension $n$ and an isomorphism of symplectic representations $(H,\omega) \simeq (T^{\vee} R=R \oplus R^{\vee},\left(\begin{smallmatrix} 0 & \id_{R^{\vee}} \\ -\can_R & 0 \end{smallmatrix} \right))$, where $\can_R: R \xrightarrow[]{\sim} R^{\vee \vee}$ is the canonical evaluation isomorphism. We fix such a choice of $R$ in what follows and refer to its weights $Q_1, \dots, Q_n \in \bos^{\vee}$ as \emph{roots} and the additive abelian group $\Lambda_R \subset \bos^{\vee}$ they generate as the \emph{root lattice}. Since $R$ is determined up to isomorphism by its roots, we often use the latter to describe the former. Choose linearly independent root vectors $E_1, \dots, E_n \in R$, that is
\[
Z \cdot E_i = Q_i(Z) E_i,
\qquad
Z \in \bos.
\]
Complete $\{E_1, \dots, E_n\}$ to a Darboux basis $\{E_1, \dots, E_n, F_1, \dots, F_n\}$ of $T^{\vee} R$, so that $\omega_{ij} := \omega(E_i,F_j) = \delta_{ij}$. It follows that
\[
Z \cdot F_i = - Q_i(Z) F_i,
\qquad
Z \in \bos.
\]
In terms of the roots, the Fundamental Identity \eqref{eq:FundIdenGen} reads
\begin{equation}
\label{eq:fundIdenAbs}
\met^{\vee}(Q_i, Q_j) =0,
\qquad
1 \leq i, j \leq n.
\end{equation}
In the Darboux basis, the non-trivial Lie bracket of odd elements of $\bos_H$ becomes
\[
[E_i, F_j] = \delta_{ij} \met^{\vee}(Q_i,-),
\]
where we view $\met^{\vee}(Q_i,-)$ as an element of $\bos$ via $\can_{\bos}: \bos \xrightarrow[]{\sim} \bos^{\vee \vee}$.

To describe $\GWH$ more explicitly, fix a basis $\{Z_a \mid 1 \leq a \leq r\}$ of $\bos$ and set $\met_{ab} = \met(Z_a, Z_b)$. In the dual basis $\{Z^{\vee}_a \mid 1 \leq a \leq r\}$ of $\bos^{\vee}$, the matrix of $\met^{\vee}$ is the inverse $(\met^{ab})$ of the matrix $(\met_{ab})$. Write $(Q_{1i}, \dots, Q_{ri}) \in \C^r$ for the image of $Q_i$ under the isomorphism $\bos^{\vee} \xrightarrow[]{\sim} \C^r$ and organize the roots of $R$ into an $r \times n$ matrix $Q=(Q_{ai})$. In the physics literature, the transpose of $Q$ is referred to as the \emph{charge matrix} of the theory. With this notation, the Fundamental Identity \eqref{eq:fundIdenAbs} reads
\begin{equation}
\label{eq:fundIden}
\sum_{a,b=1}^r \met^{ab}Q_{a i} Q_{b j} = 0,
\qquad
1 \leq i, j \leq n.
\end{equation}
We can now give a presentation of $\GWH$.

\begin{Prop}
\label{prop:aQGenRel}
The Lie superalgebra $\GWH$ is isomorphic to the Lie superalgebra $\GW$ with even generators $\{Z_a \mid 1 \leq a \leq r\}$, odd generators $\{E_i,F_i \mid 1 \leq i \leq n\}$ and defining relations
\begin{equation}
\label{eq:cartanCommuteClass}
[Z_a,Z_b]=0,
\end{equation}
\begin{equation}
\label{eq:rootsClass}
[Z_a,E_i]= Q_{ai} E_i,
\qquad
[Z_a,F_i]= -Q_{ai} F_i,
\end{equation}
\begin{equation}
\label{eq:nilpotentClass}
[E_i,E_j]= [F_i,F_j]=0,
\end{equation}
\begin{equation}
\label{eq:oddOddClass}
[E_i, F_j] = \delta_{ij} \sum_{a, b=1}^r \met^{ab} Q_{a i} Z_b.
\end{equation}
\end{Prop}

\begin{proof}
Only relations \eqref{eq:nilpotentClass} and \eqref{eq:oddOddClass} are non-obvious. These follow from the observation that, in the chosen Darboux basis, a moment map $\mu: H \rightarrow \bos^{\vee}$ is given by
\[
\mu = \sum_{k=1}^n \sum_{a=1}^r Q_{ak} E^{\vee}_k F^{\vee}_k Z^{\vee}_a,
\]
where $E^{\vee}_k$ and $F^{\vee}_k$ are seen as coordinates on $H$ and $Z^{\vee}_a$ as an element of $\bos^{\vee}$.
\end{proof}

As a consequence of Proposition \ref{prop:aQGenRel}, if $R_1$ and $R_2$ are two polarizations of $H$, then $\bos_{R_1} \simeq \bos_{R_2}$ as Lie superalgebras.

We begin with two degenerate examples.

\begin{Ex}
\label{ex:abelianSpecialization}
If $n =0$, so that $R$ is the zero representation, then $\GW=\bos$ is simply a metric abelian Lie algebra.
\end{Ex}

\begin{Ex}
\label{ex:pslSpecialization}
Let $\psloo$ be the projective special linear Lie superalgebra of $\C^{1 \vert 1}$, equivalently, the purely odd (abelian) Lie superalgebra of dimension two. If $r=0$, so that $\bos$ is the zero Lie algebra, then $\GW$ is isomorphic to $\psloo^{\oplus n}$ with its standard symplectic form.
\end{Ex}

\begin{Ex}
\label{ex:gloo}
Let $\gloo$ be the super vector space $\End_{\C}(\C^{1 \vert 1})$ with supercommutator as the Lie bracket. Let $r=2$ and $n=1$ with
\[
\met= \left(\begin{matrix} 0 & 1 \\ 1 & 0 \end{matrix}\right),
\qquad
Q= \left(\begin{matrix} 1 \\ 0 \end{matrix} \right).
\]
The assignments
\[
N = \left(\begin{matrix} 1 & 0 \\ 0 & 0 \end{matrix} \right) \mapsto Z_1,
\qquad
E = \left(\begin{matrix} 1 & 0 \\ 0 & 1 \end{matrix} \right) \mapsto Z_2,
\qquad
\psi^+ = \left(\begin{matrix} 0 & 1 \\ 0 & 0 \end{matrix} \right) \mapsto E_1,
\qquad
\psi^- = \left(\begin{matrix} 0 & 0 \\ 1 & 0 \end{matrix} \right) \mapsto F_1 
\]
extend to a Lie superalgebra isomorphism $\gloo \xrightarrow[]{\sim} \GW$. The induced bilinear form $(-,-)$ on $\gloo$ is that associated to $\str_{\C^{1 \vert 1}}$, the supertrace in the fundamental representation, and is determined by
\[
(N,E)=(E,N) = (\psi^+,\psi^-)=-(\psi^-,\psi^+) =1,
\]
with all other pairings of generators vanishing.
\end{Ex}

\begin{Ex}
\label{ex:gaugedAbHypers}
Let $\bosSm$ be a complex abelian Lie algebra of dimension $s$ and $R$ a representation of $\bosSm$ of dimension $n$ with weights $Q^{(\bosSm)}_1, \dots, Q^{(\bosSm)}_n \in \bosSm^{\vee}$. Let $\bos$ be the abelian Lie algebra $\bosSm \oplus \bosSm^{\vee}$ with metric $\met = \left(\begin{smallmatrix} 0 & \id_{\bosSm^{\vee}} \\ \can_{\bosSm} & 0 \end{smallmatrix} \right)$. View $T^{\vee} R$ as a symplectic representation of $\bos$ on which $\bosSm^{\vee}$ acts trivially. The root matrix of $R$ is $Q=\left( \begin{smallmatrix} Q^{(\bosSm)} \\ 0 \end{smallmatrix} \right)$. Since $\met$ is hyperbolic and $Q$ is concentrated in the top block, equation \eqref{eq:fundIdenAbs} imposes no constraints on $R$. With the additional assumption that $Q^{(\bosSm)}$ has integer entries, the Lie superalgebra $\GW$ appears in the context of three dimensional $\mathcal{N}=4$ abelian gauged hypermultiplets \cite[\S 4.2]{ballin2023}. When $s=n=1$ and $Q^{(\bosSm)}=(1)$, we again recover $\GW \simeq \gloo$ with its standard metric.
\end{Ex}

\begin{Lem}
\label{lem:centralElements}
For each $1 \leq i \leq n$, the element $\mathcal{Z}_i:=\sum_{a,b=1}^r \met^{ab} Q_{ai} Z_b \in \GW$ is central.
\end{Lem}

\begin{proof}
We have
\[
[\mathcal{Z}_i,E_j] = \sum_{a,b=1}^r \met^{ab} Q_{ai} Q_{bj} E_j,
\]
which vanishes due to equation \eqref{eq:fundIden}. The bracket $[\mathcal{Z}_i,F_j]$ vanishes for the same reason.
\end{proof}

\begin{Lem}
\label{lem:removeColumn}
Suppose that $R \simeq R^{\prime} \oplus \C$, where $\C$ is the trivial representation of $\bos$. Then $\GW \simeq \bos_{R^{\prime}} \oplus \psloo$ as metric Lie superalgebras.
\end{Lem}

\begin{proof}
The assumption $R \simeq R^{\prime}\oplus \C$ translates to the vanishing $Q_i =0$ for some $1 \leq i \leq n$. With this observation, the lemma follows from Proposition \ref{prop:aQGenRel} and Example \ref{ex:pslSpecialization}.
\end{proof}

\begin{Lem}
\label{lem:subAlg}
Let $1 \leq a \leq r$ and $1 \leq i \leq n$ be such that $Q_{ai} \neq 0$. Define
\[
Z_a^{\prime} = \frac{1}{Q_{ai}} Z_a - \frac{\met_{aa}}{2 Q_{ai}^2} \mathcal{Z}_i.
\]
Keeping the notation of Example \ref{ex:gloo}, the assignment $\{N, E, \psi^+, \psi^-\} \mapsto \{Z_a^{\prime}, \mathcal{Z}_i, E_i, F_i\}$ extends to a Lie superalgebra isometry
\[
\gloo \simeq \span_{\C} \{Z_a^{\prime}, \mathcal{Z}_i, E_i, F_i\} \subset \GW,
\]
where $\gloo$ is given the metric $\str_{\C^{1 \vert 1}}$.
\end{Lem}

\begin{proof}
The lemma follows from the equalities
\[
\met(\frac{1}{Q_{ai}} Z_a,\frac{1}{Q_{ai}} Z_a) = \frac{\met_{aa}}{Q_{ai}^2},
\qquad
\met(\frac{1}{Q_{ai}} Z_a,\mathcal{Z}_i) = 1,
\qquad
\met(\mathcal{Z}_i,\mathcal{Z}_j) = 0,
\]
the last of which follows from equation \eqref{eq:fundIden}. In particular, the second equality implies that $\mathcal{Z}_i$ is non-zero.
\end{proof}

\subsection{Effective metrics}
\label{sec:effMet}

Recall that in compact Chern--Simons theory, the level $\met$ receives quantum corrections, leading to the famous dual Coxeter shift \cite[\S 2]{witten1989}. Similarly, the level $\met$ receives one-loop quantum corrections in B-twisted Gaiotto--Witten theory \cite[\S 7]{costello2023}, \cite[\S 3]{garner2023}. The resulting effective metric $\emet$ will play an important role in the construction of relative modular categories in Section \ref{sec:cptBosonic}.

The effective, or quantum corrected, metric on $\bos$ is defined by
\[
\emet = \met + \sum_{i=1}^n Q_i \otimes Q_i.
\]
Using equation \eqref{eq:fundIdenAbs}, we verify that $\emet$ is indeed non-degenerate. It follows that $\emet$ induces a metric $\emet^{\vee}$ on $\bos^{\vee}$. Explicitly, we have
\[
\emet^{\vee} = \met^{\vee} - \sum_{i=1}^n \met^{\sharp}(Q_i) \otimes \met^{\sharp}(Q_i).
\]

Given $\lambda \in \bos^{\vee}$ and $1 \leq i \leq n$, let
\[
\chi_i(\lambda) = \met^{\vee}(Q_i,\lambda).
\]
Equation \eqref{eq:fundIdenAbs} implies that $\emet^{\vee}(Q_i,\lambda) = \chi_i(\lambda)$. With this notation, we have
\begin{equation*}
\label{eq:effMetChi}
\emet^{\vee}(\lambda, \mu)
=
\met^{\vee}(\lambda, \mu) - \sum_{i=1}^n \chi_i(\lambda) \chi_i(\mu),
\qquad
\lambda, \mu \in \bos^{\vee}.
\end{equation*}

\subsection{Unrolled quantum abelian Gaiotto--Witten Lie superalgebras}
\label{sec:unrolledQuant}

Motivated by earlier unrolled quantizations of Lie (super)algebras \cite[\S 2]{costantino2015}, \cite[\S 3]{geer2018}, \cite[\S 2]{geerYoung2022}, in this section we define an unrolled quantization of $\GW$.

Fix a parameter $\hbar \in \C \setminus \pi \sqrt{-1} \Z$ and set $q=e^{\hbar} \in \C^{\times} \setminus \{\pm 1\}$. For $x \in \C$, define
\[
q^x = e^{\hbar x},
\qquad
[x]_q = \frac{q^x - q^{-x}}{q-q^{-1}}.
\]

In order to quantize $\GW$, we henceforth make the following assumption on the input data from Section \ref{sec:abGWClassical}.

\begin{Assump}
\label{assump:quantAssumptDualInt} For each $1 \leq a \leq r$ and $1 \leq i \leq n$, the quantity $\sum_{b=1}^r \met^{ab}Q_{bi}$ is an integer.
\end{Assump}

\begin{Def}
\label{def:unrolledGW}
The \emph{unrolled quantum group} $\Uq$ is the unital complex superalgebra with even generators $\{Z_a,K_a^{\pm 1} \mid 1 \leq a \leq r\}$, odd generators $\{E_i, F_i \mid 1 \leq i \leq n\}$ and defining relations 
\begin{equation}
\label{eq:KInvert}
K_a K_a^{-1} = K_a^{-1} K_a =1,
\end{equation}
\begin{equation}
\label{eq:cartanCommute}
[Z_a,Z_b]=[Z_a,K_b] =[K_a,K_b] =0,
\end{equation}
\begin{equation}
\label{eq:roots}
[Z_a,E_i]= Q_{ai} E_i,
\qquad
[Z_a,F_i]= -Q_{ai} F_i,
\end{equation}
\begin{equation}
\label{eq:quantumRoots}
K_a E_i = q^{Q_{ai}} E_i K_a,
\qquad
K_a F_i = q^{-Q_{ai}} F_i K_a,
\end{equation}
\begin{equation}
\label{eq:nilpotent}
[E_i,E_j] = [F_i,F_j]=0,
\end{equation}
\begin{equation}
\label{eq:oddOdd}
[E_i,F_j] = \delta_{ij} \frac{\prod_{a=1}^r K_a^{\sum_{b=1}^r\met^{ab} Q_{bi}} - \prod_{a=1}^r K_a^{-\sum_{b=1}^r\met^{ab} Q_{bi}}}{q-q^{-1}}.
\end{equation}
\end{Def}

For each $1 \leq i \leq n$, define
\begin{equation}
\label{eq:unrolledCentral}
\mathcal{K}_i = \prod_{a=1}^r K_a^{\sum_{b=1}^r\met^{ab} Q_{bi}} \in \Uq.
\end{equation}
With this notation, relation \eqref{eq:oddOdd} becomes $[E_i,F_j] = \delta_{ij} \frac{\mathcal{K}_i - \mathcal{K}_i^{-1}}{q-q^{-1}}$. Note also that relations \eqref{eq:nilpotent} imply the nilpotency of $E_i$ and $F_i$:
\[
E_i^2=F_i^2=0,
\qquad
1 \leq i \leq n.
\]

\begin{Rem}
\label{rem:weakenAssump}
Suppose that $n=0$. In this case, relations \eqref{eq:roots}-\eqref{eq:oddOdd} are not present and we may take $\hbar \in \pi \sqrt{-1} \Z$, and hence $q=-1$, in Definition \ref{def:unrolledGW}.
\end{Rem}

\begin{Lem}
\label{lem:centralK}
For each $1 \leq i \leq n$, the elements $\mathcal{Z}_i$ and $\mathcal{K}_i$ are central in $\Uq$.
\end{Lem}

\begin{proof}
This follows from equation \eqref{eq:fundIden}, as in the proof of Lemma \ref{lem:centralElements}.
\end{proof}

View $\Uq \otimes_{\C} \Uq$ as a superalgebra via the bilinear extension of the product
\[
(a_1 \otimes b_1) \cdot (a_2 \otimes b_2)
=
(-1)^{\p b_1 \p a_2} a_1 a_2 \otimes b_1 b_2,
\]
where $a_i, b_i \in \Uq$ are homogeneous.

\begin{Lem}
\label{lem:unrolledGWHopf}
The assignments
\[
\epsilon(Z_a)=0,
\qquad
\epsilon(K_a^{\pm 1})=1,
\qquad
\epsilon(E_i)=0,
\qquad
\epsilon(F_i)=0,
\]
\[
\Delta Z_a = Z_a \otimes 1 + 1 \otimes Z_a,
\qquad
\Delta (K_a^{\pm 1}) = K_a^{\pm 1} \otimes K_a^{\pm 1},
\]
\[
\Delta E_i = E_i \otimes \mathcal{K}_i^{-1} + 1 \otimes E_i,
\qquad
\Delta F_i = F_i \otimes 1 + \mathcal{K}_i \otimes F_i,
\]
\[
S(Z_a) = -Z_a,
\qquad
S(K_a^{\pm 1}) = K_a^{\mp 1},
\qquad
S(E_i) = -\mathcal{K}_i E_i,
\qquad
S(F_i) = - F_i \mathcal{K}_i^{-1}
\]
extend to a unique Hopf superalgebra structure on $\Uq$ with counit $\epsilon$, coproduct $\Delta$ and antipode $S$.
\end{Lem}

\begin{proof}
We need to prove that the defining relations of $\Uq$ are preserved by $\epsilon$, $\Delta$ and $S$. We only treat the most involved case, relation \eqref{eq:oddOdd}. The verification for the counit is immediate. For the coproduct, note that multiplicativity of $\Delta$ implies that $\Delta (\mathcal{K}_i^{\pm 1}) = \mathcal{K}_i^{\pm 1} \otimes \mathcal{K}_i^{\pm 1}$. Using this, we compute
\[
\Delta(E_i F_i)
=
E_i F_i \otimes \mathcal{K}_i^{-1} + \mathcal{K}_i E_i \otimes \mathcal{K}_i^{-1} F_i - F_i \otimes E_i + \mathcal{K}_i \otimes E_i F_i
\]
and
\[
\Delta(F_i E_i)
=F_i E_i \otimes \mathcal{K}_i^{-1} + F_i \otimes E_i - \mathcal{K}_i E_i \otimes \mathcal{K}_i^{-1} F_i + \mathcal{K}_i \otimes F_i E_i
\]
so that
\[
\Delta \big( [E_i,F_i] - \frac{\mathcal{K}_i- \mathcal{K}_i^{-1}}{q-q^{-1}} \big)
=
\big([E_i,F_i] + \frac{\mathcal{K}_i^{-1}}{q-q^{-1}} \big) \otimes \mathcal{K}_i^{-1}
+ \mathcal{K}_i \otimes \big( [E_i,F_i] - \frac{\mathcal{K}_i}{q-q^{-1}}\big).
\]
Subtract $\frac{\mathcal{K}_i \otimes \mathcal{K}_i^{-1}}{q-q^{-1}}$ from the first term and add it to the second to get
\[
\Delta \big( [E_i,F_i] - \frac{\mathcal{K}_i- \mathcal{K}_i^{-1}}{q-q^{-1}} \big)
=
\big( [E_i,F_i] - \frac{\mathcal{K}_i - \mathcal{K}_i^{-1}}{q-q^{-1}} \big) \otimes \mathcal{K}_i
+ \mathcal{K}_i \otimes \big( [E_i,F_i] - \frac{\mathcal{K}_i - \mathcal{K}_i^{-1}}{q-q^{-1}}\big),
\]
as required. For the antipode, after noting that the anti-homomorphism property of $S$ implies $S(\mathcal{K}^{\pm 1}_i)=\mathcal{K}_i^{\mp 1}$, we compute
\begin{eqnarray*}
S([E_i,F_i] - \frac{\mathcal{K}_i- \mathcal{K}_i^{-1}}{q-q^{-1}})
&=&
-S(F_i) S(E_i) - S(E_i) S(F_i) - \frac{S(\mathcal{K}_i)- S(\mathcal{K}_i^{-1})}{q-q^{-1}} \\
&=&
- \big( [E_i,F_i] - \frac{\mathcal{K}_i- \mathcal{K}_i^{-1}}{q-q^{-1}} \big) .
\end{eqnarray*}

That the above data satisfies the Hopf condition can also be verified directly. For example, writing $m$ and $\iota$ for multiplication and the unit of $\Uq$, respectively, we have
\[
(m \circ (\id \otimes S) \circ \Delta) (K_a)
=
m(K_a \otimes K_a^{-1})
=
1,
\]
which is equal to $\iota (\epsilon(K_a))$, and
\[
(m \circ (\id \otimes S) \circ \Delta) (E_i)
=
m(E_i \otimes \mathcal{K}_i - 1 \otimes \mathcal{K}_i E_i)
=
E_i \mathcal{K}_i - \mathcal{K}_i E_i
=
0,
\]
the last equality following from Lemma \ref{lem:centralK}, which is equal to $\iota (\epsilon(E_i))$.
\end{proof}

Let $I$ and $J$ be disjoint subsets of $\{1, \dots, n\}$. Denote by $\sgn(I,J) \in \{\pm 1\}$ the sign of the $(\vert I \vert, \vert J \vert)$-shuffle associated to the decomposition $I \sqcup J$. The characteristic function of $I$ is
\[
\delta_{l,I}
=
\begin{cases}
1 & \mbox{if } l \in I,\\
0 & \mbox{if } l \notin I,
\end{cases}
\qquad
1 \leq l \leq n.
\]
Writing $I = \{i_1 < \cdots < i_k\}$, set $E_I = E_{i_1} \cdots E_{i_k} \in \Uq$ and, analogously, $F_I$ and $\mathcal{K}_I$.

\begin{Lem}
\label{lem:basicReln}
Let $M \subset \{1, \dots, n\}$ and $1 \leq l \leq n$. The following equalities hold:
\begin{equation}
\label{eq:coprodEM}
\Delta(E_M)
=
\sum_{\substack{I, J \subset \{1, \dots, n\} \\ I \sqcup J = M}} \sgn(I,J) E_I \otimes E_J \mathcal{K}^{-1}_I.
\end{equation}
\begin{equation}
\label{eq:coprodFM}
\Delta(F_M)
=
\sum_{\substack{ I, J \subset \{1, \dots, n\} \\ I \sqcup J = M}} \sgn(I,J) \mathcal{K}_J F_I \otimes F_J.
\end{equation}
\begin{equation}
\label{eq:EwithE}
E_l E_M = (1 - \delta_{l,M}) \sgn(\{l\}, M) E_{M \sqcup l}.
\end{equation}
\begin{equation}
\label{eq:EwithFComm}
[E_l, F_M]
=
\delta_{l,M} \sgn(\{l\},M \setminus \{l\}) \frac{\mathcal{K}_l - \mathcal{K}_l^{-1}}{q-q^{-1}} F_{M \setminus l}.
\end{equation}
In equations \eqref{eq:coprodEM} and \eqref{eq:coprodFM}, the sums are over disjoint subsets of $I,J \subset \{1, \dots, n\}$ whose union is $M$.
\end{Lem}

\begin{proof}
Equations \eqref{eq:coprodEM} and \eqref{eq:coprodFM} follow from repeated use of the definitions of $\Delta E_i$ and $\Delta F_i$ and centrality of $\mathcal{K}_i$ (Lemma \ref{lem:centralK}). Equations \eqref{eq:EwithE} and \eqref{eq:EwithFComm} follow from repeated application of relations \eqref{eq:nilpotent} and \eqref{eq:oddOdd}, respectively.
\end{proof}

There are obvious analogues of equations \eqref{eq:EwithE} and \eqref{eq:EwithFComm} with the roles of $E_i$ and $F_i$ reversed.

The following quantum counterpart of Lemma \ref{lem:removeColumn} can be proved directly from Definition \ref{def:unrolledGW}.

\begin{Lem}
\label{lem:removeColumnQuantum}
Suppose that $R \simeq R^{\prime} \oplus \C$, where $\C$ is the trivial representation of $\bos$. Then
\[
\Uq \simeq U^{\bos}_q(\bos_{R^{\prime}}) \otimes_{\C} U(\psloo)
\]
as Hopf superalgebras.
\end{Lem}

The quantum counterpart of Lemma \ref{lem:subAlg} is compatible with the Hopf structure of $\Uq$. Indeed, under the same assumptions, the assignment of Lemma \ref{lem:subAlg}, supplemented by $H \mapsto \mathcal{K}_i$, extends to a morphism of Hopf superalgebras $U_q^H(\gloo) \rightarrow \Uq$ which is an isomorphism onto its image. Here $U_q^H(\gloo) $ denotes the unrolled quantum group of $\gloo$, as defined in \cite[\S 2.2]{geerYoung2022}, with $H$ the grouplike element associated to the central generator $E$.

\subsection{Modified unrolled quantum groups}
\label{sec:modUnrolledQuant}

We introduce a modified version of the unrolled quantum group of Definition \ref{def:unrolledGW}. We continue to fix $\hbar \in \C \setminus \pi \sqrt{-1} \Z$ and set $q=e^{\hbar} \in \C^{\times} \setminus \{\pm 1\}$.

\begin{Def}
\label{def:modUnrolledGW}
The \emph{modified unrolled quantum group} $\Uqm$ is the unital complex superalgebra with even generators $\{Z_a,K_a^{\pm 1} \mid 1 \leq a \leq r\}$, odd generators $\{E_i, F_i \mid 1 \leq i \leq n\}$ and defining relations \eqref{eq:KInvert}-\eqref{eq:roots}, \eqref{eq:nilpotent} and \eqref{eq:oddOdd} and (replacing relations \eqref{eq:quantumRoots})
\begin{equation}
\label{eq:modQuantumRoots}
K_a E_i = q^{2Q_{ai}} E_i K_a,
\qquad
K_a F_i = q^{-2Q_{ai}} F_i K_a.
\end{equation}
\end{Def}

Continuing to define $\mathcal{K}_i$ by equation \eqref{eq:unrolledCentral}, all results from Section \ref{sec:unrolledQuant} hold for $\Uqm$ as stated. In particular, $\Uqm$ is a Hopf superalgebra with $\epsilon$, $\Delta$ and $S$ as in Lemma \ref{lem:unrolledGWHopf}.

\begin{Prop}
\label{prop:modifiedIso}
Assume that $q \neq \pm \sqrt{-1}$. The assignments
\[
Z_a \mapsto Z_a,
\qquad
K^{\pm 1}_a \mapsto K^{\pm 1}_a,
\qquad
E_i \mapsto E_i,
\qquad
F_i \mapsto [2]_q^{-1} F_i
\]
extend to a Hopf superalgebra isomorphism $U_{q^2}^{\bos}(\GW) \rightarrow \Uqm$.
\end{Prop}

\begin{proof}
This is a direct calculation.
\end{proof}

In particular, $\Uqm$ can be realized as $U^{\bos}_{q^{\prime}}(\GW)$ for some $q^{\prime} \in \C^{\times} \setminus \{\pm 1\}$ precisely when $q$ is not a primitive fourth root of unity. The TQFT constructions of this paper, given in Sections \ref{sec:cptBosonic} and \ref{sec:absRelMod}, are in terms of the superalgebras $\Uqmin$, hence our need for $\Uqm$.

To clarify the meaning of $\Uqm$, define the $h$-adic quantization $U_h(\GW)$ to be the unital topological $\C\pser{h}$-superalgebra with even generators $\{Z_a \mid 1 \leq i \leq r\}$, odd generators $\{E_i,F_i \mid 1 \leq i \leq n\}$ and defining relations \eqref{eq:cartanCommute}, \eqref{eq:roots}, \eqref{eq:nilpotent} and \eqref{eq:oddOdd}, where now $q=e^h$ and $K_a = e^{hZ_a}$, interpreted as power series in $h$. The classical limit of $U_h(\GW)$, obtained by working to order $h$, recovers the presentation of $\GW$ given in Proposition \ref{prop:aQGenRel}. Define $\widetilde{U}_h(\GW)$ similarly, where now $q = e^{\frac{h}{2}}$ and $K_a = q^{2 h Z_a} = e^{h Z_a}$. The classical limit of $\widetilde{U}_h(\GW)$ is the Lie superalgebra $\tilde{\bos}_R$ with even generators $\{Z_a \mid 1 \leq a \leq n\}$, odd generators $\{E_i, F_i \mid 1 \leq i \leq n\}$ and defining relations \eqref{eq:cartanCommuteClass}-\eqref{eq:nilpotentClass} and (replacing relation \eqref{eq:oddOddClass})
\[
[E_i, F_j] = 2\delta_{ij} \sum_{a, b=1}^r \met^{ab} Q_{a i} Z_b.
\]
In particular, the scaled basis $\{Z_a \mid 1 \leq a \leq r\} \sqcup \{E_i, \frac{1}{2} F_i \mid 1 \leq i \leq n\}$ of $\tilde{\bos}_R$ gives a presentation of $\GW$ as in Proposition \ref{prop:aQGenRel}, so that $\tilde{\bos}_R \simeq \GW$. In other words, $\Uqm$ is an unrolled quantization of a non-standard basis of $\GW$.

\section{The category of weight $\Uq$-modules}
\label{sec:weightMod}

We develop the representation theory of $\Uq$ and its modified variant $\Uqm$. Since all results for $\Uq$ carry over to $\Uqm$ with only minor changes, in Sections \ref{sec:basicDef}-\ref{sec:uniMod} we focus on $\Uq$ and explain the required modifications for $\Uqm$ in Section \ref{sec:moduleCatModUnrolled}.

\subsection{Weight $\Uq$-modules}
\label{sec:basicDef}

Let $V$ be a $\Uq$-module. A vector $v \in V$ is called a \emph{weight vector of weight $\lambda \in \bos^{\vee}$} if $Z v = \lambda(Z) v$ for all $Z \in \bos$. With respect to the fixed basis $\{Z_a \mid 1 \leq a \leq r\}$ of $\bos$, write the components of $\lambda$ as $\lambda_a$, $1 \leq a \leq r$. If $v$ is homogeneous, we often refer to the pair $(\lambda, \p v) \in \bos^{\vee} \times \Ztwo$ as the weight of $v$.

\begin{Def}
A \emph{weight $\Uq$-module} is a finite dimensional $\Uq$-module $V$ on which $\bos$ acts semisimply and $K_a v = q^{\lambda_a} v$ for all weight vectors $v$ of weight $\lambda$.
\end{Def}

Let $\cat_R$ be the category of weight $\Uq$-modules and their $\Uq$-linear maps of degree $\p 0$. The restriction to maps of degree $\p 0$ is crucial to the construction of a braiding on $\cat_R$ in Section \ref{sec:braiding}. The category $\cat_R$ is $\C$-linear, abelian and locally finite. The superbialgebra structure of $\Uq$ gives $\cat_R$ a monoidal structure with monoidal unit $\mathbb{I} = \C$ the trivial module.

Given $V \in \cat_R$ and $v \in V$ of weight $(\lambda, \p v) \in \bos^{\vee} \times \Ztwo$, relations \eqref{eq:roots} imply that $E_i v$ and $F_i v$ are of weight $(\lambda+Q_i, \p v + \p 1)$ and $(\lambda-Q_i, \p v + \p 1)$, respectively, while the operator equations $K_a=q^{Z_a}$ imply
\begin{equation}
\label{eq:unrolledCentralAct}
\mathcal{K}_i v = q^{\chi_i(\lambda)} v,
\qquad
1 \leq i \leq n.
\end{equation}
Call $v$ \emph{highest weight vector} (resp. \emph{lowest weight vector}) if $E_i v =0$ (resp. $F_i v=0$) for each $1 \leq i \leq n$.

\begin{Lem}
\label{lem:existHW}
Every non-zero object of $\cat_R$ has a homogeneous highest weight vector.
\end{Lem}

\begin{proof}
Let $V \in \cat_R$ be non-zero and $v \in V$ a non-zero homogeneous vector. Because of the relations \eqref{eq:nilpotent}, there exists a unique maximal subset $I \subset \{1, \dots, n\}$ with the property that $E_I v \neq 0$. By maximality of $I$, the homogeneous vector $E_I v$ is highest weight.
\end{proof}

Given $V \in \cat_R$, define $V^{\vee} \in \cat_R$ to be the super vector space $\Hom_{\C}(V,\C)$ with $\Uq$-module structure
\[
(x \cdot f)(v) = (-1)^{\p f \p x} f(S(x)v),
\qquad
v \in V, \; f \in V^{\vee}, \; x \in \Uq.
\]
Let $\{v_i\}_i$ be a homogeneous basis of $V$ with dual basis $\{v_i^{\vee}\}_i$. Define
\begin{equation}\label{E:tcoev}
\tev_V(f \otimes v) = f(v),
\qquad
\tcoev_V(1)=\sum_i v_i \otimes v_i^{\vee}
\end{equation}
and
\begin{equation}\label{E:coev}
\ev_V(v \otimes f) = (-1)^{\p f \p v}f(\mathcal{K} v),
\qquad
\coev_V(1) =\sum_i (-1)^{\p v_i}v_i^{\vee} \otimes \mathcal{K}^{-1} v_i,
\end{equation}
where we have written $\mathcal{K}$ for $\mathcal{K}_{\{1, \dots, n\}}$.

\begin{Lem}\label{lem:catQPivot}
The maps \eqref{E:tcoev} and \eqref{E:coev} define a pivotal structure on $\cat_R$.
\end{Lem}

\begin{proof}
We verify $\Uq$-linearity of $\ev_V$. It is immediate that $\ev_V$ commutes with the action of $Z_a$ and hence $K_a$. Using that $\mathcal{K}$ is central (Lemma \ref{lem:centralK}), we compute
\begin{eqnarray*}
\ev_V(E_i \cdot (v \otimes f))
&=&
\ev_V(E_i v \otimes \mathcal{K}_i^{-1} f + (-1)^{\p v} v \otimes E_i f) \\
&=&
(-1)^{(\p v +1) \p f} (\mathcal{K}_i^{-1} f )(\mathcal{K} E_i v) + (-1)^{\p v + \p v(\p f +1)}(E_i f)(\mathcal{K} v) \\
&=&
(-1)^{(\p v +1) \p f} f(\mathcal{K}_i \mathcal{K}E_i v) + (-1)^{(\p v+1)\p f}f(-\mathcal{K}_i \mathcal{K} E_i v) \\
&=&
0
\end{eqnarray*}
and
\begin{eqnarray*}
\ev_V(F_i \cdot (v \otimes f))
&=&
\ev_V(F_i v \otimes f + (-1)^{\p v} \mathcal{K}_i v \otimes F_i f) \\
&=&
(-1)^{(\p v +1) \p f} f(\mathcal{K} F_i v) + (-1)^{\p v + \p v(\p f +1)}(F_i f)(\mathcal{K} \mathcal{K}_i v) \\
&=&
(-1)^{(\p v +1) \p f} f(\mathcal{K} F_i v) + (-1)^{(\p v+1)\p f}f(-F_i \mathcal{K}_i^{-1} \mathcal{K} \mathcal{K}_i v) \\
&=&
0.
\end{eqnarray*}
We omit the straightforward verifications that the maps \eqref{E:tcoev} and \eqref{E:coev} satisfy the snake relations and pivotal conditions.
\end{proof}

\subsection{Verma and simple modules}
\label{sec:verma}

Given a subset $I \subset \{1, \dots, n\}$, write $Q_I \in \bos^{\vee}$ for the sum $\sum_{i \in I} Q_i$. 
 
Let $\UqN$ be the subalgebra of $\Uq$ generated by $\{Z_a, K_a^{\pm 1} \mid 1 \leq a \leq r\}$ and $\{E_i \mid 1 \leq i \leq n\}$. For each $(\lambda, \p p) \in \bos^{\vee} \times \Ztwo$, let $\C_{(\lambda, \p p)}$ be the vector space $\C$ concentrated in degree $\p p$ with $\UqN$-module structure
\[
Z_a \cdot 1 = \lambda_a,
\qquad
K_a \cdot 1 = q^{\lambda_a},
\qquad
E_i \cdot 1 =0.
\]

\begin{Def}
\label{def:Verma}
The \emph{Verma module of highest weight $(\lambda, \p p) \in \bos^{\vee} \times \Ztwo$} is
\[
V_{(\lambda, \p p)}
:=
\Uq \otimes_{\UqN} \C_{(\lambda, \p p)}.
\]
\end{Def}

The module $V_{(\lambda, \p p)}$ is generated by the highest weight vector $v_{\varnothing} = 1 \otimes 1$. The set $\{v_I := F_I v_{\varnothing} \mid I \subset \{1, \dots, n\} \}$ is a weight basis of $V_{(\lambda, \p p)}$ with $v_I$ having weight $(\lambda- Q_I, \p p + \p I)$, where we have written $\p I$ for $\overline{\vert I \vert}$. Using this basis, we verify that the quantum dimension of a Verma module vanishes:
\[
\qdim V_{(\lambda,\p p)}
=
\sum_I \ev_{V_{(\lambda,\p p)}}(v_I \otimes v_I^{\vee}) \\
=
(-1)^{\p p} q^{\sum_{i=1}^n \chi_i(\lambda)} \sum_I (-1)^{\p I}
=
0.
\]
Here and below, $\sum_I$ indicates a sum over the power set of $\{1, \dots, n\}$. We will see in Section \ref{sec:uniMod} below that, in many cases, there exists a modified trace on $\cat_R$ such that the modified dimension of a generic Verma module is non-zero.

Denote by $\lambda^{\vee} \in \bos^{\vee}$ the weight $-\lambda + Q_{\{1,\dots,n\}}$.

\begin{Lem}
\label{lem:dualVerma}
There is an isomorphism $V^{\vee}_{(\lambda, \p p)} \simeq V_{(\lambda^{\vee},\p p + \p n)}$ in $\cat_R$.
\end{Lem}

\begin{proof}
Note that $v^{\vee}_I \in V^{\vee}_{(\lambda, \p p)}$ is of weight $(-\lambda + Q_I, \p p + \p I)$. In particular, $v^{\vee}_{\{1, \dots, n\}}$ has weight $(\lambda^{\vee},\p p+\p n)$ and is of highest weight. It follows that the assignment $v^{\vee}_{\{1, \dots, n\}} \mapsto v_{\varnothing}$ extends to an isomorphism $V^{\vee}_{(\lambda, \p p)} \xrightarrow[]{\sim} V_{(\lambda^{\vee},\p p + \p n)}$.
\end{proof}

\begin{Prop}
\label{prop:simpleVerma}
The module $V_{(\lambda, \p p)}$ is simple if and only if $\prod_{i=1}^n [\chi_i(\lambda)]_q \neq 0$.
\end{Prop}

\begin{proof}
Let $M \subset V_{(\lambda, \p p)}$ be a non-zero submodule. By Lemma \ref{lem:existHW}, there exists a highest weight vector $v \in M$. Writing $v = \sum_I c_I v_I$ in the standard weight basis of $V_{(\lambda, \p p)}$, we use equation \eqref{eq:EwithFComm} to compute
\[
0 = E_l v = \sum_I \delta_{l,I} \sgn(\{l\}, I \setminus l)[\chi_l(\lambda)]_q c_I v_{I \setminus l},
\qquad
1 \leq l \leq n.
\]
It follows that if $c_I \neq 0$, then $[\chi_i(\lambda)]_q=0$ for all $i \in I$.

We can now prove the proposition. If $\prod_{i=1}^n [\chi_i(\lambda)]_q \neq 0$, then $v$ is a non-zero multiple of $v_{\varnothing}$ and $M = V_{(\lambda, \p p)}$, proving simplicity of $V_{(\lambda, \p p)}$. Conversely, if $[\chi_i(\lambda)]_q=0$ for some $1 \leq i \leq n$, then $v_{\{i\}}$ is a homogeneous highest weight vector which generates a submodule of $V_{(\lambda, \p p)}$ which is strict, since it does not contain $v_{\varnothing}$.
\end{proof}

\begin{Def}
A weight $\lambda \in \bos^{\vee}$ is called \emph{typical} if
\[
\prod_{i=1}^n [\chi_i(\lambda)]_q \neq 0.
\]
Otherwise, $\lambda$ is called \emph{atypical}.
\end{Def}

In terms of the parameter $\hbar$, atypicality of $\lambda$ is the existence of an index $1 \leq i \leq n$ such that $\chi_i(\lambda) \in \frac{\pi \sqrt{-1}}{\hbar} \Z$.

\begin{Prop}
\label{prop:typicalExist}
There exists a typical weight if and only if the representation $R$ has no trivial summands.
\end{Prop}

\begin{proof}
Note that $R$ has a trivial summand if and only if $Q_i=0$ for some $1 \leq i \leq n$. Recall that $\chi_i=\met^{\vee}(Q_i,-)$. If $Q_i$ is zero, then $\chi_i(\lambda)= 0$ for all $\lambda \in \bos^{\vee}$ and all weights are atypical. The converse follows from the proceeding lemma applied to the linear functionals $\frac{\hbar}{\pi \sqrt{-1}} \chi_i: \bos^{\vee} \rightarrow \C$, $1 \leq i \leq n$.
\end{proof}

\begin{Lem}
Let $W$ be a finite dimensional complex vector space and $\ell_1, \dots, \ell_n \in W^{\vee}$ non-zero linear functionals. There exists a vector $w \in W$ such that $\ell_i(w) \notin \Z$ for all $1 \leq i \leq n$.
\end{Lem}

\begin{proof}
Note that $\ell_i^{-1}(\C \setminus \Z) = W \setminus \ell_i^{-1}(\Z)$ with $\ell_i^{-1}(\Z)$ a countable union of affine hyperplanes. It follows that
\[
\bigcap_{i=1}^n \ell_i^{-1}(\C \setminus \Z) 
=
W \setminus \bigcup_{i=1}^n \ell_i^{-1}(\Z)
\]
is the complement of countably many affine hyperplanes and thus non-empty.
\end{proof}

Motivated by Proposition \ref{prop:typicalExist}, we henceforth work under the following assumption.

\begin{Assump}
\label{assump:noZeroRoots}
The representation $R$ has no trivial summands.
\end{Assump}

In terms of roots, Assumption \ref{assump:noZeroRoots} is the statement that $Q_1, \dots, Q_n$ are each non-zero. If Assumption \ref{assump:noZeroRoots} does not hold, then, in view of Lemma \ref{lem:removeColumnQuantum}, the study of $\Uq$ reduces to that of $U_q^{\bos}(\bos_{R^{\prime}})$ and $\dim_{\C} \Hom_{\bos}(\C,R)$ copies of $U(\psloo)$, where $R^{\prime}$ is the quotient $R \slash \Hom_{\bos}(\C,R)$. Relative modular categories arising from the representation theory of $\psloo$ are treated in Section \ref{sec:pslooCS}. In particular, we will still be able to model Gaiotto--Witten theories which do not satisfy Assumption \ref{assump:noZeroRoots}.

\begin{Lem}
\label{lem:simpQuot}
Every simple object of $\cat_R$ is a quotient of a Verma module.
\end{Lem}

\begin{proof}
Let $V \in \cat_R$ be simple and $v \in V$ a homogeneous highest weight vector of weight $(\lambda, \p p)$. The assignment $v_{\varnothing} \mapsto v$ extends to a morphism $V_{(\lambda,\p p)} \rightarrow V$ in $\cat_R$ which, since $V$ is simple, is necessarily an epimorphism.
\end{proof}

For each $1 \leq i \leq n$, let $\langle v_{\{i\}} \rangle \subset V_{(\lambda, \p p)}$ be the submodule generated by $v_{\{i\}}$.

\begin{Prop}
\label{prop:vermaQuotient}
The module $V_{(\lambda, \p p)}$ has a unique simple quotient
\[
0 \rightarrow  \sum_{\substack{1 \leq i \leq n \\ [\chi_i(\lambda)]_q=0}} \langle v_i \rangle  \rightarrow V_{(\lambda, \p p)} \rightarrow S_{(\lambda, \p p)} \rightarrow 0.
\]
Moreover, the dimension of $S_{(\lambda, \p p)}$ is $2^{n-k}$, where $k$ is the number of indices $1 \leq i \leq n$ such that $[\chi_i(\lambda)]_q=0$.
\end{Prop}

\begin{proof}
Observe that the submodule $M$ of $V_{(\lambda,\p p)}$ appearing in the statement of the proposition has weight basis
\[
\{v_I \mid [\chi_i(\lambda)]_q=0 \mbox{ for some } i \in I\}.
\]
Let $0 \neq N \subsetneq V_{(\lambda, \p p)}$ be a submodule and $n \in N$ a homogeneous highest weight vector. As in the proof of Proposition \ref{prop:simpleVerma}, we have $n= \sum_I c_I v_I$, where the sum is over those subsets $I \subset \{1, \dots, n\}$ which contain at least one index $i$ such that $[\chi_i(\lambda)]_q=0$. Then $N \subset M$ and the first statement of the lemma follows. The second statement follows from the equalities $\dim_{\C} V_{(\lambda,\p p)} = 2^n$ and $\dim_{\C} M = 2^{n-k}(2^k-1)$, the latter of which follows from the stated basis of $M$.
\end{proof}

Note that $S_{(\lambda,\p p)}=V_{(\lambda,\p p)}$ when $\lambda$ is typical.

The above results immediately lead to the following highest weight classification of simple objects of $\cat_R$.

\begin{Thm}
\label{thm:simplesSplitAb}
A simple object of $\cat_R$ is isomorphic to exactly one of the following modules:
\begin{enumerate}
\item $V_{(\lambda, \p p)}$ where $\lambda \in \bos^{\vee}$ is typical and $\p p \in \Ztwo$.
\item $S_{(\lambda, \p p)}$, where $\lambda \in \bos^{\vee}$ is atypical and $\p p \in \Ztwo$.
\end{enumerate}
\end{Thm}

\begin{Lem}
\label{lem:projVerma}
Simple Verma modules are projective and injective objects of $\cat_R$.
\end{Lem}

\begin{proof}
Let $f: V \rightarrow W$ be an epimorphism and $g: V_{(\lambda,\p p)}  \rightarrow W$ a non-zero morphism. The map $g$ is determined by the highest weight vector $g(v_{\varnothing})=w \in W$ of weight $(\lambda,\p p)$. Surjectivity of $f$ implies that $w$ has a preimage under $f$, say $v$, which is of weight $(\lambda,\p p)$ and satisfies $E_l v \in \ker f$ for each $1 \leq l \leq n$. For any constants $c_I \in \C$, $I \neq \varnothing$, the vector
\[
v^{\prime} = v + \sum_{I \neq \varnothing} c_I F_I E_I v \in V
\]
is of weight $(\lambda,\p p)$ and satisfies $f(v^{\prime})=w$. Using equations \eqref{eq:EwithE} and \eqref{eq:EwithFComm}, we compute
\[
E_l v^{\prime}
=
E_l v + \sum_{I \neq \varnothing} c_I \big( (-1)^{\p I} (1 - \delta_{l,I})\sgn(\{l\}, I) F_I E_{I \sqcup l} v + \sgn(\{l\},I \setminus \{l\}) \delta_{l,I} [\chi_l(\lambda)l]_q F_{I \setminus l} E_I v \big).
\]
We conclude that $E_l v^{\prime}=0$ if and only if
\[
(-1)^{\p I} \sgn(\{l\}, I) c_I +\sgn(\{l\}, I) [\chi_l(\lambda)]_q c_{I \sqcup l} =0
\]
for all subsets $I$ which do not contain $l$. The equations have the unique solution
\[
c_I = \frac{(-1)^{\lfloor \frac{\vert I \vert}{2} \rfloor}}{\prod_{i \in I} [\chi_i(\lambda)]_q},
\qquad
I \subset \{1, \dots, n\}.
\]
Note that typicality of $\lambda$ ensures the non-vanishing of the denominator of each $c_I$. A lift of $g$ to $V_{(\lambda, \p p)} \rightarrow V$ is then determined by the assignment $v_{\varnothing} \mapsto v^{\prime}$. It follows that $V_{(\lambda,\p p)}$ is projective.

Since $\cat_R$ satisfies the assumptions of \cite[Proposition 6.1.3]{etingof2015}, projective and injective objects coincide.
\end{proof}

\subsection{Braiding}
\label{sec:braiding}

Motivated by the connection between the Lie superalgebras $\GW$ and $\gloo$, as in Lemma \ref{lem:subAlg}, and known universal $R$-matrices for $\gloo$ \cite{kulish1989,khoroshkin1991}, we construct in this section a braiding on the category $\cat_R$.

Let $V, W \in \cat_R$. Define $\Upsilon_{V,W} \in \End_{\C}(V \otimes W)$ by
\begin{equation}
\label{eq:upsilonFactor}
\Upsilon_{V,W} (v \otimes w)
=
q^{-\met^{\vee}( \lambda_v, \lambda_w)} v \otimes w,
\end{equation}
where $v \in V$ and $w \in W$ are of weight $\lambda_v$ and $\lambda_w$, respectively. Let
\[
\tilde{R}
=
\exp \left( (q-q^{-1}) \sum_{i=1}^n E_i \mathcal{K}_i \otimes F_i \mathcal{K}_i^{-1} \right).
\]
Because of the relations \eqref{eq:nilpotent}, the element $\tilde{R}$ is a finite sum and so a well-defined element of $\Uq \otimes_{\C} \Uq$. Using centrality of $\mathcal{K}_i$, we find
\[
\tilde{R}
=
\sum_I (-1)^{{\p I \choose 2}}(q-q^{-1})^{\p I} E_I \mathcal{K}_I \otimes F_I \mathcal{K}_I^{-1},
\]
where ${\p I \choose 2} = \frac{\vert I \vert (\vert I \vert -1)}{2}$. Finally, let $c_{V,W} \in \Hom_{\C}(V \otimes W, W \otimes V)$ be the composition $\tau_{V,W} \circ \tilde{R}_{V,W} \circ \Upsilon_{V,W}$, where $\tau$ is the standard symmetric braiding on the category of super vector spaces which incorporates the Koszul sign rule and $\tilde{R}_{V,W}$ is multiplication by $\tilde{R}$ on $V \otimes W$.

\begin{Prop}
\label{prop:braiding}
The morphisms $\{c_{V,W} \mid V,W \in \cat_R\}$ define a braiding on $\cat_R$.
\end{Prop}

\begin{proof}
Naturality of $c_{V,W}$ is clear, keeping in mind that morphisms of $\cat_R$ are of degree $\p 0$; allowing morphisms of degree $\p 1$ would introduce signs which would invalidate naturality. Invertibility of $c_{V,W}$ follows from the observation that $\Upsilon_{V,W}$ is invertible and the direct verification that $\tilde{R}$ is a unit with inverse
\[
\sum_I (-1)^{{\p I \choose 2} + \p I}(q-q^{-1})^{\p I} E_I \mathcal{K}_I \otimes F_I \mathcal{K}_I^{-1}.
\]

To prove $\Uq$-linearity of $c_{V,W}$, it suffices to verify that $c_{V,W}$ commutes with the action of the generators of $\Uq$. That $c_{V,W}$ commutes with $Z_a$, and hence $K_a$, follows from the fact that $c_{V,W}$ is a composition of operators of weight zero. We verify that $c_{V,W}$ commutes with the action of $E_l$; the verification for $F_l$ is similar. Let $v \in V$ and $w \in W$ be of weight $(\lambda_v, \p v)$ and $(\lambda_w, \p w)$, respectively. Applying the definitions, we find that $E_l \cdot c_{V,W}(v \otimes w)$ is equal to
\begin{multline*}
\sum_I (-1)^{{\p I \choose 2} + \p v \p I + (\p v + \p I)(\p w + \p I)}(q-q^{-1})^{\p I}  q^{-\met^{\vee}(\lambda_v,\lambda_w) + \sum_{i \in I} \chi_i(\lambda_v-\lambda_w) - \chi_l(\lambda_v)} E_l F_I w \otimes E_I v + \\
\sum_I (-1)^{{\p I \choose 2} + \p v \p I + (\p v + \p I)(\p w + \p I) + (\p w + \p I)}(q-q^{-1})^{\p I} q^{-\met^{\vee}(\lambda_v,\lambda_w) + \sum_{i \in I} \chi_i(\lambda_v-\lambda_w) } F_I w \otimes E_l E_I v
\end{multline*}
and $c_{V,W}(E_l \cdot (v \otimes w))$ is equal to
\begin{multline*}
\sum_I (-1)^{{\p I \choose 2} + (\p v +1) \p I + (\p I + 1 + \p v)(\p I + \p w)} (q-q^{-1})^{\p I} q^{-\chi_l(\lambda_w) + \sum_{i \in I} \chi_i(\lambda_v+Q_l-\lambda_w) -\met^{\vee}(\lambda_v+Q_I, \lambda_w)}   F_I w \otimes E_I E_l v + \\
\sum_I (-1)^{\p v + {\p I \choose 2} + \p v \p I + (\p I + \p v)(\p I + 1 + \p w)} (q-q^{-1})^{\p I} q^{\sum_{i \in I} \chi_i(\lambda_w-\lambda_w-Q_l) -\met^{\vee}(\lambda_v, \lambda_w+Q_I)} F_I E_l w \otimes E_I v.
\end{multline*}
Using equation \eqref{eq:EwithFComm}, we see that the coefficients of $F_I E_l w \otimes E_I v$ in $E_l \cdot c_{V,W}(v \otimes w)$ and $c_{V,W}(E_l \cdot (v \otimes w))$ are
\[
(-1)^{{\p I \choose 2} + \p v \p I + (\p v + \p I)(\p w + \p I) + \p I}(q-q^{-1})^{\p I} q^{-\met^{\vee}(\lambda_v,\lambda_w) +\sum_{i \in I} \chi_i(\lambda_v-\lambda_w) - \chi_l(\lambda_v)}
\]
and
\[
(-1)^{\p v + {\p I \choose 2} + \p v \p I + (\p I + \p v)(\p I + 1 + \p w)} (q-q^{-1})^{\p I} q^{\sum_{i \in I} \chi_i(\lambda_v-\lambda_w-Q_l) - \met^{\vee}(\lambda_v,\lambda_w+Q_l)},
\]
respectively. By equation \eqref{eq:fundIden}, we can replace $\lambda_v-\lambda_w-Q_l$ with $\lambda_v-\lambda_w$ in the second expression. Using the definition of $\chi_i$, we then see that the coefficients are equal.

Continuing, the coefficient of $F_I w \otimes E_I E_l v$ in $c_{V,W}(E_l \cdot (v \otimes w))$ is
\[
(-1)^{{\p I \choose 2} + (\p v +1) \p I + (\p I + 1 + \p v)(\p I + \p w)} (q-q^{-1})^{\p I}
q^{-\chi_l(\lambda_w) + \sum_{i \in I} \chi_i(\lambda_v+Q_l-\lambda_w) -\met^{\vee}(\lambda_v, \lambda_w) - \chi_l(\lambda_w)}.
\]
There are two contributions to the coefficient of $F_I w \otimes E_I E_l v$ in $E_l \cdot c_{V,W}(v \otimes w)$. The first is from the coefficient of $F_I w \otimes E_l E_I v$, which appears with an additional factor of $(-1)^{\p I}$ due to the equality $E_I E_l = (-1)^{\p I} E_l E_I$, and so is equal to
\begin{equation}
\label{eq:contFromSec}
(-1)^{{\p I \choose 2} + \p v \p I + (\p v + \p I)(\p w + \p I) + (\p w + \p I) + \p I}(q-q^{-1})^{\p I} q^{-\met^{\vee}(\lambda_v,\lambda_w) + \sum_{i \in I}  \chi_i(\lambda_v-\lambda_w)}.
\end{equation}
The second is from the coefficient of $E_l F_{I^{\prime}} w \otimes E_{I^{\prime}} v$, where $I^{\prime} = I \sqcup \{l\}$, which appears due to equation \eqref{eq:EwithFComm} and with an additional factor of $(-1)^{\p I}$ due to the equation
\[
\sgn(\{l\},I)F_{I^{\prime} \setminus l} w \otimes E_{I^{\prime}} v
=
(-1)^{\p I} F_I w \otimes E_I E_l v.
\]
The second contribution is therefore
\[
(-1)^{{{\p I}^{\prime} \choose 2} + \p v {\p I}^{\prime} + (\p v + {\p I}^{\prime})(\p w + {\p I}^{\prime}) + \p I}(q-q^{-1})^{{\p I}^{\prime}} q^{-\met^{\vee}(\lambda_v,\lambda_w) + \sum_{i \in I^{\prime}} \chi_i(\lambda_v - \lambda_w) - \chi_l(\lambda_v)} [\chi_l(\lambda_w)]_q.
\]
Writing $[\chi_l(\lambda_w)]_q=\frac{q^{\chi_l(\lambda_w)}}{q-q^{-1}} - \frac{q^{-\chi_l(\lambda_w)}}{q-q^{-1}}$ so as to obtain a sum of two terms, we see that the first summand is the negative of \eqref{eq:contFromSec} while the second is equal to the coefficient of $F_I w \otimes E_I E_l v$ in $c_{V,W}(E_l \cdot (v \otimes w))$; each of these statements invokes equation \eqref{eq:fundIdenAbs} to cancel certain terms. This proves equality of the coefficients of $F_I w \otimes E_I E_l v$ and hence $E_l$-linearity of $c_{V,W}$.

Next, consider the hexagon identities, in which we suppress all associators. Let $U,V,W \in \cat_R$. We verify that $c_{U, V \otimes W} = (\id_V \otimes c_{U,W}) \circ (c_{U,V} \otimes \id_W)$; verification of the second hexagon identity is similar. Fix vectors $u\in U$, $v\in V$ and $w\in W$ of weights $(\lambda_u,\p u)$, $(\lambda_v,\p v)$ and $(\lambda_w,\p w)$, respectively. A direct computation gives for $(\id_V \otimes c_{U,W}) \circ (c_{U,V} \otimes \id_W)(u \otimes v \otimes w)$ the expression
\begin{multline*}
\sum_{I,J} (-1)^{{\p I \choose 2} + \p u \p I + (\p u + \p I)(\p v + \p I) + {\p J \choose 2} + (\p u + \p I) \p J + (\p u + \p I + \p J)(\p w + \p J)}(q-q^{-1})^{\p I + \p J  + (\p v+\p I) \p J} \cdot \\
q^{-\met^{\vee}( \lambda_u, \lambda_v) + \sum_{i \in I} \chi_i(\lambda_u-\lambda_v)-\met^{\vee}( \lambda_u+ \sum_{i \in I} Q_i, \lambda_w) + \sum_{j \in J} \chi_j(\lambda_u-\lambda_w)} F_I v \otimes F_J w \otimes E_J E_I u.
\end{multline*}
By the relations \eqref{eq:nilpotent}, a summand can be non-zero only if $I \cap J = \varnothing$. Similarly, we find that $c_{U,V \otimes W}(u \otimes v \otimes w)$ is equal to
\begin{eqnarray*}
\sum_M (-1)^{{\p M \choose 2} + \p u \p M + (\p u + \p M)(\p v + \p w + \p M)}(q-q^{-1})^{\p M}
q^{-\met^{\vee}( \lambda_u, \lambda_{v\otimes w}) + \chi_i(\lambda_u-\lambda_{v \otimes w})} F_M (v \otimes w) \otimes E_M u.
\end{eqnarray*}
Using equation \eqref{eq:coprodFM}, we can write
\[
F_M (v \otimes w)
=
\sum_{\substack{ I, J \\ I \sqcup J = M}} \sgn(I,J)(-1)^{\p w \p J} q^{\sum_{j \in J} \chi_j(\lambda_b)} F_I v  \otimes F_J w.
\]
Using the equalities $\lambda_{v \otimes w} = \lambda_v + \lambda_w$ and $E_J E_I = \sgn(I,J) E_{I \sqcup J}$ (see equation \eqref{eq:EwithE}), we conclude that $c_{U, V \otimes W} = (\id_V \otimes c_{U,W}) \circ (c_{U,V} \otimes \id_W)$.
\end{proof}

\subsection{Generic semisimplicity}
\label{sec:genSS}

Let $\chi: \bos^{\vee} \rightarrow \C^n$ be the linear map whose $i$\textsuperscript{th} component is $\chi_i(\lambda)$. Viewing the root matrix $Q$ as a linear map, equation \eqref{eq:fundIdenAbs} gives a complex
\[
0 \rightarrow \C^n \xrightarrow[]{Q} \bos^{\vee} \xrightarrow[]{\chi} \C^n \rightarrow 0.
\]
Let $\GrH$ be the image of $\chi$. For each $\xi \in \GrH$, let $\cat_{R,\xi} \subset \cat_R$ be the full subcategory of modules on which the central element $\mathcal{Z}_{i}$, defined in Lemma \ref{lem:centralElements}, acts by $\xi_i$ for each $1 \leq i \leq n$. Let
\[
\SSSY = \{(\xi_1, \dots, \xi_n) \in \GrH \mid \prod_{i=1}^n [\xi_i]_q =0\}.
\]
If $n \geq 1$, then Assumption \ref{assump:noZeroRoots} ensures that $\GrH$ is non-trivial. In this case $\SSSY$ is a small symmetric subset of $\GrH$.

\begin{Prop}
\label{prop:genericSSZWt}
If $n=0$, then the category $\cat_R$ is semisimple. If $n \geq 1$, then the $\GrH$-graded category
\[
\cat_R \simeq \bigoplus_{\xi \in \GrH} \cat_{R,\xi}
\]
is generically semisimple with small symmetric subset $\SSSY$. Moreover, for $\xi \in \GrH \setminus \SSSY$, a completely reduced dominating set of $\cat_{R,\xi}$ is
\[
\{V_{(\lambda,\p p)} \mid \chi(\lambda) = \xi, \; \p p \in \Ztwo \}.
\]
\end{Prop}

\begin{proof}
When $n=0$ the category $\cat_R$ is semisimple by construction.

Assume then that $n \geq 1$. If $V \in \cat_R$ is indecomposable, then its weights lie in a single congruence class of $\Lambda_R$ in $\bos^{\vee}$. Equation \eqref{eq:fundIden} implies that $\Lambda_R \subset \ker \chi$. It follows from this and the definitions of the coproduct and antipode of $\Uq$ that $\cat_R \simeq \bigoplus_{\xi \in \GrH} \cat_{R,\xi}$ is an $\GrH$-grading. Let $\xi \in \GrH \setminus \SSSY$ and $V \in \cat_{R,\xi}$ non-zero. By Lemmas \ref{lem:existHW} and \ref{lem:simpQuot}, the submodule generated by a homogeneous highest weight vector of $V$ is isomorphic to a quotient of $V_{(\lambda,\p p)}$ for some $\lambda \in \bos^{\vee}$ with $\chi(\lambda) =\xi$. Since $\xi \notin \SSSY$, Proposition \ref{prop:simpleVerma} implies that this submodule is simple and so isomorphic to $V_{(\lambda,\p p)}$. By Lemma \ref{lem:projVerma}, $V_{(\lambda,\p p)}$ is injective, whence there exists a splitting $V\simeq V^{\prime} \oplus V_{(\lambda,\p p)}$ with $V^{\prime} \in \cat_{R,\xi}$ of dimension strictly less than that of $V$. Iterating this process shows that $\cat_{R,\xi}$ is semisimple with the claimed completely reduced dominating set.
\end{proof}

\subsection{Ribbon structure}
\label{sec:ribbStr}

Recall that the category $\cat_R$ is pivotal (Lemma \ref{lem:catQPivot}) and braided (Proposition \ref{prop:braiding}).

\begin{Thm}
\label{thm:ribbonCat}
The natural automorphism of the identity functor $\theta : \id_{\cat_R} \Rightarrow \id_{\cat_R}$ whose components are the right partial traces of the braiding,
\[
\theta_V = \ptr_R(c_{V,V}),
\qquad
V \in \cat_R
\]
gives $\cat_R$ the structure of a $\C$-linear ribbon category.
\end{Thm}

\begin{proof}
The definition of $\theta$ as the right partial trace of the braiding ensures that it satisfies the balancing conditions.

Assume first that $n \geq 1$. Give $\cat_R$ the generically semisimple structure of Proposition \ref{prop:genericSSZWt}. We are therefore in the setting of \cite[Theorem 2]{geer2018}, which asserts that if $\theta_{V^{\vee}} = \theta_V^{\vee}$ for all generic simple objects $V \in \cat_R$, then $\theta$ is a ribbon structure on $\cat_R$. By Theorem \ref{thm:simplesSplitAb} and the definition of $\SSSY$, generic simple objects are typical Verma modules. Consider then a (not necessarily simple) Verma module $V_{(\lambda,\p p)}$. Since $\End_{\cat_R}(V_{(\lambda,\p p)}) \simeq \C$, the morphism $\theta_{V_{(\lambda,\p p)}}$ is determined by the scalar by which it acts on the highest weight vector $v_{\varnothing}$. Using that the off-diagonal terms of the braiding do not contribute to $\theta_{V_{(\lambda,\p p)}} v_{\varnothing}$, we compute
\[
\theta_{V_{(\lambda,\p p)} }
=
q^{- \met^{\vee}(\lambda, \lambda) + \sum_{i=1}^n \chi_i(\lambda)}  \id_{V_{(\lambda,\p p)} }.
\]
Lemma \ref{lem:dualVerma} and repeated use of equation \eqref{eq:fundIdenAbs}  then gives $\theta_{V_{(\lambda,\p p)} ^{\vee}} = \theta^{\vee}_{V_{(\lambda,\p p)}}$.

When $n =0$, the calculation of the previous paragraph shows that $\theta_{V^{\vee}} = \theta_V^{\vee}$ for all simple objects $V$. (We may take $\SSS=\varnothing$, so that all objects are generic.) Semisimplicity of $\cat_R$ then implies that $\theta_{V^{\vee}} = \theta_V^{\vee}$ for all $V \in \cat_R$. 
\end{proof}

\begin{Ex}
\label{ex:oneDimMod}
By Proposition \ref{prop:vermaQuotient} and Theorem \ref{thm:simplesSplitAb}, one dimensional $\Uq$-modules are of the form $\C^{\bos}_{(k, \p p)}:=S_{(k, \p p)}$ for weights $(k, \p p)$ satisfying $[\chi_i(k)]_q=0$ for all $1 \leq i \leq n$. Explicitly, $\C^{\bos}_{(k, \p p)}$ is the one dimensional module of weight $(k, \p p)$ on which each $E_i$ and $F_i$ act by zero. There are canonical isomorphisms
\[
\C^{\bos}_{(k_1, \p p_1)} \otimes \C^{\bos}_{(k_2, \p p_2)} \simeq \C^{\bos}_{(k_1+k_2, \p p_1 + \p p_2)}.
\]
In particular, each one dimensional module is invertible. As in the proof of Theorem \ref{thm:ribbonCat}, we compute
\[
\theta_{\C^{\bos}_{(k, \p p)}} = q^{-\met^{\vee}(k,k) + \sum_{i=1}^n \chi_i(k)} \id_{\C^{\bos}_{(k, \p p)}}. \qedhere
\]
\end{Ex}

\subsection{Open Hopf link invariants}
\label{sec:openHopf}

Given $V, V^{\prime} \in \cat_R$, define
\[
\Phi_{V,V^{\prime}} = (\id_{V^{\prime}} \otimes \ev_V)\circ (c_{V,{V^{\prime}}}\otimes \id_{V^{\vee}})\circ (c_{{V^{\prime}},V}\otimes \id_{V^{\vee}})\circ (\id_{V^{\prime}}\otimes \tcoev_V)\in \End_{\cat_R}({V^{\prime}}).
\]
In topological terms, $\Phi_{V,V^{\prime}}$ is the value of the Reshetikhin--Turaev functor $F_{\cat_R}$ on the coloured open Hopf link
\[
\epsh{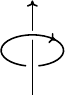}{9ex}\put(4,0){\tiny $V$}\put(-30,18){\tiny $V^{\prime}$} \hspace{15pt}.
\]
Extend the definition of $\Phi_{V,V^{\prime}}$ to formal $\C$-linear combinations of objects of $\cat_R$ by bilinearity. When $\End_{\cat_R}(V^{\prime}) \simeq \C$, write $\langle \Phi_{V,{V^{\prime}}}\rangle \in \C$ for the scalar by which $\Phi_{V,{V^{\prime}}}$ acts. 

\begin{Lem}
\label{lem:Phis}
Let $(\lambda^{\prime}, \p p^{\prime}), (\lambda, \p p) \in \bos^{\vee} \times \Ztwo$. The following equality holds:
\[
\Phi_{V_{(\lambda^{\prime}, \p p^{\prime})},V_{(\lambda, \p p)}}
=
(-1)^{\p p^{\prime} + \p n} q^{-2 \met^{\vee}(\lambda^{\prime},\lambda) + \sum_{i=1}^n \chi_i(\lambda^{\prime}+\lambda)} \prod_{i=1}^n(q^{\chi_i(\lambda)} - q^{-\chi_i(\lambda)})\id_{V_{(\lambda, \p p)}}.
\]
\end{Lem}

\begin{proof}
Since $\End_{\cat_R}(V_{(\lambda, \p p)}) \simeq \C$, the morphism $\Phi_{V_{(\lambda^{\prime}, \p p^{\prime})},V_{(\lambda, \p p)}}$ maps the highest weight vector $v_{\varnothing} \in V_{(\lambda, \p p)}$ to a multiple of itself, whence the off-diagonal terms of the braiding do not contribute to $\Phi_{V_{(\lambda^{\prime}, \p p^{\prime})},V_{(\lambda, \p p)}} v_{\varnothing}$. With this comment in mind, we use the standard weight basis of Verma modules to compute
\[
\Phi_{V_{(\lambda^{\prime}, \p p^{\prime})},V_{(\lambda, \p p)}}
=
(-1)^{\p p^{\prime}} q^{-2 \met^{\vee}(\lambda^{\prime}, \lambda) + \sum_{i=1}^n \chi_i(\lambda^{\prime})} \sum_I (-1)^{\p I} q^{2\sum_{i \in I} \chi_i(\lambda)} \id_{V_{(\lambda, \p p)}}.
\]
It remains to apply the equality
\[
\sum_I (-1)^{\p I} q^{2 \sum_{i \in I}  \chi_i(\lambda)} = \prod_{i=1}^n(1- q^{2 \chi_i(\lambda)})
\]
and simplify the result.
\end{proof}

\subsection{Unimodularity and modified traces}
\label{sec:uniMod}

We henceforth add to Assumptions \ref{assump:quantAssumptDualInt} and \ref{assump:noZeroRoots} the following assumption.

\begin{Assump}
\label{assump:unimodular}
The trivial representation appears in $\Lambda^{\bullet} R \otimes_{\C} \Lambda^{\bullet} R$ with multiplicity one.
\end{Assump}

In terms of the roots $Q_1, \dots, Q_n$ of $R$, Assumption \ref{assump:unimodular} is the statement that the sum $Q_I + Q_J \in \bos^{\vee}$ vanishes if and only if $I=J=\varnothing$.

The proof of the next result is a modification of that of \cite[Theorem 3.31]{anghel2021}.

\begin{Lem}
\label{lem:catUnimod}
If the representation $R$ satisfies Assumption \ref{assump:unimodular}, then $\cat_R$ is unimodular.
\end{Lem}

\begin{proof}
Since projectivity and injectivity coincide in $\cat_R$, it suffices to prove that the injective hull of the trivial module $\C$ is self-dual.

Let $\lambda \in \bos^{\vee}$ be typical. Then $V_{(\lambda, \p 0)}$ is projective and $V_{(\lambda, \p 0)} \otimes V_{(\lambda, \p 0)}^{\vee}$ is a direct sum of projective indecomposables, say $\bigoplus_{i=0}^m P_i$. Adjunction and simplicity of $V_{(\lambda, \p 0)}$ give isomorphisms
\[
\Hom_{\cat_R}(\C, V_{(\lambda, \p 0)} \otimes V_{(\lambda, \p 0)}^{\vee})
\simeq 
\End_{\cat_R}(V_{(\lambda, \p 0)})
\simeq \C.
\]
The injective hull of $\C$ therefore appears in $V_{(\lambda, \p 0)} \otimes V_{(\lambda, \p 0)}^{\vee}$ with multiplicity one; relabeling if necessary, denote this summand by $P_0$.

By Lemma \ref{lem:dualVerma}, the weight of $v_I \otimes v^{\vee}_J \in V_{(\lambda, \p 0)} \otimes V_{(\lambda, \p 0)}^{\vee}$ is $- Q_I +Q_J$. In particular, the vectors
\[
v_+=v_{\varnothing} \otimes v_{\{1,\dots,n\}}^{\vee},
\qquad
v_-=v_{\{1,\dots,n\}} \otimes v_{\varnothing}^{\vee}
\]
have weights $Q_{\{1,\dots,n\}}$ and $-Q_{\{1,\dots,n\}}$, respectively. Noting that $- Q_I +Q_J = Q_{\{1,\dots,n\}}$ if and only if $Q_I + Q_{\{1,\dots,n\} \setminus J}=0$ and $- Q_I +Q_J = -Q_{\{1,\dots,n\}}$ if and only if $Q_{\{1,\dots,n\} \setminus I} + Q_J=0$, we see that Assumption \ref{assump:unimodular} implies that $v_{\pm}$ spans the weight space of weight $\pm Q_{\{1,\dots,n\}}$. Hence, $v_+ \in P_i$ and $v_-\in P_j$ for some $i$ and $j$. Self-duality of $V_{(\lambda, \p 0)} \otimes V_{(\lambda, \p 0)}^{\vee}$ and the definitions of $v_+$ and $v_-$ then imply that $P_i^{\vee} \simeq P_j$. On the other hand, equation \eqref{eq:EwithFComm} and its counterpart with the roles of $E$ and $F$ reversed can be used to verify that $F_{\{1, \dots, n\}}v_+$ and $E_{\{1, \dots, n\}}v_-$ are $\Uq$-invariant vectors of $V_{(\lambda, \p 0)} \otimes V_{(\lambda, \p 0)}^{\vee}$. Moreover, $F_{\{1, \dots, n\}}v_+$ is non-zero; the coefficient of $v_{\{1, \dots, n\}} \otimes v_{\{1,\dots,n\}}^{\vee}$ in $F_{\{1, \dots, n\}}v_+$ is $1$. Similarly, $E_{\{1, \dots, n\}}v_-$ is non-zero; using equation \eqref{eq:EwithFComm}, the coefficient of $v_{\varnothing} \otimes v_{\varnothing}^{\vee}$ in $E_{\{1, \dots, n\}}v_-$ is seen to be
\[
(-1)^{\frac{n(n-1)}{2}} q^{\sum_{i=1}^n \chi_i(\lambda)} \prod_{i=1}^n [\chi_i(\lambda)]_q,
\]
which is non-zero by typicality of $\lambda$. We conclude that $F_{\{1, \dots, n\}}v_+$ and $E_{\{1, \dots, n\}}v_-$ are non-zero elements of $P_0$. It follows that $i=j$ and $P_0$ is self-dual.
\end{proof}

Assumption \ref{assump:unimodular} implies Assumption \ref{assump:noZeroRoots}. Nevertheless, we prefer to think of them as independent since Assumption \ref{assump:noZeroRoots} is necessary for relative modular constructions (see Proposition \ref{prop:typicalExist}), whereas Assumption \ref{assump:unimodular} is a technical condition which ensures unimodularity of $\cat_R$. In particular, all results which follow hold under Assumptions \ref{assump:quantAssumptDualInt} and \ref{assump:noZeroRoots} and the additional assumption that $\cat_R$ is unimodular.

\begin{Ex}
\begin{enumerate}
\item When $n=1$, Assumptions \ref{assump:noZeroRoots} and \ref{assump:unimodular} are equivalent.

\item Linear independence of the roots $Q_1, \dots, Q_n$, equivalently, injectivity of the linear map $Q: \C^n \rightarrow \C^r$, implies Assumption \ref{assump:unimodular}. This class of examples is complementary to the \emph{simple} abelian $\mathcal{N}=4$ gauge theories of \cite[\S 2.1]{ballin2023}, which in the present language correspond to those (integer) matrices $Q$ which define surjective linear maps $Q: \C^n \rightarrow \C^r$.

\item Let $r=1$ and $Q$ any root matrix which satisfies Assumptions \ref{assump:quantAssumptDualInt} and \ref{assump:noZeroRoots}. Such a theory is simple abelian. Then the root matrix $\vert Q \vert$ whose entries are the absolute values of those of $Q$ satisfies Assumptions \ref{assump:quantAssumptDualInt}, \ref{assump:noZeroRoots} and \ref{assump:unimodular}. Note that $Q$ and $\vert Q \vert$ define isomorphic symplectic representations $T^{\vee} R \simeq T^{\vee} \vert R \vert$ of $\bos$ and so isomorphic Lie superalgebras $\GW \simeq \bos_{\vert R \vert}$, as explained in Section \ref{sec:abGWClassical}. It follows that $\cat_R \simeq \cat_{\vert R \vert}$ is unimodular, even though $R$ may not satisfy Assumption \ref{assump:unimodular}. A concrete class of examples (which in fact satisfy Assumption \ref{assump:unimodular}) is given by the $n \times 1$ matrices $Q=(\begin{matrix} 1 & \cdots & 1 \end{matrix})$. This examples corresponds to $n$ copies of the $U(1)$-gauged hypermultiplet of weight $1$ or, in the language of \cite[\S 2.1.5]{ballin2023}, the B-twist of $\textnormal{SQED}_n$.

\item If no column of $Q$ is zero and each row of $Q$ has a well-defined sign, then Assumption \ref{assump:unimodular} holds. Arguing as in the previous example, it follows that any matrix $Q^{\prime}$ which differs from $Q$ by multiplying some of its columns by a sign defines a unimodular category $\cat_{R^{\prime}}$, even though $Q^{\prime}$ may not satisfy Assumption \ref{assump:unimodular}. \qedhere
\end{enumerate}
\end{Ex}

\begin{Prop}
\label{prop:mTrace}
Up to a global scalar, there exists a unique modified trace $\mt$ on the ideal of projective objects of $\cat_R$.
\end{Prop}

\begin{proof}
The category $\cat_R$ is a locally finite pivotal $\C$-linear tensor category with enough projectives. By \cite[Corollary 5.6]{geer2022}, the ideal of projectives has a non-trivial right modified trace if and only if $\cat_R$ is unimodular, in which case the right modified trace is unique up to a global scalar. Unimodularity is proved in Lemma \ref{lem:catUnimod}. By Proposition \ref{prop:braiding}, $\cat_R$ is braided. Hence, this right modified trace is a modified trace. 
\end{proof}

Let $\lambda, \lambda^{\prime} \in \bos^{\vee}$ be typical. Cyclicity of the modified trace implies that
\[
\mt_{V_{(\lambda,\p p)} }(\Phi_{V_{(\lambda^{\prime}, \p p^{\prime})},V_{(\lambda,\p p)} })=\mt_{V_{(\lambda^{\prime}, \p p^{\prime})}}(\Phi_{V_{(\lambda,\p p)} ,V_{(\lambda^{\prime}, \p p^{\prime})}}).
\]
By this and the first part of Lemma \ref{lem:Phis}, we can normalize $\mt$ so that the modified dimension of a typical Verma module is
\begin{equation}
\label{eq:modDimVerma}
\qd(V_{(\lambda, \p p)})
=
(-1)^{\p p} \prod_{i=1}^n(q^{\chi_i(\lambda)}-q^{-\chi_i(\lambda)})^{-1}.
\end{equation}
For later use, note that when $\lambda$ is typical Lemma \ref{lem:Phis} can be written as
\begin{equation}
\label{eq:openHopfSimp}
\Phi_{V_{(\lambda^{\prime}, \p p^{\prime})},V_{(\lambda,\p p)} }
=
(-1)^{\p p + \p p^{\prime} + \p n} \frac{q^{-2 \met^{\vee}(\lambda^{\prime},\lambda) + \sum_{i=1}^n \chi_i(\lambda^{\prime} + \lambda)}}{\qd(V_{(\lambda, \p p)})} \id_{V_{(\lambda, \p p)}}.
\end{equation}

\subsection{Modifications for $\Uqm$}
\label{sec:moduleCatModUnrolled}

Consider again the modified unrolled quantum group $\Uqm$ of Definition \ref{def:modUnrolledGW}.

\begin{Def}
A \emph{weight $\Uqm$-module} is a finite dimensional $\Uqm$-module $V$ on which $\bos$ acts semisimply and $K_a v = q^{2\lambda_a} v$ for all weight vectors $v$ of weight $\lambda$.
\end{Def}

Let $\mcat_R$ be the category of weight $\Uqm$-modules and their $\Uqm$-linear maps of degree $\p 0$. Note that, in view of the modified relations \eqref{eq:modQuantumRoots}, the operator equations $K_a = q^{2 Z_a}$ are required to have weight modules of dimension greater than one.

All results in Sections \ref{sec:basicDef}-\ref{sec:uniMod}, and their proofs, carry over with only minor changes in the modified setting. The changes result from the fact that the equations $K_a = q^{2 Z_a}$ imply that the action of $\mathcal{K}_i$ on a vector $v$ of weight $\lambda$ is given by
\begin{equation}
\label{eq:modUnrolledCentralAct}
\mathcal{K}_i v = q^{2\chi_i(\lambda)} v.
\end{equation}
For this reason, all appearances of $\chi_i$ in Sections \ref{sec:basicDef}-\ref{sec:uniMod} must be replaced with $2 \chi_i$. We record only the most important changes which result:
\begin{itemize}
\item A weight $\lambda \in \bos^{\vee}$ is atypical if and only if $\chi_i(\lambda) \in \frac{\pi \sqrt{-1}}{2 \hbar} \Z$ for some $1 \leq i \leq n$.
\item The exponential factor \eqref{eq:upsilonFactor} in the braiding is replaced with
\[
\Upsilon_{V,W} (v \otimes w)
=
q^{-2\met^{\vee}( \lambda_v, \lambda_w)} v \otimes w.
\]
This is required to compensate for the additional factors of $2$ resulting from the actions of $\mathcal{K}_i$ which appear in, for example, $\tilde{R}_{V,W}$.

\item In view of the previous comment, additional factors of $2$ scale both $\met^{\vee}$ and $\chi_i$ in formulae for the ribbon structure, open Hopf link invariants and modified quantum dimensions. For example, on a one dimensional module $\C^{\bos}_{(k, \p p)}$, where now $\chi_i(k) \in \frac{\pi \sqrt{-1}}{2 \hbar} \Z$ for all $1 \leq i \leq n$, the ribbon automorphism is
\[
\theta_{\C^{\bos}_{(k, \p p)}} = q^{-2\met^{\vee}(k,k) + 2\sum_{i=1}^n \chi_i(k)} \id_{\C^{\bos}_{(k, \p p)}}.
\]
\end{itemize}

\section{B-twisted abelian Gaiotto--Witten theory with compact gauge group}
\label{sec:cptBosonic}

In this section we set $\hbar= \frac{\pi \sqrt{-1}}{2}$ so that $q= e^{\frac{\pi \sqrt{-1}}{2}} = \sqrt{-1}$ and $\frac{\pi \sqrt{-1}}{2 \hbar}=1$. We work with the modified unrolled quantum group $\Uqmin$ and its category $\mcat_R$ of weight modules, as in Sections \ref{sec:modUnrolledQuant} and \ref{sec:moduleCatModUnrolled}. We work under Assumptions \ref{assump:quantAssumptDualInt}, \ref{assump:noZeroRoots} and \ref{assump:unimodular}.

\subsection{Basic set-up}
\label{sec:genConsid}

Consider the following realization of the input data from Section \ref{sec:abGWClassical}. Let $(\Gamma, \met)$ be an integral lattice of rank $r \geq 1$. Let $\bos = \Gamma \otimes_{\Z} \C$ be the associated abelian Lie algebra and $\Bos_{\Gamma} = \bos \slash \Gamma$ the resulting complex torus. For later use, let also $\Bos_{c,\Gamma} = (\Gamma \otimes_{\Z} \R)\slash \Gamma$ be the associated compact torus. The metric $\met$ induces metrics on $\bos$ and the tori $\Bos_{\Gamma}$ and $\Bos_{c,\Gamma}$, all of which are again denoted by $\met$. Let
\[
\Gamma^{\vee}: =
\Hom_{\Z}(\Gamma,\Z)
\simeq
\{ \lambda \in \bos^{\vee} \mid \lambda(\gamma) \in \Z \; \; \forall \gamma \in \Gamma \}
\]
be the dual lattice. Representations of $\bos$ whose weights lie in $\Gamma^{\vee}$ are precisely those which integrate to representations of $\Bos_{\Gamma}$.

Let $R$ be a complex representation of $\bos$ with roots $Q_1 \dots Q_n \in \bos^{\vee}$. We assume that the root lattice satisfies
\begin{enumerate}[label=(A\arabic*)]
\item \label{ite:gradingCoch} {[$\Gr$-grading]} $\Lambda_R \subset \Gamma^{\vee}$.
\end{enumerate}
In representation theoretic terms, condition \ref{ite:gradingCoch} states that $R$ integrates to a representation of $\Bos_{\Gamma}$. Moreover, condition \ref{ite:gradingCoch} implies that $\GW$ is the adjoint representation of the complex Lie supergroup $G_{\Gamma}$ with bosonic subgroup $\Bos_{\Gamma}$ and Lie superalgebra $\GW$. In terms of the category $\mcat_R$, condition \ref{ite:gradingCoch} ensures that tracking weights modulo $\Gamma^{\vee}$ defines a grading of $\mcat_R$ by the torus $\Gr = \bos^{\vee} \slash \Gamma^{\vee}$ which is Langlands dual to $\Bos_{\Gamma}$. 

Let $\FR_0$ be a subgroup of $\bos^{\vee}$ and set $\FR = \FR_0 \oplus \Ztwo$, seen as a subgroup of $\bos^{\vee} \oplus \Ztwo$. Consider the following conditions on $\FR_0$:
\begin{enumerate}[label=(A\arabic*)]
\setcounter{enumi}{1}
\item \label{ite:deg0Coch} {[$\FR$ in degree $\p 0$]} $\FR_0 \subset \Gamma^{\vee}$.

\item \label{ite:psiCoch} {[Existence of $\psi$]} $\met^{\vee}(k,\lambda) \in \Z$ for all $k \in \FR_0$ and $\lambda \in \Gamma^{\vee}$.

\item \label{ite:ribbCoch} {[Trivial ribbon]} $-\met^{\vee}(k,k) + \sum_{i=1}^n \chi_i(k) \in 2 \Z$ for all $k \in \FR_0$.

\item \label{ite:finiteCoch} {[Finiteness]} The index of $\FR_0$ in $\Gamma^{\vee}$ is finite.
\end{enumerate}
Conditions \ref{ite:gradingCoch} and \ref{ite:psiCoch} imply that $\chi_i(\gamma) = \met^{\vee}(Q_i,\gamma)$ is an integer for each $1 \leq i \leq n$. Conditions \ref{ite:deg0Coch} and \ref{ite:psiCoch} imply that $\met^{\vee}(k,k)$ is an integer for all $k \in \FR_0$. This in turn implies that $-\met^{\vee}(k,k) + \sum_{i=1}^n \chi_i(k)$ is an integer, of which condition \ref{ite:ribbCoch} is a further refinement. As we will see below, conditions \ref{ite:gradingCoch}-\ref{ite:finiteCoch} are necessary for $\FR$ to be the weights of a free realization on $\mcat_{R,\p 0}$ which is a part of a relative pre-modular structure on $\mcat_R$. The labeling of each condition indicates its relevance from this perspective.

Recall the effective metric $\emet$ introduced in Section \ref{sec:effMet}.

\begin{Prop}
\label{prop:freeRealCoch}
Let $\FR_0$ be the image of $\Gamma$ under the map $\met^{\flat}: \bos \rightarrow \bos^{\vee}$. If the lattice $\Gamma$ is even integral with respect to the effective metric $\emet$, then $\FR= \FR_0 \oplus \Ztwo$ satisfies conditions \ref{ite:deg0Coch}-\ref{ite:finiteCoch}.
\end{Prop}

\begin{proof}
We verify each of the conditions. Given $\gamma \in \Gamma$, let $\omega_{\gamma} =\met^{\flat}(\gamma)$.
\begin{enumerate}[label=(A\arabic*)]
\setcounter{enumi}{1}
\item Since the lattice $\Gamma$ is integral with respect to $\met$, the map $\met^{\flat}: \bos \rightarrow \bos^{\vee}$ restricts to a group homomorphism $\met^{\flat}: \Gamma \rightarrow \Gamma^{\vee}$. Hence, $\omega_{\gamma} \in \Gamma^{\vee}$.

\item For each $\lambda \in \Gamma^{\vee}$, the definitions of $\met^{\vee}$ and $\met^{\flat}$ give
\[
\met^{\vee}(\omega_{\gamma},\lambda)
=
\met^{\vee}(\met^{\flat}(\gamma),\lambda)
=
\lambda(\gamma).
\]
It follows that $\met^{\vee}(\omega_{\gamma},\lambda) \in \Z$.

\item Using the definition of the effective metric, we compute
\[
-\met^{\vee}(\omega_{\gamma},\omega_{\gamma}) + \sum_{i=1}^n \chi_i(\omega_{\gamma})
=
-\emet(\gamma,\gamma) + \sum_{i=1}^n \chi_i(\omega_{\gamma})(\chi_i(\omega_{\gamma})+1).
\]
Since $(\Gamma, \emet)$ is even integral and $\chi_i(\omega_{\gamma})$ is an integer, it follows that $-\met^{\vee}(\omega_{\gamma},\omega_{\gamma}) + \sum_{i=1}^n \chi_i(\omega_{\gamma}) \in 2 \Z$.

\item This follows from non-degeneracy of $\met$. Indeed, the index in question is the discriminant of $\met$. \qedhere
\end{enumerate}
\end{proof}

\subsection{Non-zero matter}
\label{sec:nonzeroHyp}

In this section we assume that $n \geq 1$, so that $R$ is not the zero representation. The case $n=0$ is treated in Section \ref{sec:toralCS}. Continue with the setting of Section \ref{sec:genConsid} so that $\Gr = \bos^{\vee} \slash \Gamma^{\vee}$ and $\FR_0 = \ima(\met^{\flat}: \Gamma \rightarrow \Gamma^{\vee})$. As in Proposition \ref{prop:freeRealCoch}, assume that $(\Gamma,\emet)$ is even integral. Keeping the notation of Section \ref{sec:genSS}, let $\SSS$ be the image of $\chi^{-1}(\SSSY)$ under the quotient map $\bos^{\vee} \rightarrow \Gr$.

\begin{Prop}
\label{prop:genericSSCpt}
The $\Gr$-graded category $\mcat_R$ is generically semisimple with small symmetric subset $\SSS$. For $\p \lambda \in \Gr \setminus \SSS$, a completely reduced dominating set of $\mcat_{R,\p \lambda}$ is
\[
\{V_{(\lambda,\p p)} \mid \lambda \mbox{ in class } \p \lambda, \; \p p \in \Ztwo \}.
\]
\end{Prop}

\begin{proof}
Conditions \ref{ite:deg0Coch}, \ref{ite:psiCoch} and \ref{ite:finiteCoch} imply that $\chi^{-1}(\SSSY) + \FR_0$ is finitely many translates of $\chi^{-1}(\SSSY)$ and so a union countably many affine hyperplanes in $\bos^{\vee}$; see the proof of Proposition \ref{prop:typicalExist}. It follows that $\chi^{-1}(\SSSY) + \FR_0$ is small in $\bos^{\vee}$ which implies that $\SSS$ is small in $\Gr$.

Since the definition of $\SSS$ is such that any weight with image in $\Gr \setminus \SSS$ is typical, the remainder of the proposition can be proved in the same way as Proposition \ref{prop:genericSSZWt}.
\end{proof}

Let $D = \coker(\met^{\flat}: \Gamma \rightarrow \Gamma^{\vee})$ be the discriminant group of the metric $\met$. The group $D$ is finite abelian of order the absolute value of the determinant of any Gram matrix of $\met$. By an abuse of notation, we sometimes identify $D$ with a subset of representatives in $\Gamma^{\vee}$ which contains $0$.

\begin{Prop}
\label{prop:relPreModCompact}
The monoidal functor
\[
\sigma: \FR \rightarrow \mcat_R,
\qquad
(\omega_{\gamma}, \p p) \mapsto \sigma_{(\omega_{\gamma},\p p)} := \C^{\bos}_{(\omega_{\gamma}, \p p)}
\]
defines a free realization of $\FR$ in $\mcat_{R,\p 0}$ and gives $\mcat_R$ the structure of a pre-modular $\Gr$-category relative to $(\FR,\SSS)$.
\end{Prop}

\begin{proof}
Recalling the computations from Example \ref{ex:oneDimMod}, we see that condition \ref{ite:deg0Coch} ensures that each $\sigma_{(\omega_{\gamma}, \p p)}$ has $\Gr$-degree $\p 0$. Conditions \ref{ite:gradingCoch} and \ref{ite:psiCoch} imply that $\chi_i(\gamma) = \met^{\vee}(Q_i,\gamma) \in \Z$, $1 \leq i \leq n$, which translates to the statement that $\sigma_{(\omega_{\gamma}, \p p)}$ is one dimensional and hence invertible. Finally, condition \ref{ite:ribbCoch} ensures that $\theta_{\sigma_{(\omega_{\gamma},\p p)}}= \id_{\sigma_{(\omega_{\gamma},\p p)}}$.

Let $\p \lambda \in \Gr \setminus \SSS$ with chosen lift $\lambda \in \bos^{\vee}$. In view of the completely reduced dominating set of Proposition \ref{prop:genericSSCpt}, we can take $\Theta(\p \lambda) = \{V_{(\lambda+k, \p 0)} \mid k \in D\}$.

Finally, for $V \in \mcat_{R, \p \lambda}$, we compute
\[
c_{\sigma_{(\omega_{\gamma},\p p)}, V} \circ c_{V,\sigma_{(\omega_{\gamma}, \p p)}}
=
q^{-4 \met^{\vee}(\lambda, \omega_{\gamma})} \id_{V \otimes \sigma_{(\omega_{\gamma}, \p p)}},
\]
where $\lambda \in \bos^{\vee}$ is any lift of $\p \lambda$. We may therefore take for the required bicharacter
\[
\psi: \Gr \times \FR \rightarrow \C^{\times},
\qquad
(\p \lambda, (\omega_{\gamma}, \p p)) \mapsto q^{-4 \met^{\vee}(\omega_{\gamma},\lambda)}.
\]
Note that independence of $\psi$ on the choice of lift follows from condition \ref{ite:psiCoch}.

In view of Propositions \ref{prop:mTrace} and \ref{prop:genericSSCpt}, this completes the proof.
\end{proof}

Turning to relative modularity, let $W \in \mcat_{R,\p 0}$. Recall that a morphism $f \in \End_{\mcat_R}(W)$ is called \emph{transparent in $\mcat_{R,\p 0}$} if 
\[
\id_U \otimes f = c_{W,U} \circ (f \otimes \id_U) \circ c_{U,W}
\]
and
\[
f \otimes \id_V =  c_{V,W} \circ (\id_V \otimes f) \circ c_{W,V}
\]
for all $U,V \in \mcat_{R, \p 0}$.

The following result is a variation of \cite[Lemma 2.3]{derenzi2020}.

\begin{Lem}
\label{lem:transparentCpt}
Let $W \in \mcat_{R,\p 0}$ and $f \in \End_{\mcat_R}(W)$ be transparent in $\mcat_{R,\p 0}$. There exist an integer $m \geq 1$, one dimensional modules $\C^{\bos}_{(k_i,\p p_i)} \in \mcat_{R,\p 0}$ and morphisms $g_i \in \Hom_{\mcat_R}(W,\C^{\bos}_{(k_i,\p p_i)})$ and $h_i \in \Hom_{\mcat_R}(\C^{\bos}_{(k_i,\p p_i)},W)$, where $1 \leq i \leq m$, such that $f = \sum_{i=1}^m h_i \circ g_i$.
\end{Lem}

\begin{proof}
Consider $V_{(0,\p 0)} \in \mcat_{R, \p 0}$ with its highest weight vector $v_{\varnothing}$ and let $w \in W$ be a weight vector. Because $v_{\varnothing}$ is annihilated by $E_I$, $I \neq \varnothing$, the explicit form of the braiding shows that $c_{V_{(0,\p 0)},W}(v_{\varnothing} \otimes w)$ is proportional to $w \otimes v_{\varnothing}$. Because $f$ is transparent, $c_{W,V_{(0,\p 0)}}(f(w) \otimes v_{\varnothing})$ is proportional to $v_{\varnothing} \otimes f(w)$. Since $\{F_I v_{\varnothing}\ \mid I \subset \{1, \dots, n\}\}$ is a basis of $V_{(0,\p 0)}$, we conclude that $E_I f(w) =0$ for all $I \neq \varnothing$.

Reversing the roles of $E_i$ and $F_i$ in Definition \ref{def:Verma}, define the Verma module of lowest weight $(\lambda,\p p) \in \bos^{\vee} \times \Ztwo$ by
\[
V^-_{(\lambda, \p p)}
:=
\Uqmin \otimes_{\UqminNn} \C^{-}_{(\lambda, \p p)}.
\]
The module $V^-_{(\lambda, \p p)}$ has a homogeneous weight basis $\{E_I v^-_{\varnothing} \mid I \subset \{1, \dots, n\} \}$, where $v^-_{\varnothing} = 1 \otimes 1$. Applying the argument of the previous paragraph with $V_{(0,\p 0)}$ and its highest weight vector $v_{\varnothing}$ replaced with $V^-_{(0, \p 0)}$ and its lowest weight vector $v^-_{\varnothing}$, we conclude that $F_I f(w) =0$ for all $I \neq \varnothing$.

It follows from the previous two paragraphs that
\[
\mathcal{K}_i - \mathcal{K}^{-1}_i = (q-q^{-1})[E_i,F_i]
\]
annihilates $f(w)$. Writing $\lambda \in \bos^{\vee}$ for the weight of $f(w)$, we conclude that $q^{4\chi_i(\lambda)} =1$, $1 \leq i \leq n$. By Example \ref{ex:oneDimMod}, each homogeneous weight vector in the image of $f$ spans a one dimensional module, necessarily in degree $\p 0 \in \Gr$ since the same is true of $W$. Since $f$ factors through its image, which is a direct sum of its weight spaces, we obtain the desired one dimensional modules and factorization of $f$.
\end{proof}

\begin{Rem}
As is evident from its proof, Lemma \ref{lem:transparentCpt} does not require the full strength of transparency: it suffices for $f$ to satisfy one-sided versions of transparency only for some (anti-)Verma modules $U$ and $V$ in degree $\p 0$.
\end{Rem}

\begin{Thm}
\label{thm:relModCpt}
The category $\mcat_R$ is a modular $\Gr$-category relative to $(\FR, \SSS)$ with stabilization coefficients
\[
\Delta_{\pm}
=
\sum_{k \in D} q^{\mp 2\met^{\vee}(k,k)} \prod_{i=1}^n \frac{q^{2\chi_i(\lambda)} - q^{-2\chi_i(\lambda)}}{q^{2\chi_i(\lambda+k)} - q^{-2\chi_i(\lambda+k)}}
\]
and relative modularity parameter $\zeta = (-1)^n \vert D \vert$. Moreover, $\mcat_R$ is TQFT finite.
\end{Thm}

\begin{proof}
The stabilization coefficients can be computed directly using Lemma \ref{lem:Phis} and the formula for the ribbon automorphism given in the proof of Theorem \ref{thm:ribbonCat}.

For relative modularity, consider Definition \ref{def:modG} with $h = \p \gamma$ and $g = \p \lambda$ and $V_i=V_{(\alpha,\p 0)},V_j= V_{(\beta,\p 0)} \in I_{\p \lambda}$. The morphism $f_{\p \gamma,(\alpha,\p 0),(\beta,\p 0)} \in \End_{\mcat_R} (V_{(\alpha,\p 0)} \otimes V_{(\beta,\p 0)}^{\vee})$ determined by the left hand side of diagram \eqref{eq:mod} is transparent in $\mcat_{R, \p 0}$; see \cite[Lemma 5.9]{costantino2014}. Lemma \ref{lem:transparentCpt} therefore ensures the existence of one dimensional modules $\C^{\bos}_{(k_i,\p p_i)} \in \mcat_{R, \p 0}$ such that
\[
f_{\p \gamma,(\alpha,\p 0),(\beta,\p 0)}
=
\sum_{i=1}^m h_{\p \gamma,(\alpha,\p 0),(\beta,\p 0),i} \circ g_{\p \gamma,(\alpha,\p 0),(\beta,\p 0),i}
\]
for some morphisms
\[
g_{\p \gamma,(\alpha,\p 0),(\beta,\p 0),i} \in \Hom_{\mcat_R}(V_{(\alpha,\p 0)} \otimes V_{(\beta,\p 0)}^{\vee},\C^{\bos}_{(k_i,\p p_i)}),
\]
\[
h_{\p \gamma,(\alpha,\p 0),(\beta,\p 0),i} \in \Hom_{\mcat_R}(\C^{\bos}_{(k_i,\p p_i)},V_{(\alpha,\p 0)} \otimes V_{(\beta,\p 0)}^{\vee}).
\]
The set $I_{\p \lambda}$ is defined so that
\[
\Hom_{\mcat_R}(V_{(\alpha,\p 0)} \otimes V^{\vee}_{(\beta,\p 0)}, \C^{\bos}_{(k_i,\p p_i)})
\simeq
\Hom_{\mcat_R}(V_{(\alpha,\p 0)}, V_{(\beta,\p 0)} \otimes \C^{\bos}_{(k_i,\p p_i)})
\]
is non-zero for a unique $\C^{\bos}_{(k_i,\p p_i)}$. Comparing weights gives $\p p_i=0$ and $\alpha = \beta + k_i$. It follows that 
\begin{equation}
\label{eq:relModCpt}
f_{\p \gamma,(\alpha,\p 0),(\beta,\p 0)}
=
c_{\alpha,\beta} \id_{\C^{\bos}_{(\alpha-\beta,\p 0)}} \otimes \tcoev_{V_{(\alpha,\p 0)}} \circ \ev_{V_{(\alpha,\p 0)}}
\end{equation}
for some scalar $c_{\alpha,\beta} \in \C$. The modified trace of the right hand side of equation \eqref{eq:relModCpt} is $q^{2\sum_{i=1}^n \chi_i(\alpha-\beta)} c_{\alpha,\beta} \qd(V_{(\alpha,\p 0)})$, which vanishes if and only if $c_{\alpha,\beta}$ does. Fix a lift of $\p \gamma$ to $\gamma \in \bos^{\vee}$. By isotopy invariance and the defining  properties of the modified trace, the modified trace of the left hand side of equation \eqref{eq:relModCpt} is
\begin{eqnarray*}
&&
\qd(V_{(\alpha,\p 0)}) \sum_{k \in D} \qd(V_{(\gamma+k,\p 0)}) \mt_{V_{(\alpha,\p 0)} \otimes V_{(\beta,\p 0)}^{\vee}}(f_{(\gamma+k,\p 0),(\alpha,\p 0),(\beta,\p 0)}) \\
&=&
\qd(V_{(\alpha,\p 0)}) \sum_{k \in D} \qd(V_{(\gamma+k,\p 0)}) \mt_{V_{(\gamma+k,\p 0)}}(\Phi_{V^{\vee}_{(\beta,\p 0)},V_{(\gamma+k,\p 0)}} \circ \Phi_{V_{(\alpha,\p 0)},V_{(\gamma+k,\p 0)}}) \\
&=&
\qd(V_{(\alpha,\p 0)}) \sum_{\beta \in D} \qd(V_{(\gamma+k,\p 0)})^2 \langle \Phi_{V^{\vee}_{(\beta,\p 0)},V_{(\gamma+k,\p 0)}} \rangle \langle \Phi_{V_{(\alpha,\p 0)},V_{(\gamma+k,\p 0)}} \rangle \\
&=&
(-1)^n q^{2\sum_{i=1}^n \chi_i(\alpha-\beta)-4 \met^{\vee}(\alpha-\beta,\gamma)} \qd(V_{(\alpha,\p 0)}) \sum_{k \in D} q^{-4 \met^{\vee}(\alpha-\beta,k)},
\end{eqnarray*}
the final equality following from equation \eqref{eq:openHopfSimp}. The canonical isomorphism $\Gamma \simeq \Gamma^{\vee \vee}$ implies that
\[
q^{-4\met^{\vee}(\alpha-\beta,-)} = e^{-2\pi \sqrt{-1}\met^{\vee}(\alpha-\beta,-)} : D \rightarrow \C^{\times}
\]
is the trivial character if and only if $\alpha-\beta \in \ima \met^{\flat}$. Note however that the definition of $I_{\p \lambda}$ implies that $\alpha-\beta \in \ima \met^{\flat}$ if and only if $\alpha=\beta$. Fourier theory for the finite abelian group $D$ then gives
\[
\mt_{V_{(\alpha,\p 0)} \otimes V_{(\beta,\p 0)}^{\vee}} (f_{\p \gamma,(\alpha,\p 0),(\beta,\p 0)})
=
\begin{cases}
(-1)^n \qd(V_{(\alpha,\p 0)}) \vert D \vert & \mbox{if }  \alpha=\beta, \\
0 & \mbox{otherwise}.
\end{cases}
\]
It follows that $c_{\alpha,\beta}=0$ unless $\alpha=\beta$ and we conclude that $\zeta = (-1)^n \vert D \vert$ is a relative modularity parameter.

Turning to TQFT finiteness (see Definition \ref{def:relModFinite}), note that $\mcat_R$ is locally finite abelian with enough injectives and projectives. In particular, $\mcat_R$ is Krull--Schmidt and its isomorphism classes of projective indecomposable objects are in bijection with its isomorphism classes of simple objects \cite[Lemma 3.5]{krause2015}. Since the latter are in bijection with $\bos^{\vee} \times \Ztwo$ by Theorem \ref{thm:simplesSplitAb}, property \ref{ite:finCat1} holds by condition \ref{ite:finiteCoch}. Property \ref{ite:finCat2} follows from the fact that objects of $\mcat_R$ are finite dimensional weight modules and property \ref{ite:finCat3} follows from the Krull--Schmidt property.
\end{proof}

Denote the relative modular category of Theorem \ref{thm:relModCpt} by $\mcat_{R,\Gamma}$ and the TQFT resulting from Theorem \ref{thm:relModTQFT} by $\TQFT_{\mcat_{R,\Gamma}}$. In the remainder of this section, we argue that $\TQFT_{\mcat_{R,\Gamma}}$ models the homological truncation of Chern--Simons-Rozansky--Witten theory of the Hamiltonian $\Bos_{c,\Gamma}$-manifold $T^{\vee} R$ at level $\met$ or, equivalently, the truncation of Chern--Simons theory with gauge supergroup $G_{c,\Gamma}$ at level $\met$.

\subsubsection{Topological flavour symmetry}
\label{sec:magneticsymmetry}
The grading group $\Gr$ of the category $\mcat_{R,\Gamma}$ agrees with the topological flavour symmetry group (also known as the magnetic flavour symmetry group) of the $\Bos_{c,\Gamma}$-gauge theory underlying the Chern--Simons-Rozansky--Witten theory of $T^{\vee} R$. Following the general discussion of such symmetries in \cite[\S 4]{gaiotto2014}, this can be seen as follows. The current associated to the topological flavour symmetry is the curvature $F$ of the $\Bos_{c,\Gamma}$-connection $A$; it is conserved due the Bianchi identity. The infinitesimal action of this $0$-form symmetry is implemented by surface integrals of $\frac{1}{2 \pi} F$. For example, its action on a local operator $\mathcal{O}(x)$ is realized by integrating $\frac{1}{2 \pi} F$ over a $2$-sphere $S^2$ linking the point $x$ and hence measures the Chern number of the gauge bundle sourced by $\mathcal{O}(x)$. In other words, operators charged under this symmetry are magnetic monopoles. The weights for the topological flavour symmetry group, that is, the allowed Chern numbers of gauge bundles on $S^2$, are identified with the cocharacter lattice $\Gamma \simeq \pi_1(\Bos_{c,\Gamma})$ of $\Bos_{c,\Gamma}$ and hence the topological flavour symmetry group is identified with Pontryagin dual $\hat{\Gamma} \simeq \bos^{\vee}/\Gamma^{\vee} \simeq \Gr$, as claimed
	
	Finite symmetry transformations are implemented by the topological surface operators given by exponentiated integrals of the curvature,
	\[ \mathcal{U}_{\p g}[\Sigma] = \exp\bigg(\sqrt{-1} \xi \int_\Sigma F\bigg)\,,\]
	where $\Sigma$ is a surface in $3$-dimensional spacetime and $\p g = \xi + \Gamma^{\vee} \in \Gr$. Importantly, if $\Sigma$ has non-empty boundary, then the surface operator is bounded by an improperly-quantized Wilson line \cite{alford1992,gaiotto2014}; schematically, this is due to the relation
	\[ \mathcal{U}_{\p g}[\Sigma] = \exp\bigg(\sqrt{-1} \xi \int_\Sigma F\bigg) = \exp\bigg(\sqrt{-1} \xi \oint_{\partial \Sigma} A\bigg)\,.\]
	The way to interpret this formula is that a Wilson line labelled by the infinitesimal weight $\xi$, that is, an object of the subcategory $\mcat_{R,\Gamma,\xi}$, is not gauge invariant on its own unless $\xi \in \Gamma^{\vee}$, that is, $\xi$ is a weight of $\Bos_{c,\Gamma}$. If $\xi \notin \Gamma^{\vee}$, then the Wilson line must be attached to the corresponding magnetic surface operator $\mathcal{U}_{\p g}$, whose support $\Sigma$ is extra data necessary to define the Wilson line. The attached surface $\Sigma$ can be thought of as the worldsheet of the electric analogue of a Dirac string \cite{alford1992}, which is visible to magnetically charged local operators unless the Wilson line is properly quantized.
	
	Line operators that are well-defined without the additional data of a topological surface defect are called \emph{genuine} while those that require this additional data are \emph{non-genuine} \cite[\S 2]{kapustin2014}. A consistent network of such topological surface defects precisely encodes (the holonomies of) a flat $\Gr$-connection on the complement of the support of these line operators or, equivalently, the Poincar\'{e} dual of this network realizes a $\Gr$-valued cohomology class $\coh$. This is the physical origin of the cohomology classes which decorate objects and morphisms of $\CobAd_{\mcat_{R,\Gamma}}$.

\subsubsection{Screening by gauge vortex lines}
To describe the full, non-perturbative, theory the effects of screening by gauge vortex lines must be included. Namely, the physically inequivalent line operators in the non-perturbative theory are identified with $\FR$-orbits of objects in $\mcat_{R,\Gamma}$. In particular, simple line operators are labelled by their weights modulo screening, that is, elements of $\bos^{\vee} / \FR_0$. We note that all of these simple Wilson lines introduce a monodromy in the bosonic gauge fields, a phenomenon familiar to compact Chern--Simons theories \cite[\S 3.3]{witten1989}. The monodromy induced by simple line operators is controlled by the map $\bos^{\vee}/\FR_0 \to \Bos_{\Gamma}$ induced by $\met^\sharp: \bos^{\vee} \to \bos$. The simple line operators in degree $\p 0$, that is, those line operators which source a flat connection for the topological flavour symmetry with trivial holonomy, are thus genuine and labelled by elements of the discriminant group $D \simeq \Gamma^{\vee}/\FR_0$, as expected.

\subsubsection{Orientations and spin structures}
\label{sec:spinExpect}
The boundary vertex operator algebras which generalize those of \cite{garner2023b} to the present setting have boundary monopole operators with half-integral conformal weights/spins if and only if the lattice $(\Gamma, \emet)$ is not even. In particular, if $(\Gamma, \emet)$ is even, then the physical theory is defined on oriented $3$-manifolds; other wise a spin structures is required. Our evenness assumption in Proposition \ref{prop:freeRealCoch} therefore matches physics predictions. When the lattice $(\Gamma,\emet)$ is odd, we expect a spin analogue of Theorem \ref{thm:relModCpt} and the resulting TQFT $\TQFT_{\mcat_{R,\Gamma}}$, along the lines of \cite{blanchet2014}.

\subsubsection{Verlinde formula}
\label{sec:verlindeCpt}

Let $\D= \sqrt{(-1)^n\vert D \vert}$. We compute the value of $\TQFT=\TQFT_{\mcat_{R,\Gamma}}$ on trivial circle bundles over closed surfaces with insertions and relate the result to Euler characteristics of state spaces of $\TQFT$, yielding a Verlinde formula for $\TQFT$.

Let $\CS=(\Sigma,\{p_1,\ldots, p_m\},\coh, \mathcal{L})$ be a decorated connected closed surface of genus $g$. Since the Lagrangian subspace $\mathcal{L}$ plays no role in what follows, we henceforth ignore it. For each $\p \beta \in \Gr$, consider the decorated closed $3$-manifold
\[
\CS \times S^1_{\p \beta}=(\Sigma \times S^1,T=\{p_1, \dots, p_m\}\times S^1,\coh \oplus \p \beta),
\]
where we use the isomorphism
\[
H^1((\Sigma \times S^1) \setminus T; \Gr) \simeq H^1(\Sigma \setminus \{p_i\}; \Gr) \oplus \Gr
\]
to extend $\coh$ to $\omega \oplus \p \beta$. Assuming that all colours of $\CS \times S^1_{\p \beta}$ have degree in $\Gr \setminus \SSS$, the partition function $\TQFT(\CS \times S^1_{\p \beta})$ can be computed by the following surgery presentation:
\begin{equation*}
\label{diag:surgPres}
\epsh{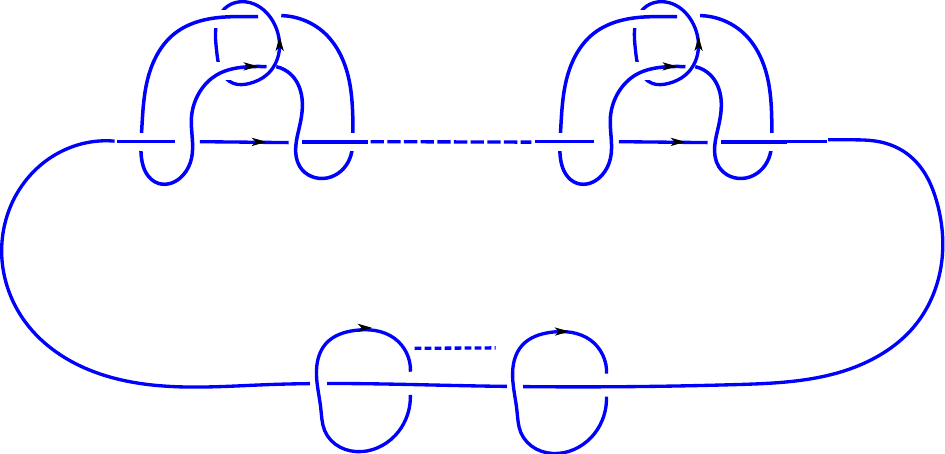}{30ex}
\put(-230,64){\ms{\Omega_{\p{\beta}_1}}}\put(-300,40){\ms{\Omega_{\p{\alpha}_1}}}\put(-83,64){\ms{\Omega_{\p{\beta}_g}}}
\put(-155,40){\ms{\Omega_{\p{\alpha}_g}}}\put(0,-29){{\ms{\Omega_{\p{\beta}}}}}\put(-205,-25){\ms{V_1}}
\put(-140,-25){\ms{V_m}}
\;\;.
\end{equation*}
Here $V_i=V_{(\mu_i,\p q_i)}$ is the colour of the point $p_i$ and $\p{\alpha}_j = \coh(a_j), \p{\beta}_j=\coh(b_j)$ for a symplectic basis $\{a_j, b_j \mid 1 \leq j \leq g\}$ of $H_1(\Sigma;\Z)$. By applying equation \eqref{eq:mod} first to the $\Omega_{\p{\alpha}_j}$-coloured strand and then to the $\Omega_{\p{\beta}_j}$-coloured strand, we simplify the surgery presentation using the equalities
\[
\epsh{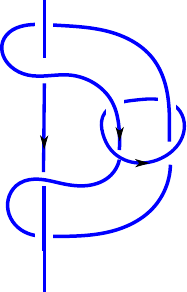}{20ex}
\put(-13,40){\ms{\Omega_{\p{\alpha}_j}}}\put(-22,-12){\ms{\Omega_{\p{\beta}_j}}}\put(-80,7){\ms{V_{(\beta+k,\p 0)}}}
=-\frac{\zeta}{\qd(V_{(\beta+k,\p 0)}^{\vee})}\epsh{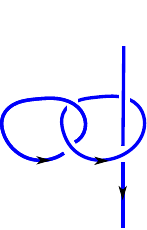}{20ex}
\put(-1,-12){\ms{\Omega_{\p{\beta}_j}}}\put(-5,-31){\ms{V_{(\beta+k,\p 0)}}}\put(-60,-25){\ms{V_{(\beta+k,\p 0)}}}
= \frac{\zeta^2}{\qd(V_{(\beta+k,\p 0)}^{\vee})^2}\id_{V_{(\beta+k,\p 0)}},
\]
where $V_{(\beta+k,\p 0)}$ is any colour appearing in $\Omega_{\p \beta}$. Applying this simplification for $j=1, \dots, g$ reduces the surgery presentation to an $\Omega_{\p \beta}$-coloured unknot encircled by $m$ pairwise unlinked unknots coloured by $V_1, \dots, V_m$, whose associated scalar we denote by $\langle \Phi_{\{V_i\},\Omega_{\p \beta}} \rangle$. Evaluating the simplified diagram gives
\[
\TQFT(\CS \times S^1_{\p \beta})
=
\D^{-2g-2} \zeta^{2g} \sum_{k \in D} \qd(V_{(\beta+k,\p 0)})^{2-2g} \langle \Phi_{\{V_i\},\Omega_{\p \beta}} \rangle.
\]
Setting $\mu = \sum_{i=1}^m \mu_i$ and $\p q = \sum_{i=1}^m \p q_i$, we use Lemma \ref{lem:Phis} to compute
\begin{multline}
\label{eq:verlindeCpt}
\TQFT(\CS \times S^1_{\p \beta})
=
(-1)^{(g+1+m)n + \p q} \vert D \vert^{g-1} q^{-4\met^{\vee}(\beta, \mu) +2m \sum_{i=1}^n \chi_i(\beta) + 2\sum_{i=1}^n \chi_i(\mu)} \cdot \\
\sum_{k \in D} \prod_{i=1}^n (q^{2\chi_i(\beta+k)}-q^{-2\chi_i(\beta+k)})^{2g-2+m}.
\end{multline}
Since $\TQFT(\CS \times S^1_{\p \beta})$ is holomorphic in $\coh$, equation \eqref{eq:verlindeCpt} holds whenever $\CS$ is admissible.

Let $s$ be a variable and $t = (t_1, \dots, t_r)$ a multivariable. For $\lambda=(\lambda_1, \dots, \lambda_r) \in \bos^{\vee}$, introduce the notation $t^{\lambda} = t_1^{\lambda_1} \cdots t_r^{\lambda_r}$. Given $\beta \in \bos^{\vee}$, we take specializing $t$ to $q^{-4 \met^{\vee}(-,\beta)}$ to mean setting each $t^{\lambda}$, $\lambda \in \bos^{\vee}$, equal to $q^{-4 \met^{\vee}(\lambda,\beta)}$. Write $\TQFT_{(k,\p p)}(\CS)$ for the homogeneous summand of $\TQFT(\CS)$ of degree $(k,\p p) \in \FR$ and define the generating function of graded dimensions
\[
\dim_{(t,s)} \TQFT(\CS) = \sum_{(k, \p p) \in \FR} (-1)^{\p p} \dim_{\C} \TQFT_{(k,\p p)}(\CS) t^k s^{\p p}.
\]

\begin{Thm}[Verlinde formula]
\label{thm:verlindeCpt}
There is an equality
\[
\TQFT(\CS \times S^1_{\p \beta}) = \dim_{(q^{-4 \met^{\vee}(-,\beta)},1)} \TQFT(\CS).
\]
\end{Thm}

\begin{proof}
Since this can be proved in the same way as \cite[Theorem 5.9]{blanchet2016} and \cite[Theorem 3.5]{geerYoung2022}, we omit the details.
\end{proof}

More explicitly, the right hand side of the equality of Theorem \ref{thm:verlindeCpt} reads
\[
\sum_{\gamma \in \Gamma} \chi( \TQFT_{(\omega_{\gamma},\bullet)}(\CS)) q^{-4 \met^{\vee}(\omega_{\gamma},\beta)},
\]
where $\chi( \TQFT_{(\omega_{\gamma},\bullet)}(\CS))$ denotes the Euler characteristic of the $\Ztwo$-graded subspace of $\TQFT(\CS)$ of $\FR_0$-degree $\omega_{\gamma}$.

\begin{Ex}
\label{ex:VerlindeCpt}
Setting $m=0$ in equation \eqref{eq:verlindeCpt} gives
\[
\TQFT(\CS \times S^1_{\p \beta})
=
(-1)^{(g+1)n} \vert D \vert^{g-1} \sum_{k \in D} \prod_{i=1}^n (q^{2\chi_i(\beta+k)}-q^{-2\chi_i(\beta+k)})^{2g-2}.
\]
Applying Theorem \ref{thm:verlindeCpt}, we conclude that
\[
\chi(\TQFT(\CS))
:=
\sum_{\gamma \in \Gamma} \chi( \TQFT_{(\omega_{\gamma},\bullet)}(\CS))
=
\lim_{\beta \rightarrow 0} \TQFT(\CS \times S^1_{\p \beta}) \\
=
\vert D \vert^{g-1} \sum_{k \in D} \prod_{i=1}^n\Big( 2 \sin \big(\chi_i(k) \pi \big) \Big)^{2g-2}.
\]
In particular, $\dim_{\C} \TQFT(\CS) = \chi(\TQFT(\CS)) = \vert D \vert$ when $g=1$.
\end{Ex}

\begin{Rem}
\label{rem:Bethe}
The Euler characteristic $\chi(\TQFT(\CS))$ from Example \ref{ex:VerlindeCpt} can be reproduced physically by way of supersymmetric localization. See \cite{closset2019} and references therein for a detailed overview of localization techniques and \cite[\S 5]{creutzig2021} for an implementation in a related context. In brief, this Euler characteristic can be realized as a sum over solutions $\{y^*\}$ to the Bethe vacuum equations of the form
\[
	\chi(\TQFT(\CS)) = \sum_{\{y^*\}} \mathcal{H}(y^*)^{g-1},
\]
where $\mathcal{H}$ denotes the handle-gluing operator. In the present context, the Bethe vacuum equations take the form
\[
	\prod_{b=1}^r y_b^{\met_{ab}} = (-1)^{\sum_{i=1}^n Q_{ai}},
	\qquad
	1 \leq a \leq r
\]
and the handle-gluing operator is
\[
\mathcal{H}(y) = \vert D \vert \prod_{i=1}^{n}\bigg(\big(1- \prod_{a=1}^r y_a^{Q_{ai}}\big)\big(1- \prod_{a=1}^r y_a^{-Q_{ai}}\big)\bigg),
\]
from which the above formula for $\chi(\TQFT(\CS))$ is a straightforward computation.
\end{Rem}

\subsection{Compact global forms of $\gloo$-Chern--Simons theory}
\label{sec:uooCS}

Consider the special case of Section \ref{sec:nonzeroHyp} corresponding to Example \ref{ex:gloo}, so that $r=2$ and $n=1$ with weight $Q= \left(\begin{smallmatrix} 1 \\ 0 \end{smallmatrix} \right)$ and metric $\met= \left(\begin{smallmatrix} 0 & 1 \\ 1 & 0 \end{smallmatrix}\right)$. The effective metric is $\emet=\left(\begin{smallmatrix} 1 & 1 \\ 1 & 0 \end{smallmatrix}\right)$. We identify $\GW$ with $\gloo$, so that $\Lambda_R = \Z \cdot N^{\vee}$ and $\chi_1(\lambda)=-\lambda_E$.

Let $\Gamma \subset \bos$ be a full rank lattice which is integral with respect to $\met$ and satisfies condition \ref{ite:gradingCoch}. As argued in \cite[\S 4.2]{mikhaylov2015}, there exist unique integers $s,t \in \Z_{> 0}$ and half integer $u \in \frac{1}{2} \Z \slash s \Z$ such that $\Gamma = \Z \cdot \gamma_1 \oplus \Z \cdot \gamma_2$, where
\[
\gamma_1 = \frac{s}{t} E,
\qquad
\gamma_2 = t N + \frac{u}{t} E.
\]
We furthermore assume that $(\Gamma,\emet)$ is even, which translates to the condition that $t^2+2u$ is an even integer. In the basis $\{\gamma_1,\gamma_2\}$, the Gram matrix of $\met$ is $B=\left(\begin{smallmatrix} 0 & s \\ s & 2u \end{smallmatrix} \right)$, whence $\vert D \vert = s^2$. 
The discriminant group is
\[
D \simeq \Z \slash d_1 \Z \oplus \Z \slash d_2 \Z,
\]
where $d_1 = \gcd(s,2u)$ and $d_2 = \frac{s^2}{d_1}$. We compute
\[
\gamma_1^{\vee} = - \frac{u}{ts} N^{\vee} + \frac{t}{s}E^{\vee},
\qquad
\gamma_2^{\vee} = \frac{1}{t}N^{\vee}
\]
and
\[
\omega_{\gamma_1}=\frac{s}{t} N^{\vee},
\qquad
\omega_{\gamma_2} = \frac{u}{t}N^{\vee} + t E^{\vee}.
\]

Denote by $\mcat_{R,\Gamma}$ the relative modular category constructed in Theorem \ref{thm:relModCpt}. The TQFT $\TQFT_{\mcat_{R,\Gamma}}$ models Chern--Simons theory with gauge supergroup $G_{c,\Gamma}$ at level $\met$, that is, a global form of $\uoo$-Chern--Simons theory. Let $Y \in GL(2,\Z)$ be a right Smith multiplier for the Gram matrix $B$. Using the material recalled in Appendix \ref{sec:lattices} with the dual basis $\{\gamma_1^{\vee},\gamma_2^{\vee}\}$, the result of Example \ref{ex:VerlindeCpt} can be written as
\[
\chi(\TQFT_{\mcat_{R,\Gamma}}(\CS))
=
s^{2g-2} \sum_{i=0}^{d_1-1} \sum_{j=0}^{d_2-1} \Big( 2 \sin \big((Y_{11} i + Y_{21} j)\frac{t}{s} \pi \big) \Big)^{2g-2}.
\]

\begin{Ex}
Suppose that $(s,t,u)=(1,t,u)$ with $t^2+2u \in 2 \Z$. Then $D$ is the trivial group and
\[
\chi(\TQFT(\CS))
=
\begin{cases}
0 & \mbox{if } g \geq 2, \\
1 &\mbox{if } g =1.
\end{cases} \qedhere
\]
\end{Ex}

\begin{Ex}
\label{ex:diagonalSuper}
Let $(s,t,u) = (s,1,\frac{s}{2})$ with $s$ odd, so that $d_1 = d_2=s$ and $D \simeq \Z \slash s \Z \oplus \Z \slash s \Z$. A right Smith multiplier for $B=\left(\begin{smallmatrix} 0 & s \\ s & s\end{smallmatrix} \right)$ is $Y = \left(\begin{smallmatrix} 1 & 0 \\ -1 & 1\end{smallmatrix} \right)$ so that
\[
\chi(\TQFT(\CS))
=
s^{2g-2} \sum_{i,j=0}^{s-1} \Big( 2 \sin \big(\frac{(i -j)\pi}{s} \big) \Big)^{2g-2}
=
s^{2g-1} \sum_{k=0}^{s-1} \Big( 2 \sin \big(\frac{k\pi}{s} \big) \Big)^{2g-2}
=
s^{2g} {2g-2 \choose g-1},
\]
the final equality following from \cite[Eqn. (4.4.2.1)]{prudnikov1986}. In physical terms, this example is Chern--Simons theory of a pair of $U(1)$ gauge fields $A^{\pm}$ at levels $\pm s$ coupled to a B-twisted hypermultiplet $Z^{\pm}$ of $(T_+,T_-)$-weights $\pm (1,-1)$. When the underlying $U(1)^2$-bundle is trivial, the action for a closed $3$-manifold $M$ reads
\[
S(M) = \int_M \left( \frac{s}{4 \pi}(A^+ \wedge dA^+ - A^- \wedge dA^-) + Z^- \wedge d_A Z^+ \right). \qedhere
\]
\end{Ex}

\begin{Ex}
Take $(s,t,u)=(s,1,\frac{s-1}{2})$ with $s$ even, so that $d_1 = 1$ and $d_2 =s^2$ and $D \simeq \Z \slash s^2 \Z$. A right Smith multiplier for $B=\left(\begin{smallmatrix} 0 & s \\ s & s-1 \end{smallmatrix} \right)$ is $Y = \left(\begin{smallmatrix} s & s-1 \\ -s+1 & -s+2\end{smallmatrix} \right)$ so that
\[
\chi(\TQFT(\CS))
=
s^{2g-2} \sum_{j=0}^{s^2-1} \Big( 2 \sin \big(\frac{(-s+1)j \pi}{s} \big) \Big)^{2g-2}
=
s^{2g-1} \sum_{j=0}^{s-1} \Big( 2 \sin \big(\frac{j\pi}{s} \big) \Big)^{2g-2}
=
s^{2g} {2g-2 \choose g-1}. \qedhere
\]
\end{Ex}

\begin{Rem}
\label{rem:hyperbolicSuper}
Another example of physical interest is $(s,t,u) = (s,1,0)$, which  corresponds to the Chern--Simons theory of a hyperbolic pair of $U(1)$ gauge fields $A^N$, $A^E$ at level $s$ coupled to a B-twisted hypermultiplet $Z^{\pm}$ of $(N,E)$-weights $\pm (1,0)$. The action is
\[
S(M) = \int_M  \left( \frac{s}{4 \pi}(A^N \wedge dA^E + A^E \wedge dA^N) + Z^- \wedge d_{A^N \oplus A^E} Z^+ \right).
\]
However, this example does not satisfy our assumptions since $\emet$ is not even integral: $t^2+2u=1$ is odd. Instead, we expect this example to define a spin TQFT.
\end{Rem}

\subsection{Zero matter and toral Chern--Simons theory}
\label{sec:toralCS}

Consider the degenerate setting of Section \ref{sec:genConsid} in which $R$ is the zero representation, that is, there is no matter. In this case $\GW=\bos$ and there are no effective corrections, that is, $\met=\emet$.

Since $\UqTor$ is concentrated in degree $\p 0$, we consider its category $\mcat$ of finite dimensional ungraded weight $\UqTor$-modules. The category $\mcat$ is semisimple with isomorphism classes of simple objects given by the one dimensional modules $V_{\lambda} = \C^{\bos}_{\lambda}$ of weight $\lambda \in \bos^{\vee}$ (with no condition on $\lambda$). As in Section \ref{sec:nonzeroHyp}, grade $\mcat$ by the dual torus $\Gr = \bos^{\vee} \slash \Gamma^{\vee}$ and set $\FR = \ima(\met^{\flat}: \Gamma \rightarrow \Gamma^{\vee})$.

\begin{Thm}
\label{thm:relModToralCS}
The category $\mcat$ is a relative modular $\Gr$-category with respect to $(\FR, \SSS=\varnothing)$ with with stabilization coefficients
\[
\Delta_{\pm} = \sum_{k \in D} q^{\mp 2\met^{\vee}(k,k)}
\]
and relative modularity parameter $\zeta = \vert D \vert$. Moreover, $\mcat$ is TQFT finite.
\end{Thm}

\begin{proof}
Since $\mcat$ is semisimple, we can take for the modified trace the standard trace. Proposition \ref{prop:freeRealCoch} ensures that conditions \ref{ite:deg0Coch}-\ref{ite:finiteCoch} hold. We can therefore argue as in Section \ref{sec:nonzeroHyp} to prove that $\mcat$ is relative modular with the stated numerical invariants.
\end{proof}

\begin{Rem}
There is an analogue of Theorem \ref{thm:relModToralCS} for the category $\cat$ of ungraded weight modules over the $U_{-1}^{\bos}(\bos)$; see Remark \ref{rem:weakenAssump}. The results which follow can also be proved in this context.
\end{Rem}

Let $\TQFT_{\mcat_{\Gamma}}$ be the $\FR$-graded TQFT associated to the relative modular structure of Theorem \ref{thm:relModToralCS}. Since $\SSS = \varnothing$, all decorated bordisms are admissible, that is $\CobAd_{\mcat}=\Cob_{\mcat}$. In particular, $\TQFT_{\mcat_{\Gamma}}$ produces invariants of arbitrary decorated closed $3$-manifolds.

Working in the same topological setting as Section \ref{sec:verlindeCpt}, we find for the partition of trivial circle bundles with insertions
\begin{equation}
\label{eq:toralPartitionCircleBund}
\TQFT_{\mcat_{\Gamma}}(\CS \times S^1_{\p \beta})
=
\vert D \vert^{g-1}\sum_{k \in D} q^{-2\met^{\vee}(\beta+k, \mu)}.
\end{equation}

\begin{Prop}
\label{prop:stateSpaceToralCS}
Let $\CS$ be a decorated surface of genus $g \geq 1$ with no marked points. Then $\TQFT_{\mcat_{\Gamma}}(\CS)$ is concentrated in degree $0 \in \FR$, where it has dimension $\vert D \vert^g$.
\end{Prop}

\begin{proof}
Specializing equation \eqref{eq:toralPartitionCircleBund} to $m=0$ gives $\TQFT_{\mcat_{\Gamma}}(\CS \times S^1_{\p \beta}) = \vert D \vert^g$. Taking the limit $\beta \rightarrow 0$ then gives $\chi(\TQFT_{\mcat_{\Gamma}}(\CS)) = \vert D \vert^g$. It remains to prove that $\TQFT_{\mcat_{\Gamma}}(\CS)$ is concentrated in degree $\p 0 \in \FR$. A direct modification of the proof of \cite[Theorem 3.3]{geerYoung2022} applies in the present setting to show that the homogeneous summand $\TQFT_{\mcat_{\Gamma},-k}(\CS)$ of degree $-k \in \FR$ is spanned by colourings of degree $k$ of a particular oriented trivalent spine, denoted in \emph{loc. cit.} by $\tilde{\Gamma}$, of a handlebody bounded by $\CS$. See \cite[Definition 3.2]{geerYoung2022}. Degree $-k$ colourings are determined by solving a system of recursive equations, with each equation governing a balancing condition associated to a node of $\tilde{\Gamma}$, that is, the condition that the colour of the incoming edge appears as a summand of the tensor product of the colours of the outgoing edges. Since the category $\mcat$ is pointed, the colour of any two edges incident to a node uniquely determines the third. From this one can verify directly that this system of equations admits a solution only if $k=\p 0$.
\end{proof}

The theory $\TQFT_{\mcat_{\Gamma}}$ models $\Bos_{c,\Gamma}$-Chern--Simons theory at level $\met$. Note that the latter theory is semisimple, so that no homological truncation is needed. More precisely, whereas standard mathematical models of toral Chern--Simons theory incorporate only genuine Wilson line operators, $\TQFT_{\mcat_{\Gamma}}$ incorporates also non-genuine line operators, as discussed in Section \ref{sec:magneticsymmetry}.

Recall that the standard Reshetikhin--Turaev approach to $\Bos_{c,\Gamma}$-Chern--Simons theory at level $\met$ is via the modular tensor category $\vect_{\C}(D,q)$ whose underlying abelian category is finite dimensional $D$-graded vector spaces and whose associator and braiding are determined by the quadratic form
\[
q : D \rightarrow \Q \slash \Z,
\qquad
d \mapsto \met^{\vee}(d,d),
\]
interpreted as an element of the abelian cohomology $H^3_{\textnormal{ab}}(D; \Q \slash \Z)$ \cite{joyal1993,stirling2008,kapustin2011}. Physically, $\vect_{\C}(D,q)$ is the category of genuine Wilson line operators. For other approaches to toral Chern--Simons theory, again incorporating only genuine line operators, see \cite{dijkgraaf1990,manoliu1998,freed2010}. For a generalization to spin TQFTs when the lattice $(\Gamma,\met)$ is odd, see \cite{belov2005}.

We explain how to recover $\vect_{\C}(D,q)$ from the relative modular category $\mcat_{\Gamma}$ and hence the Reshetikhin--Turaev model from $\TQFT_{\mcat_{\Gamma}}$. First, note that the ribbon subcategory $\mcat_{\Gamma,\p 0}$ is naturally identified with the category of finite dimensional representations of $\Bos_{c,\Gamma}$ and so as labels of genuine Wilson line operators in classical $\Bos_{c,\Gamma}$-Chern--Simons theory. Since the group $\FR$ acts freely on isomorphism classes of objects of $\mcat_{\Gamma}$, we can form the orbit category $\mcat_{\Gamma,\p 0} \slash \FR$. Explicitly, $\mcat_{\Gamma,\p 0} \slash \FR$ has the same objects as $\mcat_{\Gamma,\p 0}$ and morphisms
\[
\Hom_{\mcat_{\Gamma,\p 0} \slash \FR}(V,W) = \bigoplus_{k \in \FR} \Hom_{\mcat_{\Gamma,\p 0}}(\sigma_k \otimes V,W).
\]
Isomorphism classes of simple objects of $\mcat_{\Gamma,\p 0} \slash \FR$ are then in bijection with $\Gamma^{\vee} \slash \FR \simeq D$. It follows that the category of $D$-graded vector spaces is a skeleton of $\mcat_{\Gamma,\p 0} \slash \FR$. Transferring the ribbon structure from $\mcat_{\Gamma,\p 0}$ to this skeleton recovers the modular tensor category $\vect_{\C}(D,q)$. Passing to the orbit category $\mcat_{\Gamma,\p 0} \slash \FR$, and so identifying classical line operators labeled representations in the same $\FR$-orbit, is a non-perturbative quantum effect implemented by monopole operators \cite[\S 3.1]{kapustin2011}. 

A standard model for the state space of a genus $g$ surface $\Sigma_g$ in $\Bos_{c,\Gamma}$-Chern--Simons theory at level $\met$, obtained by geometric quantization, is the space of holomorphic sections of a $\Theta$-line bundle over the moduli space of flat $\Bos_{c,\Gamma}$-bundles on $\Sigma_g$. This space of sections is of dimension $\vert D \vert^g$, in agreement with Proposition \ref{prop:stateSpaceToralCS}.

\section{Additional relative modular structures}
\label{sec:absRelMod}

We continue to set $\hbar= \frac{\pi \sqrt{-1}}{2}$ and work with the category $\mcat_R$. We work under Assumptions \ref{assump:quantAssumptDualInt}, \ref{assump:noZeroRoots} and \ref{assump:unimodular}. In this section, we consider various modifications of Section \ref{sec:genConsid}. Since many of the proofs are similar to those of Section \ref{sec:cptBosonic}, we will at points be brief.

\subsection{Abelian gauged $\mathcal{N}=4$ hypermultiplets}
\label{sec:abelianHypers}

Consider again Example \ref{ex:gaugedAbHypers}. Fix a basis $\bosSm \simeq \bigoplus_{a=1}^s \C \cdot N_a$ and let $\Gamma=\bigoplus_{a=1}^s \Z \cdot N_a$. Assumption \ref{assump:quantAssumptDualInt} holds if and only if the matrix $Q^{(\bosSm)}$ has integer entries. Assumptions \ref{assump:noZeroRoots} and \ref{assump:unimodular} hold if and only if they hold with $Q$ replaced by $Q^{(\bosSm)}$. The effective metric on $\bos = \bosSm \oplus \bosSm^{\vee}$ is
\[
\emet = \left( \begin{matrix} c_2(R) & \id_{\bosSm^{\vee}} \\ \can_{\bosSm} & 0 \end{matrix} \right),
\]
where $c_2(R) \in \End_{\C}(\bosSm)$ is the quadratic Casimir of $R$; with respect to the above basis of $\bosSm$, its entries are $c_2(V)_{ab} = \sum_{i=1}^n Q^{(\bosSm)}_{ai} Q^{(\bosSm)}_{bi}$.

Since $Q^{(\bosSm)}$ has integer entries, $R$ lifts to a representation of the complex torus $T^{(\bosSm)}_{\Gamma}=(\Gamma \otimes_{\Z} \C)\slash \Gamma$. The above data determines a theory of $T^{(\bosSm)}_{c,\Gamma}$-gauged $\mathcal{N}=4$ hypermultiplets with charge matrix the transpose of $Q^{(\bosSm)}$ \cite[\S 2]{ballin2023} or, equivalently, Chern--Simons-Rozansky--Witten theory with non-compact gauge group $T^{(\bosSm)}_{c,\Gamma} \times \bosSm^{\vee}$ at level $\met$ and holomorphic symplectic target $T^{\vee} R$. A vertex operator algebra whose (not necessarily local) modules model the category $\widehat{\cat}_R$ of (not necessarily genuine) line operators in such a theory was proposed in \cite[\S 8]{ballin2023}; a Kazhdan--Lusztig correspondence for these boundary vertex operator algebras is currently being developed by Creutzig--Niu \cite{creutzig2024a}. We note that the associated quantum groups can be derived directly from the underlying field theory without passing through an auxiliary vertex operator algebra \cite{creutzig2024b}, and should be related to the ones described here by uprolling \cite{creutzig2022a}. As we explain below, we expect that the derived category of the full ribbon subcategory $\mcat_R^{\Int}$ of $\mcat_R$ consisting of $\Uqmin$-modules whose $\bosSm$-weights lie in the integral lattice $\Gamma^{\vee} \subset \bosSm^{\vee}$ is a full subcategory of $\widehat{\cat}_R$ after quotienting by the free realization $\FR$ below. We emphasize that objects of $\widehat{\cat}_R$ are not subject to any semisimplicity assumption on the action of the subalgebra $\bosSm^{\vee} \subset \bos$.

Since the lattice $\Gamma \subset \bos$ is not full rank, the present setting is not that of Section \ref{sec:genConsid}. In particular, the groups $\Gamma^{\vee}$ and
\[
\{ \lambda \in \bos^{\vee} \mid \lambda(\gamma) \in \Z \; \; \forall \gamma \in \Gamma \}
\simeq
\Gamma^{\vee} \oplus \bosSm
\]
differ. Here, and below, we identify $\bosSm^{\vee \vee}$ with $\bosSm$. Since $\Gamma^{\vee} \subset \Gamma^{\vee} \oplus \bosSm$, we can grade $\mcat_R$ by the group $\Gr = \bos^{\vee} \slash \Gamma^{\vee} \oplus \bosSm \simeq \bosSm^{\vee} \slash \Gamma^{\vee}$, which is known to be the topological flavour symmetry of abelian gauged hypermultiplets.

\begin{Prop}
\label{prop:freeRealGaugedAbelHyp}
Let $\FR_0$ be the image of $\Gamma$ under the map $\met^{\flat}: \bos \rightarrow \bos^{\vee}$. If the integral lattice $\Gamma$ is even with respect to $\emet$, then $\FR= \FR_0 \oplus \Ztwo$ satisfies conditions \ref{ite:deg0Coch}-\ref{ite:ribbCoch} but fails condition \ref{ite:finiteCoch}, where in each condition the group $\Gamma^{\vee}$ of Section \ref{sec:genConsid} is replaced with $\Gamma^{\vee} \oplus \bosSm$.
\end{Prop}

\begin{proof}
We discuss each condition. Given $\gamma \in \Gamma$, let $\omega_{\gamma} =\met^{\flat}(\gamma)$.
\begin{enumerate}[label=(A\arabic*)]
\addtocounter{enumi}{1}
\item Since $\met$ is hyperbolic and $\Gamma \subset \bosSm$, we have $\omega_{\gamma} \in \bosSm^{\vee \vee} \simeq \bosSm \subset \bos^{\vee}$. Since $\bosSm \subset \Gamma^{\vee} \oplus \bosSm$, the condition holds.

\item For each $\lambda \in \Gamma^{\vee} \oplus \bosSm$, we have $\met^{\vee}(\omega_{\gamma},\lambda) = \lambda(\gamma) \in \Z$.

\item As in Proposition \ref{prop:freeRealCoch}, this follows from the even integrality of $\emet$.

\item This condition fails since the index in question is (uncountably) infinite. \qedhere
\end{enumerate}
\end{proof}

In particular, we do not obtain a relative modular structure on $\mcat_R$ or, for the same reasons, its integral subcategory $\mcat_R^{\Int}$. Theorem \ref{thm:simplesSplitAb} implies that $\FR$-orbits of isomorphism classes of simple objects of $\mcat_{R,\p 0}^{\Int}$ are in bijection with $\Gamma^{\vee} \times \bosSm \slash \Gamma$. This agrees with the classification of simple objects of the vertex algebraic model $\widehat{\cat}_R$ \cite[Corollary 8.3.1]{ballin2023}. The braiding and twist of $\mcat_R^{\Int}$ are compatible with those of $\widehat{\cat}_R$.

\begin{Rem}
The category $\mcat_R$ includes line operators sourcing background flat connections for the topological flavour symmetry group $\Gr$ but not the Higgs branch flavour symmetry group. While the inclusion of the latter is a crucial ingredient in reproducing physical expectations, it will not resolve the above lack of finiteness. See \cite[\S 2]{creutzig2021} for such an example (that violates Assumption \ref{assump:noZeroRoots}).
\end{Rem} 

\subsection{Free realizations via $\ker \chi$}
\label{sec:kerGrad}

We consider a slight generalization of Section \ref{sec:genConsid}. Fix subgroups $\Lambda_0$ and $\Lambda$ of $\bos^{\vee}$. Consider the following conditions, where $k \in \Lambda_0$ and $\lambda \in \Lambda$:
\begin{enumerate}[label=(B\arabic*)]
\item \label{ite:grading} {[$\Gr$-grading]} $\Lambda_R \subset \Lambda$.

\item \label{ite:deg0} {[$\FR$ in degree $\p 0$]} $\Lambda_0 \subset \Lambda$.

\item \label{ite:psi} {[Existence of $\psi$]} $\met^{\vee}(k,\lambda) \in \Z$.

\item \label{ite:ribb} {[Trivial ribbon]} $-\met^{\vee}(k,k) + \sum_{i=1}^n \chi_i(k) \in 2 \Z$.

\item \label{ite:fin} {[Finiteness]}  The index of $\Lambda_0$ in $\Lambda$ is finite.

\item \label{ite:invert} {[Invertibility of $\FR$]} $\chi_i(k) \in \Z$ for all $1 \leq i \leq n$.
\end{enumerate}

As in Section \ref{sec:genConsid}, the above conditions are necessary for tracking weights modulo $\Lambda$ to define a grading of $\mcat_R$ by the group $\Gr = \bos^{\vee} \slash \Lambda$ and for $\FR = \Lambda_0 \oplus \Ztwo$ to be the weights of a free realization on $\mcat_{R,\p 0}$ which is a part of a relative pre-modular structure on $\mcat_R$. In the restricted setting of Section \ref{sec:genConsid}, where $\Lambda = \Gamma^{\vee}$ and $\Lambda_0 = \ima(\met^{\flat}: \Gamma \rightarrow \Gamma^{\vee})$, the analogue of condition \ref{ite:invert} follows from conditions \ref{ite:gradingCoch} and \ref{ite:psiCoch}.

Assume that $\Lambda$ is a subgroup of $\ker (\chi: \bos^{\vee} \rightarrow \C^n)$ and let $\Lambda_0 = \Lambda$. Conditions \ref{ite:deg0}, \ref{ite:psi}, \ref{ite:fin} and \ref{ite:invert} then hold automatically. We assume that conditions \ref{ite:grading} and \ref{ite:ribb} hold as well. Since $\Lambda \subset \ker \chi$, there is an induced group homomorphism $\overline{\chi}: \Gr \rightarrow \GrH$.  By Lemma \ref{lem:sssPullback}, the preimage of $\SSSY$ under $\overline{\chi}$ is a small symmetric subset of $\Gr$ which is also a subgroup; call it $\SSS$.

\begin{Thm}
\label{thm:relModKernel}
The monoidal functor
\[
\sigma: \FR \rightarrow \mcat_R,
\qquad
(k, \p p) \mapsto \sigma_{(k,\p p)} := \C^{\bos}_{(k, \p p)}
\]
defines a free realization of $\FR$ on $\mcat_{R,\p 0}$ and gives $\mcat_R$ the structure of a modular $\Gr$-category relative to $(\FR,\SSS)$ with stabilization coefficients $\Delta_{\pm}=(\pm 1)^n$ and relative modularity parameter $\zeta = (-1)^n$.
\end{Thm}

\begin{proof}
The definition of $\SSS$ is such that any weight with image in $\Gr \setminus \SSS$ is typical. Generic semisimplicity of $\mcat_R$ can therefore be proved in the same way as Proposition \ref{prop:genericSSZWt}. Conditions \ref{ite:deg0} and \ref{ite:invert} ensure that each $\sigma_{(k, \p p)}$ has $\Gr$-degree $\p 0$ and is one dimensional, respectively. Condition \ref{ite:ribb} ensures that $\theta_{\sigma_{(k,\p p)}}= \id_{\sigma_{(k,\p p)}}$ for all $(k, \p p) \in \FR$. 
Let $\p \lambda \in \Gr \setminus \SSS$ with chosen lift $\lambda \in \bos^{\vee}$; any other lift is of the form $\lambda + k$ for a unique $k \in \Lambda$. A completely reduced dominating set of $\cat_{R,\p \lambda}$ is
\[
\{V_{(\lambda,\p p)}  \mid \lambda \mbox{ in class } \p \lambda, \; \p p \in \Ztwo \}
\]
and we may take $\Theta(\p \lambda) = \{V_{(\lambda, \p 0)}\}$. The required bicharacter is
\[
\psi: \Gr \times \FR \rightarrow \C^{\times},
\qquad
(\p \lambda, (k, \p p)) \mapsto q^{-4 \met^{\vee}(k,\lambda)}.
\]
This proves relative pre-modularity of $\mcat_R$. 

A direct computation using Lemma \ref{lem:Phis} shows that for $\p \lambda \in \Gr \setminus \SSS$ with lift $\lambda \in \bos^{\vee}$, we have
\[
\Delta_-
=
\qd(V_{(\lambda, \p 0)}) q^{2 \met^{\vee}(\lambda, \lambda) - 4 \sum_{i=1}^n \chi_i(\lambda)} \langle \Phi_{V_{(\lambda,\p 0)},V_{(\lambda,\p 0)}} \rangle
=
(-1)^n
\]
and, similarly, $\Delta_+ = 1$.

It remains to establish the existence of a relative modularity parameter. Consider Definition \ref{def:modG} with $h = \p \gamma$ and $g = \p \lambda$ with $V_i, V_j=V_{(\lambda,\p 0)} \in I_{\p \lambda}$. Applying the direct analogue of Lemma \ref{lem:transparentCpt}, we can write the endomorphism of $V_{(\lambda,\p 0)} \otimes V_{(\lambda,\p 0)}^{\vee}$ determined by the left hand side of diagram \eqref{eq:mod} as
\[
f_{\p \gamma,(\lambda,\p 0),(\lambda,\p 0)}
=
\sum_{i=1}^m h_{\p \gamma,(\lambda,\p 0),(\lambda,\p 0),i} \circ g_{\p \gamma,(\lambda,\p 0),(\lambda,\p 0),i}
\]
for some
\[
g_{\p \gamma,(\lambda,\p 0),(\lambda,\p 0),i} \in \Hom_{\mcat_R}(V_{(\lambda,\p 0)} \otimes V_{(\lambda,\p 0)}^{\vee},\C^{\bos}_{(k_i,\p p_i)}),
\]
\[
h_{\p \gamma,(\lambda,\p 0),(\lambda,\p 0),i} \in \Hom_{\mcat_R}(\C^{\bos}_{(k_i,\p p_i)},V_{(\lambda,\p 0)} \otimes V_{(\lambda,\p 0)}^{\vee})
\]
and one dimensional modules $\C^{\bos}_{(k_i,\p p_i)} \in \mcat_{R,\p 0}$. Clearly
\[
\Hom_{\mcat_R}(V_{(\lambda,\p 0)} \otimes V^{\vee}_{(\lambda,\p 0)}, \C^{\bos}_{(k_i,\p p_i)})
\simeq
\Hom_{\mcat_R}(V_{(\lambda,\p 0)}, V_{(\lambda,\p 0)} \otimes \C^{\bos}_{(k_i,\p p_i)})
\]
vanishes unless $(k_i,\p p_i) =0$. It follows that
\[
f_{\p \gamma,(\lambda,\p 0),(\lambda,\p 0)}
\in
\Hom_{\mcat_R}(V_{(\lambda,\p 0)} \otimes V^{\vee}_{(\lambda,\p 0)}, \C)
\simeq
\End_{\mcat_R}(V_{(\lambda,\p 0)}) \simeq \C,
\]
whence $f_{\p \gamma,(\lambda,\p 0),(\lambda,\p 0)}$ is proportional to $\tcoev_{V_{(\lambda,\p 0)}} \circ \ev_{V_{(\lambda,\p 0)}}$. By \cite[Proposition 1.2]{derenzi2022}, this proportionality constant is necessarily $\zeta=\Delta_+ \Delta_- =(-1)^n$.

TQFT finiteness is proved as in Theorem \ref{thm:relModCpt}.
\end{proof}

\subsubsection{Verlinde formula}
\label{sec:verlindeKer}

Write $\mcat_{R,\chi}$ for the relative modular structure of Theorem \ref{thm:relModKernel}. Considering the analogue of Section \ref{sec:verlindeCpt}, we find that $\TQFT_{\mcat_{R,\chi}}(\CS \times S^1_{\p \beta})$ is equal to
\[
(-1)^{(g+1+m)n + \p q} q^{-4\met^{\vee}(\beta, \mu) +2m \sum_{i=1}^n \chi_i(\beta) + 2\sum_{i=1}^n \chi_i(\mu)} \prod_{i=1}^n (q^{2\chi_i(\beta)}-q^{-2\chi_i(\beta)})^{2g-2+m}.
\]
The Verlinde formula takes the form $\TQFT_{\mcat_{R,\chi}}(\CS \times S^1_{\p \beta}) = \dim_{(q^{-4 \met^{\vee}(-,\beta)},1)} \TQFT_{\mcat_{R,\chi}}(\CS)$.

\begin{Ex}
\label{ex:VerlindeArb}
Setting $m=0$ gives
\[
\TQFT_{\mcat_{R,\chi}}(\CS \times S^1_{\p \beta})
=
(-1)^{(g+1)n + \p q}  \prod_{i=1}^n (q^{2\chi_i(\beta)}-q^{-2\chi_i(\beta)})^{-\chi(\CS)}.
\]
Applying the Verlinde formula, we conclude that $\dim_{\C} \TQFT_{\mcat_{R,\chi}}(\CS) = 2^{n(2g-2)}$ and
\[
\chi(\TQFT_{\mcat_{R,\chi}}(\CS))
=
\begin{cases}
1 & \mbox{if } g=1, \\
0 & \mbox{if } g\geq 2.
\end{cases}\qedhere
\]
\end{Ex}

\subsection{$\psloo$-Chern--Simons theory}
\label{sec:pslooCS}

Consider again the setting of Example \ref{ex:gloo}, so that $\GW \simeq \gloo$; see also Section \ref{sec:uooCS}. We have $\Lambda_R = \Z \cdot N^{\vee}$. The linear map $\chi: \C^2 \rightarrow \C$ is represented by the matrix $\left(\begin{matrix} 0 & 1 \end{matrix} \right)$ so that $\Lambda:= \ker \chi = \C \cdot N^{\vee}$ contains $\Lambda_R$. This verifies condition \ref{ite:grading}. It follows that the restriction of $\met^{\vee}$ to $\ker \chi$ is trivial. In particular, $\ker \chi$ is even integral with respect to $\met^{\vee}$. Condition \ref{ite:ribb} therefore also holds and we obtain a relative modular category $\mcat_{R,\chi}$. The grading $\Gr$ is by $E$-weights and the free realization $\FR$ consists of all one dimensional modules with vanishing $E$-weight. The relative modular category $\mcat_{R,\chi}$ recovers that constructed from the representation theory of the unrolled quantum group $U_q^H(\gloo)$ at arbitrary parameter $q$ \cite[Theorem 2.16]{geerYoung2022}. The TQFT $\TQFT_{\mcat_{R,\chi}}$ models Chern--Simons theory with gauge supergroup $\psloo$ coupled to background flat $\C^{\times}$-connections, as studied in \cite{mikhaylov2015}. Alternatively, $\TQFT_{\mcat_{R,\chi}}$ models the Rozansky--Witten theory of $T^{\vee} \C$ with Higgs branch flavour symmetry $\C^{\times}$ which acts on $T^{\vee} \C$ with weight $1$ on the base and weight $-1$ on the fibre.

In particular, the results of this section allow us to effectively circumvent Assumption \ref{assump:noZeroRoots} in the construction of abelian linear Gaiotto--Witten theories with compact gauge group, given in Section \ref{sec:cptBosonic}. Indeed, if Assumption \ref{assump:noZeroRoots} does not hold, then the theory splits as a product of the Gaiotto--Witten theory associated to the quotient representation $R \slash \Hom_{\bos}(\C,R)$ and $\dim_{\C} \Hom_{\bos}(\C,R)$-many copies of the Rozansky--Witten theory of $T^{\vee} \C$. The former theory is treated in Section \ref{sec:cptBosonic} while the latter---when considered with Higgs branch flavour symmetry---is treated in this section.

\appendix

\section{Reminders on lattices}
\label{sec:lattices}

Let $\met$ be a non-degenerate symmetric bilinear form on $\R^r$. Let $\Gamma \subset \R^r$ be a full rank integral lattice with basis $\{\gamma_1, \dots, \gamma_r\}$. The $r \times r$ Gram matrix of $(\Gamma, \met)$ is $B=(B_{ij})=( \met( \gamma_i,\gamma_j ))$ with inverse $B^{-1}=(B^{ij})$. Integrality of the lattice is the statement that $B$ has integer entries. If the diagonals of $B$ are even integers, then $\Gamma$ is called even. The dual lattice $\Gamma^{\vee}$ has basis $\{\gamma^{\vee}_1, \dots, \gamma^{\vee}_r\}$. There is an isomorphism
\begin{equation}
\label{eq:dualLattice}
\Gamma^{\vee} \rightarrow \{v \in \Gamma \otimes_{\Z} \Q \mid \met( v, \gamma ) \in \Z \;\, \forall \gamma \in \Gamma\},
\qquad
\gamma_i^{\vee} \mapsto \sum_{j=1}^r B^{ij} \gamma_j.
\end{equation}
Let $B = X\tilde{B}Y$ be the Smith normal form of $B$, so that $\tilde{B} = \mbox{diag}(d_1, \dots, d_r)$ is diagonal with entries satisfying $d_1 \vert d_2 \vert \cdots \vert d_r$ and the left and right Smith multipliers $X$ and $Y$, respectively, are in $GL(r,\Z)$. The isomorphism \eqref{eq:dualLattice} gives
\[
\sum_{k=1}^r X^{ij} \gamma_k \mapsto d_i \sum_{j=1}^r Y_{ij} \gamma_j^{\vee},
\]
where $X^{-1} = (X^{ij})$. It follows that $\delta_i := \sum_{j=1}^r Y_{ij} \gamma_j^{\vee}$ generates the $i$\textsuperscript{th} factor of the discriminant group
\[
D=\coker(\met^{\flat}: \Gamma \rightarrow \Gamma^{\vee}) \simeq \bigoplus_{i=1}^r \Z \slash d_i \Z.
\]

\newcommand{\etalchar}[1]{$^{#1}$}



\begin{thebibliography}{dMFOME09}

\bibitem[AGP21]{anghel2021}
C.~Anghel, N.~Geer, and B.~{Patureau-Mirand}.
\newblock Relative (pre)-modular categories from special linear {L}ie
  superalgebras.
\newblock {\em J. Algebra}, 586:479--525, 2021.

\bibitem[AKSZ97]{alexandrov1997}
M.~Alexandrov, M~Kontsevich, A.~Schwarz, and O.~Zaboronsky.
\newblock The geometry of the master equation and topological quantum field
  theory.
\newblock {\em Int. J. Mod. Phys. A}, 12(7):1405--1429, 1997.

\bibitem[ALMRP92]{alford1992}
M.~Alford, K.-M. Lee, J.~March-Russell, and J.~Preskill.
\newblock Quantum field theory of non{A}belian strings and vortices.
\newblock {\em Nucl. Phys. B}, 384:251--317, 1992.

\bibitem[AP95]{andersen1995}
H.~Andersen and J.~Paradowski.
\newblock Fusion categories arising from semisimple {L}ie algebras.
\newblock {\em Comm. Math. Phys.}, 169(3):563--588, 1995.

\bibitem[BCDN23]{ballin2023}
A.~Ballin, T.~Creutzig, T.~Dimofte, and W.~Niu.
\newblock {$3$}d mirror symmetry of braided tensor categories.
\newblock ar{X}iv:2304.11001, 2023.

\bibitem[BCGPM14]{blanchet2014}
C.~Blanchet, F.~Costantino, N.~Geer, and B.~Patureau-Mirand.
\newblock Non semi-simple {$\mathfrak{sl}(2)$} quantum invariants, spin case.
\newblock {\em Acta Math. Vietnam.}, 39(4):481--495, 2014.

\bibitem[BCGPM16]{blanchet2016}
C.~Blanchet, F.~Costantino, N.~Geer, and B.~Patureau-Mirand.
\newblock Non-semi-simple {TQFT}s, {R}eidemeister torsion and {K}ashaev's
  invariants.
\newblock {\em Adv. Math.}, 301:1--78, 2016.

\bibitem[BCR23]{brunner2023}
I.~Brunner, N.~Carqueville, and D.~Roggenkamp.
\newblock Truncated affine {R}ozansky-{W}itten models as extended {TQFT}s.
\newblock {\em Comm. Math. Phys.}, 400(1):371--415, 2023.

\bibitem[BF19]{braverman2019}
A.~Braverman and M.~Finkelberg.
\newblock Coulomb branches of 3-dimensional gauge theories and related
  structures.
\newblock In {\em Geometric representation theory and gauge theory}, volume
  2248 of {\em Lecture Notes in Math.}, pages 1--52. Springer, 2019.

\bibitem[BFN10]{benzvi2010}
D.~{Ben-Zvi}, J.~Francis, and D.~Nadler.
\newblock Integral transforms and {D}rinfeld centers in derived algebraic
  geometry.
\newblock {\em J. Amer. Math. Soc.}, 23(4):909--966, 2010.

\bibitem[BFN18]{braverman2018}
A.~Braverman, M.~Finkelberg, and H.~Nakajima.
\newblock Towards a mathematical definition of {C}oulomb branches of
  3-dimensional {$\mathcal{N}=4$} gauge theories, {II}.
\newblock {\em Adv. Theor. Math. Phys.}, 22(5):1071--1147, 2018.

\bibitem[BM05]{belov2005}
D.~Belov and G.~Moore.
\newblock Classification of abelian spin {C}hern--{S}imons theories.
\newblock ar{X}iv:hep-th/0505235, 2005.

\bibitem[BN22]{ballin2022}
A.~Ballin and W.~Niu.
\newblock {$3$}d mirror symmetry and the {$\beta \gamma$} {VOA}.
\newblock {\em Commun. Contemp. Math.}, 2022.

\bibitem[CCG19]{costello2019b}
K.~Costello, T.~Creutzig, and D.~Gaiotto.
\newblock Higgs and {C}oulomb branches from vertex operator algebras.
\newblock {\em J. High Energy Phys.}, 03:066, 48, 2019.

\bibitem[CDG23]{costello2023}
K.~Costello, T.~Dimofte, and D.~Gaiotto.
\newblock Boundary chiral algebras and holomorphic twists.
\newblock {\em Comm. Math. Phys.}, 399(2):1203--1290, 2023.

\bibitem[CDGG21]{creutzig2021}
T.~Creutzig, T.~Dimofte, N.~Garner, and N.~Geer.
\newblock A {QFT} for non-semisimple {TQFT}.
\newblock ar{X}iv:2112.01559, 2021.

\bibitem[CDN24]{creutzig2024b}
T.~Creutzig, T.~Dimofte, and W.~Niu.
\newblock Work in progress.
\newblock 2024.

\bibitem[CG19]{costello2019}
K.~Costello and D.~Gaiotto.
\newblock Vertex {O}perator {A}lgebras and 3d {$\mathcal{N}=4$} gauge theories.
\newblock {\em J. High Energy Phys.}, 05:018, 37, 2019.

\bibitem[CGP15]{costantino2015}
F.~Costantino, N.~Geer, and B.~{Patureau-Mirand}.
\newblock Some remarks on the unrolled quantum group of {$\mathfrak{sl}(2)$}.
\newblock {\em J. Pure Appl. Algebra}, 219(8):3238--3262, 2015.

\bibitem[CGPM14]{costantino2014}
F.~Costantino, N.~Geer, and B.~Patureau-Mirand.
\newblock Quantum invariants of 3-manifolds via link surgery presentations and
  non-semi-simple categories.
\newblock {\em J. Topol.}, 7(4):1005--1053, 2014.

\bibitem[CK19]{closset2019}
C.~Closset and H.~Kim.
\newblock Three-dimensional {$\mathcal{N} = 2$} supersymmetric gauge theories
  and partition functions on {S}eifert manifolds: {A} review.
\newblock {\em Int. J. Mod. Phys. A}, 34(23):1930011, 2019.

\bibitem[CLL17]{chan2017}
K.~Chan, N.-C. Leung, and Q.~Li.
\newblock B{V} quantization of the {R}ozansky-{W}itten model.
\newblock {\em Comm. Math. Phys.}, 355(1):97--144, 2017.

\bibitem[CN24]{creutzig2024a}
T.~Creutzig and W.~Niu.
\newblock Work in progress.
\newblock 2024.

\bibitem[CR22]{creutzig2022a}
T.~Creutzig and M.~Rupert.
\newblock {U}prolling unrolled quantum groups.
\newblock {\em Commun. Contemp. Math.}, 24(04):2150023, 2022.

\bibitem[DEF{\etalchar{+}}99]{deligne1999}
P.~Deligne, P.~Etingof, D.~Freed, L.~Jeffrey, D.~Kazhdan, J.~Morgan,
  D.~Morrison, and E.~Witten, editors.
\newblock {\em Quantum fields and strings: a course for mathematicians. {V}ol.
  1, 2}.
\newblock American Mathematical Society, Providence, RI; Institute for Advanced
  Study (IAS), Princeton, NJ, 1999.

\bibitem[dMFOME09]{medeiros2009}
P.~de~Medeiros, J.~Figueroa-O'Farrill, and E.~M\'{e}ndez-Escobar.
\newblock Superpotentials for superconformal {C}hern-{S}imons theories from
  representation theory.
\newblock {\em J. Phys. A}, 42(48):485204, 56, 2009.

\bibitem[DR22]{derenzi2022}
M.~De~Renzi.
\newblock Non-semisimple extended topological quantum field theories.
\newblock {\em Mem. Amer. Math. Soc.}, 277(1364):v+161, 2022.

\bibitem[DRGPM20]{derenzi2020}
M.~De~Renzi, N.~Geer, and B.~Patureau-Mirand.
\newblock Nonsemisimple quantum invariants and {TQFT}s from small and unrolled
  quantum groups.
\newblock {\em Algebr. Geom. Topol.}, 20(7):3377--3422, 2020.

\bibitem[DW90]{dijkgraaf1990}
R.~Dijkgraaf and E.~Witten.
\newblock Topological gauge theories and group cohomology.
\newblock {\em Comm. Math. Phys.}, 129(2):393--429, 1990.

\bibitem[EGNO15]{etingof2015}
P.~Etingof, S.~Gelaki, D.~Nikshych, and V.~Ostrik.
\newblock {\em Tensor categories}, volume 205 of {\em Mathematical Surveys and
  Monographs}.
\newblock American Mathematical Society, Providence, RI, 2015.

\bibitem[EMSS89]{elitzur1989}
S.~Elitzur, G.~Moore, A.~Schwimmer, and N.~Seiberg.
\newblock Remarks on the canonical quantization of the
  {C}hern-{S}imons-{W}itten theory.
\newblock {\em Nuclear Phys. B}, 326(1):108--134, 1989.

\bibitem[FHLT10]{freed2010}
D.~Freed, M.~Hopkins, J.~Lurie, and C.~Teleman.
\newblock Topological quantum field theories from compact {L}ie groups.
\newblock In {\em A celebration of the mathematical legacy of {R}aoul {B}ott},
  volume~50 of {\em CRM Proc. Lecture Notes}, pages 367--403. Amer. Math. Soc.,
  Providence, RI, 2010.

\bibitem[FQ93]{freed1993}
D.~Freed and F.~Quinn.
\newblock Chern-{S}imons theory with finite gauge group.
\newblock {\em Comm. Math. Phys.}, 156(3):435--472, 1993.

\bibitem[Fre94]{freed1994}
D.~Freed.
\newblock Higher algebraic structures and quantization.
\newblock {\em Comm. Math. Phys.}, 159(2):343--398, 1994.

\bibitem[Gai19]{gaiotto2019}
D.~Gaiotto.
\newblock Twisted compactifications of {$3{\rm d}\ \mathcal{N}=4$} theories and
  conformal blocks.
\newblock {\em J. High Energy Phys.}, 2:061, front matter+41, 2019.

\bibitem[Gar23]{garner2023}
N.~Garner.
\newblock Vertex operator algebras and topologically twisted
  {C}hern--{S}imons-matter theories.
\newblock {\em J. High Energy Phys.}, 25, 2023.

\bibitem[GHN{\etalchar{+}}21]{gukov2021}
S.~Gukov, P.-S. Hsin, H.~Nakajima, S.~Park, D.~Pei, and N.~Sopenko.
\newblock Rozansky-{W}itten geometry of {C}oulomb branches and logarithmic knot
  invariants.
\newblock {\em J. Geom. Phys.}, 168:Paper No. 104311, 22, 2021.

\bibitem[GKPM11]{geer2011}
N.~Geer, J.~Kujawa, and B.~Patureau-Mirand.
\newblock Generalized trace and modified dimension functions on ribbon
  categories.
\newblock {\em Selecta Math. (N.S.)}, 17(2):453--504, 2011.

\bibitem[GKPM22]{geer2022}
N.~Geer, J.~Kujawa, and B.~Patureau-Mirand.
\newblock {M}-traces in (non-unimodular) pivotal categories.
\newblock {\em Algebr. Represent. Theor.}, 125:759--776, 2022.

\bibitem[GKSW15]{gaiotto2014}
D.~Gaiotto, A.~Kapustin, N.~Seiberg, and B.~Willett.
\newblock {G}eneralized {G}lobal {S}ymmetries.
\newblock {\em J. High Energy Phys.}, 02:172, 2015.

\bibitem[GN23]{garner2023b}
N.~Garner and W.~Niu.
\newblock Line operators in {$U(1 \vert 1)$} {C}hern--{S}imons theory.
\newblock ar{X}iv:2304.05414, 2023.

\bibitem[GPM18]{geer2018}
N.~Geer and B.~Patureau-Mirand.
\newblock The trace on projective representations of quantum groups.
\newblock {\em Lett. Math. Phys.}, 108(1):117--140, 2018.

\bibitem[GPMT09]{geer2009}
N.~Geer, B.~Patureau-Mirand, and V.~Turaev.
\newblock Modified quantum dimensions and re-normalized link invariants.
\newblock {\em Compos. Math.}, 145(1):196--212, 2009.

\bibitem[GW10]{gaiotto2010b}
D.~Gaiotto and E.~Witten.
\newblock Janus configurations, {C}hern-{S}imons couplings, and the
  {$\theta$}-angle in {${\mathcal{N}}=4$} super {Y}ang-{M}ills theory.
\newblock {\em J. High Energy Phys.}, 06:097, 58, 2010.

\bibitem[GY22]{geerYoung2022}
N.~Geer and M.~Young.
\newblock Three dimensional topological quantum field theory from
  {$U_q(\mathfrak{gl}(1 \vert 1))$} and {$U(1 \vert 1)$} {C}hern--{S}imons
  theory.
\newblock ar{X}iv:2210.04286, 2022.

\bibitem[Ha22]{ha2022}
N.~P. Ha.
\newblock Anomaly-free {TQFT}s from the super {L}ie algebra
  {$\mathfrak{sl}(2\vert 1)$}.
\newblock {\em J. Knot Theory Ramifications}, 31(5):Paper No. 2250029, 14,
  2022.

\bibitem[Hen17]{henriques2017}
A.~Henriques.
\newblock What {C}hern-{S}imons theory assigns to a point.
\newblock {\em Proc. Natl. Acad. Sci. USA}, 114(51):13418--13423, 2017.

\bibitem[IS96]{intriligator1996}
K.~Intriligator and N.~Seiberg.
\newblock Mirror symmetry in three-dimensional gauge theories.
\newblock {\em Phys. Lett. B}, 387(3):513--519, 1996.

\bibitem[JS93]{joyal1993}
A.~Joyal and R.~Street.
\newblock Braided tensor categories.
\newblock {\em Adv. Math.}, 102(1):20--78, 1993.

\bibitem[Kap99]{kapranov1999}
M.~Kapranov.
\newblock Rozansky-{W}itten invariants via {A}tiyah classes.
\newblock {\em Compositio Math.}, 115(1):71--113, 1999.

\bibitem[KLL09]{koh2009}
E.~Koh, S.~Lee, and S.~Lee.
\newblock Topological {C}hern--{S}imons sigma model.
\newblock {\em J. High Energy Phys.}, 09:122, 2009.

\bibitem[Kon99]{kontsevich1999}
M.~Kontsevich.
\newblock Rozansky-{W}itten invariants via formal geometry.
\newblock {\em Compositio Math.}, 115(1):115--127, 1999.

\bibitem[KQZ13]{kallen2013}
J.~K\"all\'en, J.~Qiu, and M.~Zabzine.
\newblock Equivariant {R}ozansky-{W}itten classes and {TFT}s.
\newblock {\em J. Geom. Phys.}, 64:222--242, 2013.

\bibitem[KR10]{kapustin2010}
A.~Kapustin and L.~Rozansky.
\newblock Three-dimensional topological field theory and symplectic algebraic
  geometry {II}.
\newblock {\em Commun. Number Theory Phys.}, 4(3):463--549, 2010.

\bibitem[Kra15]{krause2015}
H.~Krause.
\newblock Krull-{S}chmidt categories and projective covers.
\newblock {\em Expo. Math.}, 33(4):535--549, 2015.

\bibitem[KRS09]{kapustin2009}
A.~Kapustin, L.~Rozansky, and N.~Saulina.
\newblock Three-dimensional topological field theory and symplectic algebraic
  geometry. {I}.
\newblock {\em Nuclear Phys. B}, 816(3):295--355, 2009.

\bibitem[KS09]{kapustin2009b}
A.~Kapustin and N.~Saulina.
\newblock Chern-{S}imons-{R}ozansky-{W}itten topological field theory.
\newblock {\em Nuclear Phys. B}, 823(3):403--427, 2009.

\bibitem[KS11]{kapustin2011}
A.~Kapustin and N.~Saulina.
\newblock Topological boundary conditions in abelian {C}hern-{S}imons theory.
\newblock {\em Nuclear Phys. B}, 845(3):393--435, 2011.

\bibitem[KS14]{kapustin2014}
A.~Kapustin and N.~Seiberg.
\newblock {C}oupling a {QFT} to a {TQFT} and {D}uality.
\newblock {\em J. High Energy Phys.}, 04:001, 2014.

\bibitem[KT91]{khoroshkin1991}
S.~Khoroshkin and V.~Tolstoy.
\newblock Universal {$R$}-matrix for quantized (super)algebras.
\newblock {\em Comm. Math. Phys.}, 141(3):599--617, 1991.

\bibitem[Kul89]{kulish1989}
P.~Kulish.
\newblock Quantum {L}ie superalgebras and supergroups.
\newblock In {\em Problems of modern quantum field theory ({A}lushta, 1989)},
  Res. Rep. Phys., pages 14--21. Springer, Berlin, 1989.

\bibitem[Lur09]{lurie2009}
J.~Lurie.
\newblock On the classification of topological field theories.
\newblock In {\em Current developments in mathematics, 2008}, pages 129--280.
  Int. Press, Somerville, MA, 2009.

\bibitem[Man98]{manoliu1998}
M.~Manoliu.
\newblock Abelian {C}hern-{S}imons theory. {I}. {A} topological quantum field
  theory.
\newblock {\em J. Math. Phys.}, 39(1):170--206, 1998.

\bibitem[Mik15]{mikhaylov2015}
V.~Mikhaylov.
\newblock Analytic torsion, {$3d$} mirror symmetry and supergroup
  {C}hern--{S}imons theories.
\newblock ar{X}iv:1505.03130, 2015.

\bibitem[Nak16]{nakajima2016}
H.~Nakajima.
\newblock Towards a mathematical definition of {C}oulomb branches of
  3-dimensional {$\mathcal{N}=4$} gauge theories, {I}.
\newblock {\em Adv. Theor. Math. Phys.}, 20(3):595--669, 2016.

\bibitem[OR23]{oblomkov2023}
A.~Oblomkov and L.~Rozansky.
\newblock 3{D} {TQFT} and {HOMFLYPT} homology.
\newblock {\em Lett. Math. Phys.}, 113(3):Paper No. 71, 62, 2023.

\bibitem[PBM86]{prudnikov1986}
A.~Prudnikov, Y.~Brychkov, and O.~Marichev.
\newblock {\em Integrals and series. {V}ol. 1}.
\newblock Gordon \& Breach Science Publishers, New York, 1986.
\newblock Translated from the Russian by N. Queen.

\bibitem[QZ09]{qiu2009}
J.~Qiu and M.~Zabzine.
\newblock On the {AKSZ} formulation of the {R}ozansky-{W}itten theory and
  beyond.
\newblock {\em J. High Energy Phys.}, 09:024, 12, 2009.

\bibitem[QZ10]{qiu2010}
J.~Qiu and M.~Zabzine.
\newblock Odd {C}hern-{S}imons theory, {L}ie algebra cohomology and
  characteristic classes.
\newblock {\em Comm. Math. Phys.}, 300(3):789--833, 2010.

\bibitem[RS92]{rozansky1992}
L.~Rozansky and H.~Saleur.
\newblock Quantum field theory for the multi-variable {A}lexander-{C}onway
  polynomial.
\newblock {\em Nuclear Phys. B}, 376(3):461--509, 1992.

\bibitem[RS93]{rozansky1993}
L.~Rozansky and H.~Saleur.
\newblock {$S$}- and {$T$}-matrices for the super {${\rm U}(1,1)\ {\rm WZW}$}
  model. {A}pplication to surgery and {$3$}-manifolds invariants based on the
  {A}lexander-{C}onway polynomial.
\newblock {\em Nuclear Phys. B}, 389(2):365--423, 1993.

\bibitem[RS94]{rozansky1994}
L.~Rozansky and H.~Saleur.
\newblock Reidemeister torsion, the {A}lexander polynomial and {${\rm U}(1,1)$}
  {C}hern-{S}imons theory.
\newblock {\em J. Geom. Phys.}, 13(2):105--123, 1994.

\bibitem[RT91]{reshetikhin1991}
N.~Reshetikhin and V.~Turaev.
\newblock Invariants of {$3$}-manifolds via link polynomials and quantum
  groups.
\newblock {\em Invent. Math.}, 103(3):547--597, 1991.

\bibitem[RW97]{rozansky1997}
L.~Rozansky and E.~Witten.
\newblock Hyper-{K}\"ahler geometry and invariants of three-manifolds.
\newblock {\em Selecta Math. (N.S.)}, 3(3):401--458, 1997.

\bibitem[RW10]{roberts2010}
J.~Roberts and S.~Willerton.
\newblock On the {R}ozansky-{W}itten weight systems.
\newblock {\em Algebr. Geom. Topol.}, 10(3):1455--1519, 2010.

\bibitem[Sti08]{stirling2008}
S.~Stirling.
\newblock Abelian {C}hern--{S}imons theory with toral gauge group, modular
  tensor categories, and group categories.
\newblock ar{X}iv:0807.2857, 2008.

\bibitem[Tel14]{teleman2014}
C.~Teleman.
\newblock Gauge theory and mirror symmetry.
\newblock In {\em Proceedings of the {I}nternational {C}ongress of
  {M}athematicians---{S}eoul 2014. {V}ol. {II}}, pages 1309--1332. Kyung Moon
  Sa, Seoul, 2014.

\bibitem[Tel20]{teleman2020}
C.~Teleman.
\newblock Matrix factorization of {M}orse-{B}ott functions.
\newblock {\em Duke Math. J.}, 169(3):533--549, 2020.

\bibitem[Tur94]{turaev1994}
V.~Turaev.
\newblock {\em Quantum invariants of knots and 3-manifolds}, volume~18 of {\em
  De Gruyter Studies in Mathematics}.
\newblock Walter de Gruyter \& Co., Berlin, 1994.

\bibitem[Wit88a]{witten1988b}
E.~Witten.
\newblock Topological quantum field theory.
\newblock {\em Comm. Math. Phys.}, 117(3):353--386, 1988.

\bibitem[Wit88b]{witten1988a}
E.~Witten.
\newblock Topological sigma models.
\newblock {\em Comm. Math. Phys.}, 118(3):411--449, 1988.

\bibitem[Wit89]{witten1989}
E.~Witten.
\newblock Quantum field theory and the {J}ones polynomial.
\newblock {\em Comm. Math. Phys.}, 121(3):351--399, 1989.

\bibitem[Wit92]{witten1992}
E.~Witten.
\newblock Mirror manifolds and topological field theory.
\newblock In {\em Essays on mirror manifolds}, pages 120--158. Int. Press, Hong
  Kong, 1992.

\end{thebibliography}
\end{document}